\documentclass[11pt]{article}
\usepackage{amsmath,amssymb,amsthm,eucal}
\usepackage{color}
\usepackage[all]{xy}
\usepackage{stackengine}
\stackMath
\usepackage{slashed}

\title{\vspace*{-1pc}
Constructing KMS states from infinite-dimensional spectral triples}

\author{Magnus Goffeng${}^*$, Adam Rennie\ddag, Alexandr Usachev\thanks{email: 
\texttt{goffeng@chalmers.se}, \texttt{renniea@uow.edu.au},
\texttt{usachev@chalmers.se}
}
\\[3pt]
${}^*$ Department of Mathematical Sciences,\\ 
Chalmers University of Technology and
University of Gothenburg,\\ 
Gothenburg, Sweden
\\[3pt]
\ddag School of Mathematics and Applied Statistics,\\ 
University of Wollongong, Northfields Ave\\
Wollongong, Australia\\
}

\advance\topmargin by -\headheight
\advance\topmargin by -\headsep
\evensidemargin=\oddsidemargin
\makeatletter
\def\section{\@startsection{section}{1}{\z@}{-3.5ex plus -1ex minus
  -.2ex}{2.3ex plus .2ex}{\large\bf}}
\def\subsection{\@startsection{subsection}{2}{\z@}{-3.25ex plus -1ex
  minus -.2ex}{1.5ex plus .2ex}{\normalsize\bf}}
\makeatother

\numberwithin{equation}{section} 

\theoremstyle{plain} 
\newtheorem{thm}{Theorem}[section]
\newtheorem{thm*}{Theorem}
\newtheorem{lemma}[thm]{Lemma}
\newtheorem{prop}[thm]{Proposition}
\newtheorem{corl}[thm]{Corollary}

\theoremstyle{definition} 
\newtheorem{defn}[thm]{Definition}
\newtheorem{example}[thm]{Example}
\theoremstyle{remark} 
\newtheorem{rmk}[thm]{Remark}

\DeclareMathOperator{\Cl}{{\C\ell}} 
\DeclareMathOperator{\Dom}{Dom}   
\DeclareMathOperator{\End}{End}   
\DeclareMathOperator{\Tr}{Tr}     

\newcommand{\A}{\mathcal{A}}  
\newcommand{\B}{\mathcal{B}}  
\newcommand{\C}{\mathbb{C}}   
\newcommand{\D}{\mathcal{D}}  
\renewcommand{\H}{\mathcal{H}}  
\newcommand{\N}{\mathbb{N}}   
\newcommand{\cN}{\mathcal{N}} 
\newcommand{\Ko}{\mathbb{K}} 
\newcommand{\ox}{\otimes}     
\renewcommand{\O}{\mathcal{O}}  
\newcommand{\R}{\mathbb{R}}   
\newcommand{\T}{\mathcal{T}} 
\newcommand{\Z}{\mathbb{Z}}   

\newcommand{\Tau}{\mathcal{T}}      

\newcommand{\e}{\mathrm{e}}  

\newcommand{\stroke}{\mathbin|}     

\def\pairL_#1(#2|#3){{}_{#1}(#2\stroke#3)} 
\def\pairR(#1|#2)_#3{(#1\stroke#2)_{#3}} 
\def\scal<#1|#2>{\langle#1\stroke#2\rangle} 

\newbox\ncintdbox \newbox\ncinttbox 
	\setbox0=\hbox{$-$}
	\setbox2=\hbox{$\displaystyle\int$}
	\setbox\ncintdbox=\hbox{\rlap{\hbox
		to \wd2{\hskip-.125em \box2\relax\hfil}}\box0\kern.1em}
	\setbox0=\hbox{$\vcenter{\hrule width 4pt}$}
	\setbox2=\hbox{$\textstyle\int$}
	\setbox\ncinttbox=\hbox{\rlap{\hbox
		to \wd2{\hskip-.175em \box2\relax\hfil}}\box0\kern.1em}

\newcommand{\phimod}{\Xi_A} 
\newcommand{\Fock}{\mathcal{F}_E} 
\newcommand{\algFock}{\mathcal{F}_E^{\textnormal{alg}}} 
\newcommand{\topop}{\mathfrak{q}} 

\begin{document}

\maketitle

\vspace{-2pc}

\begin{abstract}
We construct KMS-states from $\mathrm{Li}_1$-summable semifinite spectral triples and show that in several important examples the construction coincides with well-known direct constructions of KMS-states for naturally defined flows. Under further summability assumptions 
the constructed KMS-state can be computed in terms of Dixmier traces.  For closed manifolds, we recover the ordinary Lebesgue integral. For Cuntz-Pimsner algebras with their gauge flow, the construction produces KMS-states from traces on the coefficient algebra and recovers the Laca-Neshveyev correspondence. For a discrete group acting on its Stone-\v{C}ech boundary, we recover the Patterson-Sullivan measures on the Stone-\v{C}ech boundary for a flow defined from the Radon-Nikodym cocycle.
\end{abstract}

\tableofcontents

\parskip=6pt
\parindent=0pt

\addtocontents{toc}{\vspace{-1pc}}

\section{Introduction}
\label{sec:intro}

The construction of the JLO cocycle \cite{GS,JLO1,JLO2} 
from $\theta$-summable spectral triples 
\cite{Con-trace} has from the start been closely linked with the idea of KMS states. 
A $\theta$-summable spectral triple $(\A,\H,\D)$ on a $C^*$-algebra $A$ gives rise to 
a state $\phi(a):= \mathrm{Tr}(a\e^{-\D^2})$ on $A$ and under suitable conditions 
this is a KMS-state on the saturation of $A$ by the 
$\R$-action defined from the wave operators $\e^{it\D^2}$.
By \cite{JLO2} the JLO-cocycle can be defined starting from this KMS-state. 
On the other hand, \cite{Con-trace} shows that a finitely summable spectral triple $(\A,\H,\D)$ 
on a $C^*$-algebra $A$ defines a tracial state on $A$. 
Similar constructions were studied in \cite{Voics}. 

The idea since then has been to understand the measure theory associated to $\theta$-summable 
spectral triples in terms of `twisted traces', and more specifically KMS states. 
Indeed this idea was present early in the development, \cite{JLO2}.
Two viewpoints make it interesting to study states associated with 
spectral triples having specified summability degrees: 
the associated states obstructs summability degrees, and 
the states provide a notion of measure theory.

In this paper we present a construction of 
KMS states from $\mathrm{Li}_1$-summable spectral triples.
By definition, a spectral triple $(\A,\H,\D)$ is $\mathrm{Li}_1$-summable if and only if $\e^{-t|\D|}$ is trace class for $t$ large enough
-- a slight strengthening of being $\theta$-summable.

{\em It is an important observation that large classes of examples of $\theta$-summable 
spectral triples are also $\mathrm{Li}_1$-summable.}

For the spectral triple defined from a Dirac operator on a 
closed manifold, our construction 
recovers the Lebesgue integral. 
For Cuntz-Pimsner algebras we also relate our construction to previous work 
of Laca and Neshveyev \cite{LN}, and
the authors \cite{GMR,RRS}. We also examine spectral triples arising from 
certain Hilbert space valued cocycles on discrete groups.

In the examples we consider, the  KMS-states
are associated to  flows that are 
well-suited to the geometries. This is usually not the case for 
the KMS-state $\phi(a)= \mathrm{Tr}(a\e^{-\D^2})$ associated with a 
$\theta$-summable spectral triple. 
It is our hope that our construction provides a more natural approach to the 
KMS-states appearing in the JLO-cocycle and that in the future it will have a 
bearing on the index theory of $\mathrm{Li}_1$-summable spectral triples.

\subsection{Main results}

We now state our main results. All our results make sense for 
general semifinite spectral triples, and so we fix a semifinite trace $\Tau$
for this discussion.

First, we state the main technical construction of KMS-states from $\mathrm{Li}_1$-summable spectral triples. After that, we state the implications of this construction to more specific examples. We use the notation $P_\D$ for the non-negative spectral projection of $\D$, i.e. $P_\D:=\chi_{[0,\infty)}(\D)$. If for some $\beta_\D\geq 0$, {\bf $\Tau(P_\D\e^{-t\D})$ is finite for $t>\beta_\D$ and diverges as $t\searrow \beta_\D$}, we say that $\D$ has positive $\Tau$-essential spectrum. We define the $C^*$-algebra $A_\D$ as the saturation of $A$ under the action of the wave group $\e^{it\D}$, that is
$$A_\D:=C^*\left(\cup_{t\in \R} \sigma_t(A)\right),\quad\mbox{where}\quad \sigma_t(a):=\e^{it\D}a\e^{-it\D}.$$
At this stage, we formulate our results in terms of $A_\D$. In Subsection \ref{subsec:toplitz} we  refine the construction to a smaller $C^*$-algebra. In examples, the construction often applies to $A$ directly. Recall from \cite[Definition 5.3.1]{BRII} that a state $\phi$ on an $\R$-$C^*$-algebra $\sigma:\R\curvearrowright A$ is said to be KMS at inverse temperature $\beta$ if $\phi(ab)=\phi(\sigma_{-i\beta}(b)a)$ for $a,b$ from an $\R$-invariant norm dense $*$-subalgebra of $A$. If $\phi$ is a state on an $\R$-von Neumann algebra $\sigma:\R\curvearrowright A$ we say that it is KMS if the same condition holds on an $\R$-invariant $\sigma$-weakly dense $*$-subalgebra of $A$. 

The following theorem is the main result of the paper.

\begin{thm*}
\label{mainthmconstr}
Let $(\A,\H,\D,\cN,\Tau)$ be a unital $\mathrm{Li}_1$-summable semifinite spectral triple such that $\D$ has positive $\Tau$-essential spectrum (see Definition \ref{ass:minus-one} on page \pageref{ass:minus-one}) and is $\beta$-analytic (see Definition \ref{ass:zero} on page \pageref{ass:zero}). Define 
$$
\beta_\D:=\inf\{t>0: \Tau(P_\D\e^{-t\D})<\infty\}.
$$
For any extended limit $\omega\in L^\infty(\beta_\D,\infty)^*$ as $t\to\beta_\D$ (see Definition \ref{extendedlimitdef} on page \pageref{extendedlimitdef}), we define the state $\phi_\omega$ on $A_\D$ as 
$$\phi_\omega(a):=\lim_{t\to \omega} \frac{\Tau(P_\D a\e^{-t\D})}{\Tau(P_\D \e^{-t\D})}.$$ 
Then $\phi_\omega$ is a KMS-state at inverse temperature $\beta_\D$ for the $\R$-action defined from $\sigma_t$. In particular, if $\beta_\D=0$ then $\phi_\omega$ is a tracial state on $A$.

If $\beta_\D=0$, and there is a decreasing function $\psi:[0,\infty)\to (0,\infty)$ with regular variation of index $-1$, satisfying the conditions~\eqref{exp2} and~\eqref{invas}, and for some $d>0$ we have that $\mu_\Tau(t,P_\D \D)\sim \psi(t)^{-1/d}$ as $t\to \infty$, then for any exponentiation invariant extended limit $\omega$ as $t\to\infty$,
$$
\phi_{\tilde{\omega}}(a)=\Tau_{\omega,\psi}(P_\D a(1+\D^2)^{-d/2}),
$$
where $\tilde{\omega}$ is an extended limit as $t\to 0$ defined in Theorem \ref{Frohlich} (see page \pageref{Frohlich}), and $\Tau_{\omega,\psi}$ is the Dixmier trace defined from $\Tau$ and $\omega$ on the weak ideal $\mathcal{L}_\psi(\cN):=\{T\in \Ko_\cN: \mu_\Tau(t,T)=O(\psi(t))\}$.
\end{thm*}

The first part of this result can be found as Corollary \ref{cor:phi-omega} (see page \pageref{cor:phi-omega}) in the body of the text and the second part as Corollary \ref{dixmiercorforphiom} (see page \pageref{dixmiercorforphiom}).

\begin{rmk}
\label{rmarkonbetazero}
If $\beta_\D=0$, any unital $\mathrm{Li}_1$-summable semifinite spectral triple with $\Tau(P_\D)=\infty$ has positive $\Tau$-essential spectrum and is $\beta$-analytic. Therefore, Theorem \ref{mainthmconstr} shows that any unital $\mathrm{Li}_1$-summable semifinite spectral triple $(\A,\H,\D,\cN,\Tau)$ with $\Tau(P_\D\e^{-t\D})<\infty$ for $t>0$ and $\Tau(P_\D)=\infty$ gives rise to a tracial state on $A$. This extends a result of Voiculescu \cite[Proposition 4.6]{Voics}. For details on this case, see Theorem \ref{thm:voics} (see page \pageref{thm:voics}).
\end{rmk}

The following three results compute the KMS-state in specific examples.

\begin{thm*}
Let $M$ be a closed Riemannian manifold, $\A:=C^\infty(M)$, $\D$ be a Dirac operator on a Clifford bundle $S\to M$ and $\H:=L^2(M,S)$. Then the KMS-state $\phi_\omega$ constructed in Theorem \ref{mainthmconstr} is independent of $\omega$ and is a tracial state on $C(M)$ that takes the form 
$$
\phi_\omega(a)=\displaystyle\stackinset{c}{}{c}{}{-\mkern4mu}{\displaystyle\int_M} a\,\mathrm{d}V,
$$
where $\mathrm{d}V$ denotes the volume measure defined from the Riemannian metric on $M$ and $\displaystyle\stackinset{c}{}{c}{}{-\mkern4mu}{\displaystyle\int}$ the normalized integral.
\end{thm*}

This result appears as Theorem \ref{simpleacse} (see page \pageref{simpleacse}) in the body of the text.

\begin{thm*}
Let $A$ be a unital $C^*$-algebra, $E$ be a strictly W-regular 
fgp bi-Hilbertian bimodule (see Definitions \ref{cond:one} and \ref{ass:two} on 
pages \pageref{cond:one} and \pageref{ass:two}, respectively) and 
$(\mathcal{O}_E,\Xi_A,\D)$ the associated unbounded 
$(O_E,A)$-cycle as in \cite{GMR}. If $\tau$ is a positive trace on 
$A$, then the semifinite spectral triple 
$(\mathcal{O}_E,\Xi_A\otimes_A L^2(A,\tau),\D\otimes 1_A, (\End^*_A(\Xi_A)\otimes 1)'',\mathrm{Tr}_\tau)$
is $\mathrm{Li}_1$-summable. 

Moreover, if $\tau$ is critical for $E$ (see Definition \ref{defn:critical} 
on page \pageref{defn:critical}), the assumptions in 
Theorem \ref{mainthmconstr} are satisfied and the state 
$\phi_\omega$ is KMS 
for the gauge action on $O_E$. If $\tau$ satisfies the Laca-Neshveyev condition for $\alpha\geq 0$ (see Definition \ref{ass:3.5} on page \pageref{ass:3.5}), then $\phi_\omega$ is independent of $\omega$ and takes the form $\phi_\omega=\phi_{LN,\tau}$ where $\phi_{LN,\tau}$ is the 
KMS-state defined from $\tau$ via the Laca-Neshveyev correspondence.
\end{thm*}

This result is found in Section \ref{diraccpkms} (starting on page \pageref{diraccpkms}). We also discuss extensions of these results to more general $A-A$-correspondences in Subsection \ref{sub:no-left} (starting on page \pageref{sub:no-left}) dispensing the assumption of strict W-regularity.

\begin{thm*}
\label{mainthmongamma}
Let $\Gamma$ be a discrete group and 
$c:\Gamma\to \H_0$ a Hilbert space valued 
proper $1$-cocycle defining a length function of at most exponential growth. 
The semifinite spectral 
triple $(\A,\H,\D,\cN,\Tau)$ constructed from 
$c$ in Subsection \ref{groupcstarexam} is an 
$\mathrm{Li}_1$-summable semifinite spectral 
triple on $C_b(\Gamma)\rtimes\Gamma$. 
Moreover, if $c$ is critical (see Definition \ref{criticaldefn} on page \pageref{criticaldefn}) 
the assumptions of Theorem \ref{mainthmconstr} are satisfied and the 
associated KMS-state $\phi_\omega$ on $C(\partial_{SC} \Gamma)\rtimes\Gamma$ 
is given by 
$$
\phi_\omega\left(\sum_{g\in \Gamma} a_g\lambda_g\right)
=\int_{\partial \Gamma} a_e\,\mathrm{d}\mu_\omega,
$$
where $\mu_\omega$ is a quasi-invariant Patterson-Sullivan 
measure on the Stone-Cech boundary $\partial_{SC}\Gamma$. 
The state $\phi_\omega$ extends to a KMS-state on the von Neumann algebra $L^\infty(\partial_{SC}\Gamma,\mu_\omega)\overline{\rtimes}\Gamma$ where it is KMS with inverse temperature $1$ for the $\R$-action 
defined from the Radon-Nikodym cocycle 
$$
\sigma_t\left( \sum_{g\in \Gamma} a_g\lambda_g\right)
:= \sum_{g\in \Gamma} 
\left(\frac{\mathrm{d}g_*\mu}{\mathrm{d}\mu}\right)^{it}a_g\lambda_g.
$$
\end{thm*}

This result appears as Theorem \ref{themforgamma} (see page \pageref{themforgamma}) below. Our method extends to proper quasi-cocycles, and as such would allow for the construction of KMS-states from semifinite spectral triples with possible $K$-homological content on a-TT-menable groups.

\begin{rmk}
We will prove that the spectral triple of a length 
function (which is $K$-homologically trivial) gives rise to 
the same KMS state as that appearing in Theorem \ref{mainthmongamma}.
\end{rmk}

\subsection{Connection to some earlier work}
\label{sub:Connes-lift}

Here we show how our approach relates to 
some results obtained by Connes in \cite[Section IV.8.$\alpha$, Theorem 4]{BRB}. 
Connes proves that $\theta$-summable Fredholm 
modules can be lifted to 
$\theta$-summable spectral triples. 
We show that Connes' result can be extended to $\mathrm{Li}_s$-summability
for $0<s\leq 1$, 
and discuss obstructions to summability properties of $K$-homology classes. 
For terminology and notations concerning summability and operator ideals, 
the reader is referred forward to Subsection \ref{subsecsemiffe}.

Recall \cite{CC,CGRS2} that a semifinite Fredholm module is a collection $(\A,\H,F,\cN,\Tau)$ where $\A$ acts on the Hilbert space $\H$ by operators from $\cN$ and $F\in \cN$ is an operator with $a(F-F^*), a(F^2-1),[F,a]\in \Ko_\Tau$ for all $a\in \A$. We say that $(\A,\H,F,\cN,\Tau)$ is unital if $\A$ acts unitally. A unital semifinite Fredholm is said to be $\mathrm{Li}_s$-summable if $[F,a]\in \mathrm{Li}_s(\Tau)$ for all $a\in \A$ and $F^2-1,F-F^*\in \mathrm{Li}_{2s}(\Tau)$. If the same conditions holds with $ \mathrm{Li}_s(\Tau)$ replaced by $\mathcal{L}^p(\Tau)$, and $\mathrm{Li}_{2s}(\Tau)$ by $\mathcal{L}^{p/2}(\Tau)$, we say that $(\A,\H,F,\cN,\Tau)$ is $p$-summable. If $(\A,\H,F,\cN,\Tau)$ is a semifinite Fredholm module we say that a semifinite spectral triple $(\A,\H,\D,\cN,\Tau)$ is a lift if $F-\mathrm{sign}(\D)\in \Ko_\Tau$.

\begin{thm*}
\label{liftinglis}
Let $s\in (0,1]$ and $(\A,\H,F,\cN,\Tau)$ be a unital semifinite $\mathrm{Li}_s$-summable Fredholm module with $F^2=1$ and $F=F^*$.
Assume that $\A$ is countably generated. Then there is a self-adjoint operator $\D$ affiliated with $\cN$ making $(\A,\H,\D,\cN,\Tau)$ into a unital semifinite $\mathrm{Li}_s$-summable spectral triple with 
$$F=F_\D:=\D|\D|^{-1}.$$
Moreover, $(\A,\H,\D,\cN,\Tau)$ satisfies that $a\Dom(|\D|^{1/s})\subseteq \Dom(|\D|^{1/s})$ and $[|\D|^{1/s},a]$ has a bounded extension for all $a\in \A$.
\end{thm*} 

This theorem is found in \cite[Section IV.8.$\alpha$, Theorem 4]{BRB} in the special case $s=1/2$ and $\cN=\mathbb{B}(\H)$. We will not give the full details of the proof in the general case, but merely indicate how Connes' proof extends. The starting point of Connes' proof is a reduction to the case that $\mathcal{A}$ contains $F$ and is generated by a countable group of unitaries $\Gamma$ generated by a countable set of unitaries $(u^\mu)_{\mu\in \N}$. This argument extends to a general von Neumann algebra $\cN$. Connes introduces the operator 
$$G:=\sum_{\mu\in \N} \frac{[F,u^\mu]^*[F,u^\mu]}{2^\mu \|[F,u^\mu]^*[F,u^\mu]\|_{\mathrm{Li}_1}}.$$
Since $[F,u^\mu]\in \mathrm{Li}_{1/2}$ for all $\mu$, the series converges in $\mathrm{Li}_1$. The proof proceeds by using an average procedure $\Theta$ over the group $\Gamma$ applied to $G$ and Connes proves that $\D:=F\Theta(G)^{-1/2}$ fulfils the statement of the theorem. For general $s\in (0,1]$, the proof goes mutatis mutandis using the operator 
$$G_s:=\sum_{\mu\in \N} \frac{([F,u^\mu]^*[F,u^\mu])^{2s}}{2^\mu \|([F,u^\mu]^*[F,u^\mu])^{2s}\|_{\mathrm{Li}_1}}\in \mathrm{Li}_1(\Tau),$$
and setting $\D:=F\Theta(G)^{-s}$.

In the special case $s=1$, we obtain that $(\A,\H,F,\cN,\Tau)$ lifts to a unital semifinite $\mathrm{Li}_s$-summable spectral triple $(\A,\H,\D,\cN,\Tau)$ which is Lipschitz regular, i.e.
for all $a\in\A$ the commutator $[|\D|,a]$ is bounded.

The lifting theorem for $\mathrm{Li}_s$-summable spectral triples (Theorem \ref{liftinglis}) stands in sharp contrast to the finitely summable setup, or even the $\mathrm{Li}_{(0),s}$-summable setup. The two upcoming theorems show that a statement as in Theorem \ref{liftinglis} could not extend to the ideal $\mathrm{Li}_{(0),1}$.

\begin{thm*}
\label{tracialobs}
Let $A$ be a unital $C^*$-algebra with no tracial states and $(\A,\H,\D,\cN,\Tau)$ be a unital semifinite spectral triple on $A$ defining a non-trivial class in $KK_1(A,\Ko_\Tau)$. Then $P_\D(i\pm \D)^{-1}\notin \mathrm{Li}_{(0),1}(\Tau)$.
\end{thm*}

\begin{proof}
Consider a unital $C^*$-algebra $A$ and an $\mathrm{Li}_{(0),1}$-summable unital semifinite spectral triple $(\A,\H,\D,\cN,\Tau)$ on $A$ defining a non-trivial class in $KK_1(A,\Ko_\Tau)$. In particular, $\Tau(P_\D)=\infty$; otherwise $P_\D\in \Ko_\Tau$ which contradicts the non-triviality of the $KK_1(A,\Ko_\Tau)$-class defined by $(\A,\H,\D,\cN,\Tau)$. By Remark \ref{rmarkonbetazero} (see page \pageref{rmarkonbetazero}) all assumptions of Theorem \ref{mainthmconstr} reduces to $\Tau(P_\D)=\infty$ in the $\mathrm{Li}_{(0),1}$-summable case. Therefore, the existence of $\mathrm{Li}_{(0),1}$-summable unital semifinite spectral triples on $A$ being non-trivial in $KK$ implies that $A$ admits a tracial state. This argument shows that if $A$ admits no tracial states, it admits no $\mathrm{Li}_{(0),1}$-summable unital semifinite spectral triple. In fact, a careful inspection of the results used show that as soon as there is a unital semifinite spectral triple on $A$ with $P_\D(i\pm \D)^{-1}\in \mathrm{Li}_{(0),1}(\Tau)$, there is an associated tracial state on $A$. The theorem follows.
\end{proof}

There are several $C^*$-algebras carrying no traces, for instance any purely infinite $C^*$-algebra. 
Using Theorems \ref{mainthmconstr} and \ref{tracialobs}, we will give an example of a finitely summable Fredholm module that can not lift to an $\mathrm{Li}_{(0),s}$-summable spectral triple. In particular, lifting of finite summability and $\mathrm{Li}_{(0),s}$-summability fails in general. 

\begin{thm*}
\label{liftingfails}
There is a $C^*$-algebra $A$ with $K^1(A)\neq 0$, such that for a dense $*$-subalgebra $\A\subseteq A$ we can represent any $x\in K^1(A)$ by a Fredholm module $(\A,\H_x,F_x)$ with $F_x^2=1$, $F_x^*=F_x$ and for any $a\in \mathcal{A}$, $[F_x,a]$ is of finite rank. Moreover, any lift $(\A,\H_x,\D_x)$ of $(\A,\H_x,F_x)$ will satisfy that $(1+\D_x^2)^{-1/2}\notin \mathrm{Li}_{(0),1}(\H)$.
\end{thm*}

\begin{proof}
Consider the Cuntz algebra $A=O_N$ and $\A$ the $*$-algebra generated by isometries $S_1,S_2,\ldots, S_N\in O_N$ with orthogonal ranges. It is a well known fact that $K^1(O_N)\cong \Z/(N-1)\Z\neq 0$. By \cite{GM}, we can represent the generator of $K^1(O_N)\cong \Z/(N-1)\Z$ by the Fredholm module 
$(\A,L^2(O_N,\phi),2P-1)$ where $\phi$ is the KMS-state on $O_N$ and $P$ is the orthogonal projection onto the closed linear span of $S_\mu$, where $\mu$ ranges over all finite words on the alpabet $\{1,\ldots, N\}$. By the results of \cite[Section 2.2]{GM}, $[2P-1,a]=2[P,a]$ is finite rank for all $a\in \A$. The first statement of the theorem follows. 

There are no tracial states on $O_N$ since 
$$1_{O_N}=\frac{1}{N-1}(N-1)1_{O_N}=\frac{1}{N-1}\left(\sum_{j=1}^N S_j^*S_j-\sum_{j=1}^N S_jS_j^*\right)=\frac{1}{N-1}\sum_{j=1}^N [S_j^*,S_j].$$
We can now deduce the second statement of the theorem from Theorem \ref{tracialobs}.
\end{proof}

\begin{rmk}
It is not of importance that $K^1(O_N)$ is torsion for the argument in Theorem \ref{liftingfails} to work. In \cite{GM}, non-torsion examples satisfying the conclusions of Theorem \ref{liftingfails} can be found. The reader should also note that the proof of Theorem \ref{liftingfails} obstructs all lifts $(\A,\H_x,\D_x)$ of $(\A,\H_x,F_x)$ with $P_{\D_x}(1+\D_x^2)^{-1/2}\in \mathrm{Li}_{(0),1}(\H)$.
\end{rmk}

In the nonunital case, the techniques of \cite{CGRS1} will
likely be required. The substantial technical considerations 
in the nonunital case goes beyond this paper, and is left to future work.

\subsection{Structure of the paper}

Section \ref{subsec:defns} recalls the basics of (unital) semifinite 
spectral triples and their summability. We also recall our main examples from the literature 
in this section for later use. 

Section \ref{sec:KMS} presents our construction of 
KMS states from $\mathrm{Li}_1$-summable spectral triples. 
We close the section by discussing connections to modular spectral triples.
We consider the case $\beta_\D=0$ in Section \ref{kmsanddix} and compute the tracial states 
constructed in Section \ref{sec:KMS} by means of Dixmier traces.
In Section \ref{kmsinexamplesec} we apply the techniques of Section \ref{sec:KMS} to the examples. 

The final Section \ref{diraccpkms} examines the construction of KMS states for Cuntz-Pimsner algebras. 
In this case we apply our ideas to derive obstructions to the existence of fgp bi-Hilbertian bimodule structures compatible with the underlying correspondence of the Cuntz-Pimsner algebra.

\subsection{Notations}

\begin{tabular}{ll}
$\cN$ & semifinite von Neumann algebra\\ 
$\Tau$ & positive, faithful, normal, semifinite trace  on $\cN$\\
$\Ko_\cN$ & ideal of $\Tau$-compact operators\\
$\End^*_A(X)$ & $C^*$-algebra of adjointable endomorphisms of an $A$-Hilbert $C^*$-module X\\
$\Ko_A(X)$ & $C^*$-algebra of compact endomorphisms of an $A$-Hilbert $C^*$-module X\\
$\A$ & $*$-algebra\\
$\A'$ & the commutant of an algebra $\A$\\
$A$ & $C^*$-closure of an algebra $\A$\\
$A_\D$ &  saturation of $A$ under the action of the wave group $\e^{it\D}$\\
$P_\D:=\chi_{[0,\infty)}(\D)$ & non-negative spectral projection of an operator $\D$\\
$F_\D:=2 P_\D-1$\\
$\mu_\Tau(\cdot, T)$ & singular values function of an operator $T$ affiliated with $\cN$\\
$n_\Tau(\cdot,T)$ & distribution function of an operator $T$ affiliated with $\cN$\\
$\mathrm{Li}_s(\Tau), \mathrm{Li}_{(0),s}(\Tau)$ & ideals of compact operators in Definition~\ref{derjugendvonheutebrauchidealen} on page~\pageref{derjugendvonheutebrauchidealen}\\
$\mathcal{L}_\psi(\Tau), \mathcal{L}_{(0),\psi}(\Tau)$ & ideals of compact operators in Definition~\ref{derjugendvonheutebrauchidealen} on page~\pageref{derjugendvonheutebrauchidealen}\\
$\Tau_{\omega, \psi}(T)$ & Dixmier trace on $\mathcal L_\psi(\Tau)$\\
$\mathfrak{T}(A)$ & set of positive traces on a unital $C^*$-algebra $A$\\

$L^\infty(a,\infty), \ a\ge0$ & space of essentially bounded functions on $(a,\infty)$ equipped with\\ & the essential supremum norm\\
$C_0(a,\infty), \ a\ge0$ & subspace of $L^\infty(a,\infty)$ of all continuous functions vanishing at infinity\\
$\lim_{t\to \omega} f(t)$ & value of an extended limit $\omega$ on a function $f$\\
$\ell^\infty(\N)$ &space of bounded sequences equipped with the supremum norm\\
$\lim_{k\to \omega} x_k$ & value of an extended limit $\omega$ on a sequence $x$\\
$\mathfrak{L}_g$ & transfer operator defined by formula~\eqref{transferoperator} on page~\pageref{transferoperator}\\
$f \sim g$ & for two functions or sequences $f$ and $g$ if $f=g+o(f)$ and $g=f+o(g)$\\

\end{tabular}

\subsection{Acknowledgements}

A. R. thanks the Gothenburg Centre for Advanced Studies in Science and Technology 
for funding and the University of Gothenburg and Chalmers University of Technology 
for their hospitality in 2017 when this work was begun.
M. G. and A. U. were supported by the Swedish Research Council Grant 2015-00137 and 
Marie Sklodowska Curie Actions, Cofund, Project INCA 600398. The
authors acknowledge the support of the
Erwin Schr\"{o}dinger Institute where part of this work was conducted.

The authors are grateful to Alan Carey, Heath Emerson and Bram Mesland for inspiring discussions. We also thank Branimir \'{C}a\'{c}i\'{c}  for sharing his construction of semifinite spectral triples from proper group cocycles, and Edward McDonald for references on previous work in that direction.

\section{Preliminaries}
\label{subsec:defns}

Before entering into the body of the paper, we recall some 
basic definitions that we will require and provide some examples 
that motivated this work. These examples will be studied further in the later sections of the paper. 
The results will be formulated for semifinite spectral triples. 
We do however remark that there are several examples of `vanilla' spectral triples 
that will be used throughout the paper.

\subsection{Semifinite spectral triples and summability}
\label{subsecsemiffe}

To set the stage for the paper, we summarize the basic definitions and properties 
of semifinite spectral triples. 
The reader familiar with semifinite spectral triples 
and symmetrically normed operator ideals can skip this subsection.

We let $\cN$ denote a semifinite von Neumann algebra and 
we fix a positive, faithful, normal, semifinite trace $\Tau$ on $\cN$. 
The $\Tau$-compact operators are denoted by $\Ko_\cN$. 
The $C^*$-algebra $\Ko_\cN$ can be defined as the norm closed ideal 
generated by the projections $E\in\mathcal N$ with $\Tau(E)<\infty$. 
Equivalently, one can define $\Ko_\cN:=\{T\in \cN: \mu_\Tau(t,T)=o(1)$ as $t\to \infty\}$
where the singular value function $\mu_\Tau(t,T)$ is defined as
\begin{equation}\label{mu}
\mu_\Tau(t,T):=\inf\big\{\|T(1-E)\|_\cN: \mbox{ where $E\in \cN$ is a projection with $\Tau(E)\leq t$}\big\}.
\end{equation}

\begin{defn}
\label{thedefn}
A semifinite spectral triple $(\A,\H,\D,\cN,\Tau)$ consists of
\begin{itemize}
\item A $*$-algebra $\A$ represented on a Hilbert space $\H$ as operators in $\cN\subseteq \mathbb{B}(\H)$, that is, we have a specified $*$-homomorphism $\pi:\A\to\cN$.
\item A densely defined self-adjoint operator
$
\D:{\rm dom}\ \D\subset\H\to\H
$
which is affiliated with $\cN$ such that for all $a\in \A$
we have $a\cdot\mathrm{dom}\D\subset\mathrm{dom}\D$ and
\begin{enumerate}
\item $[\D,\pi(a)]:=\D \pi(a)-\pi(a)\D$ initially defined on $\mathrm{dom}(\D)$ 
is bounded in operator norm.
\item $\pi(a)(1+\D^2)^{-1/2}\in \Ko_\cN$.
\end{enumerate}
\end{itemize}
\end{defn}

\begin{rmk}
Sometimes we write a spectral triple as a collection of three objects  $(\A,\H,\D)$. In this case, it is implicitly assumed that $\cN=\mathbb{B}(L^2(M,S))$ and $\Tau$ is the standard trace.
\end{rmk}

\begin{rmk}
If in addition to the data $(\A,\H,\D,\cN,\Tau)$ we have specified an operator 
$\gamma\in\B(\H)$ with $\gamma=\gamma^*$, $\gamma^2=1$, 
$\D\gamma+\gamma\D=0$ on $\Dom(\D)$, and for all $a\in\A$ we have $\gamma\pi(a)=\pi(a)\gamma$, 
we call the semifinite spectral triple even, or sometimes graded. 
If $\gamma$ has not been specified, we say that the semifinite spectral triple is odd, or ungraded.
This distinction plays an important role in the topological properties of the spectral triple, 
but since this paper deals with measure theory it will not play a role in this paper. 
\end{rmk}

\begin{rmk}
We will nearly always dispense with the
representation $\pi$, treating $\A$ as a 
subalgebra of $\cN\subseteq \B(\H)$.
\end{rmk}

\begin{rmk}
In the sequel we assume that the algebra $\A$ is unital and that $1\in\A$ 
acts as the identity of the Hilbert space. In particular, the operator
$(1+\D^2)^{-1/2}$ is a $\Tau$-compact operator. 
To emphasize this assumption, we refer to the data $(\A,\H,\D,\cN,\Tau)$ 
as a \emph{unital} semifinite spectral triple.
\end{rmk}

Examples of semifinite spectral triples often satisfy a finer summability structure, 
i.e. a refinement  of the condition  $(1+\D^2)^{-1/2}\in \Ko_\cN$. We formulate such conditions in terms of symmetrically quasi-normed operator ideals. We will  use  the Schatten ideals, the $\mathrm{Li}$-ideals and more generally weak ideals. 

\begin{defn}
\label{derjugendvonheutebrauchidealen}
Let $\cN$ denote a semifinite von Neumann algebra and 
$\Tau$ a positive, faithful, normal, semifinite trace on $\cN$. 
For parameters $p,d\in [1,\infty)$ and $s>0$ we define the following operator ideals.
\begin{itemize}
\item $\mathcal{L}^p(\Tau):=\{T\in \Ko_\cN: \mu_\Tau(\cdot,T)\in L^p(0,\infty)\}$.
\item $\mathcal{L}^{(d,\infty)}(\Tau):=\{T\in \Ko_\cN: \mu_\Tau(\cdot,T)=O(t^{-1/d})\mbox{  as $t\to \infty$}\}$.
\item $\mathrm{Li}_s(\Tau):=\{T\in \Ko_\cN:\mu_\Tau(t,T)=O((\log(t))^{-s})\mbox{  as $t\to \infty$}\}$.
\item $\mathrm{Li}_{(0),s}(\Tau):=\{T\in \Ko_\cN:\mu_\Tau(t,T)=o((\log(t))^{-s})\mbox{  as $t\to \infty$}\}$.
\item If $\psi:[0,\infty)\to (0,\infty)$ is a decreasing function satisfying that $\sup_{t>0}\frac{\psi(t)}{\psi(2t)}<\infty$, we define the associated weak ideal
$$\mathcal{L}_\psi(\Tau):=\{T\in \Ko_\cN: \mu_\Tau(t,T)=O(\psi(t))\},$$
and its separable subspace
$$\mathcal{L}_{(0),\psi}(\Tau):=\{T\in \Ko_\cN: \mu_\Tau(t,T)=o(\psi(t))\},$$
\end{itemize}
\end{defn}

The condition $\sup_{t>0}\frac{\psi(t)}{\psi(2t)}<\infty$ guarantees that $\mathcal{L}_\psi(\Tau)$ is a vector space, and in fact even a quasi-Banach space in the quasi-norm 
$$\|T\|_{\mathcal{L}_\psi}:=\sup_{t>0}\frac{\mu_\Tau(t,T)}{\psi(t)}.$$ 
Note that $\mathrm{Li}_s(\Tau)=\mathcal{L}_\psi(\Tau)$ and $\mathrm{Li}_{(0),s}(\Tau)=\mathcal{L}_{(0),\psi}(\Tau)$ for $\psi(t):=(\log(2+t))^{-s}$. 
It is immediate from the definition that 
$\mathcal{L}^p(\Tau)\subseteq \mathrm{Li}_{(0),s}(\Tau)$ for any $p$ and $s$. 
More generally, if $\psi_1,\psi_2:[0,\infty)\to (0,\infty)$ are two decreasing functions satisfying that $\sup_{t>0}\frac{\psi_j(t)}{\psi_j(2t)}<\infty$, then $\mathcal{L}_{\psi_1}(\Tau)\subseteq \mathcal{L}_{\psi_2}(\Tau)$ as soon as $\psi_1=O(\psi_2)$.

\begin{rmk}
Our definition of symmetrically normed operator ideals in the semifinite setting 
differs slightly from the standard definition unless $\cN$ is atomic. 
In the usual definition, the symmetrically normed operator ideals are 
defined from operators affiliated with $\cN$ that potentially are unbounded. 
Since we only use bounded operators from these ideals, we have incorporated 
this fact in our definition.
\end{rmk}

\begin{defn}
\label{summdefn}
Let $(\A,\H,\D,\cN,\Tau)$ be a unital semifinite spectral triple.
\begin{itemize}
\item $(\A,\H,\D,\cN,\Tau)$ is said to be $p$-summable if $(1+\D^2)^{-1/2}\in \mathcal{L}^p(\Tau)$.
\item $(\A,\H,\D,\cN,\Tau)$ is said to be $(d,\infty)$-summable if $(1+\D^2)^{-1/2}\in \mathcal{L}^{(d,\infty)}(\Tau)$.
\item $(\A,\H,\D,\cN,\Tau)$ is said to be $\mathrm{Li}_s$-summable if $(1+\D^2)^{-1/2}\in \mathrm{Li}_s(\Tau)$.
\item $(\A,\H,\D,\cN,\Tau)$ is said to be $\mathrm{Li}_{(0),s}$-summable if $(1+\D^2)^{-1/2}\in \mathrm{Li}_{(0),s}(\Tau)$.
\item $(\A,\H,\D,\cN,\Tau)$ is said to be $\psi$-summable if $(1+\D^2)^{-1/2}\in \mathcal{L}_\psi(\Tau)$.
\end{itemize}
\end{defn}

The standard terminology in the literature for the special case $s=1/2$ is to refer to $\mathrm{Li}_{1/2}$-summability as weak $\theta$-summability and to $\mathrm{Li}_{(0),1/2}$-summability as $\theta$-summability. Since $\mathcal{L}^{(d,\infty)}(\Tau)\subseteq \mathcal{L}^{p}(\Tau)$ for all $p>d$, $(d,\infty)$-summability refines $p$-summability.

The notion of $(d,\infty)$-summability is a noncommutative generalization of being $d$-dimensional as the spectral triple defined from a Dirac operator on a closed $d$-dimensional manifold (as in Subsection~\ref{diracmfdfirst}) is $(d,\infty)$-summable. We shall see an abundance of $\mathrm{Li}_1$-summable, truly noncommutative, examples where $p$-summability and $(d,\infty)$-summability fails for all $p$ and $d$. 

The notion of $\psi$-summability generalizes both $(d,\infty)$-summability and $\mathrm{Li}_1$-summability, and appears naturally in examples of (semi-) group actions on manifolds (see Subsection \ref{diracmfdkms} and \cite{DGMW,GU}). We will make use of this notion in Section \ref{kmsanddix} where certain conditions on $\psi$ allows one to compute the tracial state defined from a $\psi$-summable unital semifinite spectral triples in terms of Dixmier traces on $\mathcal{L}_\psi(\Tau)$.

\begin{rmk}
It is readily verified that $\mathrm{Li}_s$-summability is equivalent to 
$$\Tau(\e^{-t|\D|^{1/s}})<\infty, \quad\mbox{for $t>t_0$ for some critical value $t_0$},$$
and that $\mathrm{Li}_{(0),s}$-summability is equivalent to 
$$\Tau(\e^{-t|\D|^{1/s}})<\infty, \quad\mbox{for $t>0$}.$$

In particular, $(\A,\H,\D,\cN,\Tau)$ is $\theta$-summable if and only if 
$$\Tau(\e^{-t\D^2})<\infty, \quad\mbox{for all $t>0$}.$$ 

Historically, $\theta$-summability has been studied more in depth than $\mathrm{Li}_1$-summability. This can in part be explained from the two facts that the JLO-cocycle only requires $\theta$-summability and classically, the heat operator $\mathrm{e}^{-t\D^2}$ is geometrically more interesting than $\mathrm{e}^{-t|\D|}$ to study on a manifold. The two operators $\mathrm{e}^{-t\D^2}$ and $\mathrm{e}^{-t|\D|}$ can be compared by explicit integral formulas, see \cite[Chapter 4]{ggneumann}.  We will exploit the observation that large classes of examples of $\theta$-summable spectral triples are also $\mathrm{Li}_1$-summable.
\end{rmk}

In the bulk of the paper, we are interested in computing asymptotics of heat traces 
of the form $\Tau(B\e^{-t|\D|^{1/s}})$ for $B\in \cN$ as $t$ approaches a critical value. When $(\A,\H,\D,\cN,\Tau)$ is $\mathrm{Li}_{(0),s}$-summable, 
the critical value of $t$ is $0$, and in several classical examples (e.g. on closed manifolds)
the heat trace $\Tau(B\e^{-t|\D|^{1/s}})$ admits an asymptotic expansion. 
The following result is useful for relating heat trace asymptotics to 
zeta function asymptotics in the case of the nice behaviour appearing when $t_0=0$.

\begin{thm}
\label{heatvszeta} Let $s\in(0,1]$. 
Let $(\A,\H,\D,\cN,\Tau)$ be an $\mathrm{Li}_{(0),s}$-summable semifinite spectral triple and $B\in \cN$. The following are equivalent.
\begin{enumerate}
\item There are constants $p^{\rm heat}>0$, $\epsilon>0$ and $c_B^{\rm heat}\in \C$ such that 
$$\Tau(B\e^{-t|\D|^{1/s}})=c_B^{\rm heat}t^{-sp^{\rm heat}}+O(t^{-sp^{\rm heat}+\epsilon}), \quad\mbox{as $t\to 0$}.$$
\item The $\zeta$-function $\zeta(z;B,|\D|^{1/s}):=\Tau(B|\D|^{-z/s})$ is well-defined for large $\mathrm{Re}(z)$ and there are constants $p^{\zeta}>0$, $\epsilon'>0$, $c_B^{\zeta}\in \C$ and a function $f=f(z)$ holomorphic in the region $\mathrm{Re}(z)>sp^{\zeta}-\epsilon'$ such that 
$$\zeta(z;B,|\D|^{1/s})=\frac{1}{\Gamma(sp^\zeta)}\frac{c_B^{\zeta}}{z-sp^\zeta}+f(z).$$
\end{enumerate}
In this case, $p^{\rm heat}=p^{\zeta}$ and $c_B^{\rm heat}=c_B^\zeta$. Moreover, if the conditions above hold for one $s\in (0,1]$, it holds for all $s\in (0,1]$.

If $B=1$ and either of the conditions above hold, then $(\A,\H,\D,\cN,\Tau)$ is $p^\zeta$-summable. 
\end{thm}

The proof of Theorem \ref{heatvszeta} follows by noting that $\Gamma(z)\zeta(z;B,|\D|^{1/s})$ 
is the Mellin transform of $\Tau(B\e^{-t|\D|^{1/s}})$ and using \cite[Proposition 5.1]{grubbsee}.

Recall that we use the notation $P_\D:=\chi_{[0,\infty)}(\D)$ for 
the non-negative spectral projection of $\D$. 
Let us state a fundamental lemma 
on the commutators of $A$ with the 
function of $\D$ defined by 
$$
F_\D:=
2P_\D-1.
$$
Note that $F_\D$ differs from the phase $\D|\D|^{-1}$ by
the $\Tau$-finite kernel projection of $\D$.

\begin{lemma}
\label{bddcommsum}
Let $(\A,\H,\D,\cN,\Tau)$ be a semifinite spectral triple. Then for any $a\in A$ 
$$
[F_\D,a]\in \Ko_\cN.
$$ 
Moreover if $(\A,\H,\D,\cN,\Tau)$ is unital, then if $(\A,\H,\D,\cN,\Tau)$ is $p$-summable, then $[F_\D,a]\in \mathcal{L}^p(\Tau)$ for all $a\in \A$, and if $(\A,\H,\D,\cN,\Tau)$ is $\mathrm{Li}_s$-summable, then $[F_\D,a]\in \mathrm{Li}_s(\Tau)$ for all $a\in \A$. More generally, if $(\A,\H,\D,\cN,\Tau)$ is $\psi$-summable then $[F_\D,a]\in \mathcal{L}_\psi(\Tau)$ for all $a\in \A$.
\end{lemma}

\begin{proof} The proof of the operator inequality $-\|[\D,a]\||\D|^{-1} \leq [F_\D,a]\leq\|[\D,a]\||\D|^{-1}$ for invertible $\D$ and $a=-a^*$ is found in the proof of \cite[Proposition 1]{sww}. The assertion follows from the definition of $p$-, $\rm{Li}_s$- and $\psi$-summability, resp.

In the non-invertible case we replace $(\A,\H,\D,\cN,\Tau)$ by
\begin{equation}
\label{double-trick}
\left(\begin{pmatrix}\A&0\\ 0 & 0\end{pmatrix},\H\oplus \H,\D_\mu=\begin{pmatrix} \D&\mu\\ \mu&-\D\end{pmatrix},M_2(\cN),\Tau\otimes \mathrm{Tr}_{M_2}\right),
\end{equation}
for $\mu\in[0,1]$.  When $\mu>0$ we are back in the invertible case. In \cite[Proposition 2.25]{CGRS2} it is shown that for any $a\in\A$ we have the norm limit
$$
\begin{pmatrix} [F_\D,a] & 0\\
0 & 0\end{pmatrix}=\lim_{\mu\to 0}[F_{\D_\mu},a]
$$
and so $[F_\D,a]$ is compact. Indeed, the proof 
of \cite[Proposition 2.25]{CGRS2} shows that 
if $(1+\D^2)^{-1/2}$ is in a symmetrically quasi-normed
ideal $J$ of $\Tau$-compact operators then 
$[F_\D,a]\in J$ for all $a\in\A$. 
Again the assertion follows from the definitions of summability.
\end{proof}

\subsubsection{Semifinite spectral triples from unbounded Kasparov modules}

For several kinds of $C^*$-algebras one can capture the noncommutative geometry through an unbounded Kasparov module. This is a bivariant generalization of spectral triples. Localizing an unbounded Kasparov module in a positive trace gives rise to a semifinite spectral triple as in Theorem \ref{localintra} below. Several of the examples in this paper arises in this way. We briefly recall this construction, which has been informally used for some years. 

Let $A$ and $B$ be unital $C^*$-algebras. A unital unbounded $(B,A)$-Kasparov module is a collection $(\mathcal{B},X,\D)$ where 
\begin{itemize}
\item $\mathcal{B}\subseteq B$ is a dense $*$-subalgebra, 
\item $X$ is an $A$-Hilbert $C^*$-module carrying a left action of $B$ as adjointable operators, 
\item $\D$ is an $A$-linear, densely defined, self-adjoint, regular operator on $X$ with $A$-compact resolvent $(i\pm \D)^{-1}\in \Ko_A(X)$ and 
\item for $a\in \mathcal{B}$ the operator $[\D,a]$ is defined on $\mathrm{dom}(\D)$ and is bounded in the norm on $X$. 
\end{itemize} 
If $\tau$ is a positive trace on $A$, we write $L^2(X,\tau):=X\otimes_A L^2(A,\tau)$ where $L^2(A,\tau)$ is the GNS-representation associated with $\tau$. 

For $\xi,\eta\in X$, we write $\Theta_{\xi,\eta}$ for the rank one operator $\Theta_{\xi,\eta}(\nu)=\xi(\eta|\nu)_A$. The von Neumann algebra $\cN_\tau(X):=(\End^*_A(X)\otimes 1_A)''\subseteq \mathbb{B}(L^2(X,\tau))$ coincides with the weak closure of the set of operator spanned by $\{\Theta_{\xi,\eta}\otimes 1_A: \xi,\eta\in X\}$ and carries a positive, normal, semifinite, faithful trace $\Tr_\tau$ characterized by $\Tr_\tau(\Theta_{\xi,\eta}\otimes 1_A):=\tau((\eta|\xi)_A)$, see \cite[Section 3]{LN}. The following theorem also appears in \cite{MacR}.

\begin{thm}
\label{localintra}
Let $(\B,X_A,\D)$ be a unital unbounded $(B,A)$-Kasparov module
and $\tau:\,A\to\C$ a faithful norm densely defined norm lower semicontinuous tracial weight. Then the data
$(\B,L^2(X,\tau),\D\ox 1,\cN_\tau(X),\Tr_\tau)$ defines a semifinite spectral triple.
The von Neumann algebra is $\cN_\tau(X)=(\End_A^*(X)\ox 1)''$ and $\Tr_\tau:\,\cN\to\C$ is
the (positive faithful semifinite normal) trace dual to the normal extension of $\tau$ to $A''\subseteq \mathbb{B}(L^2(A,\tau))$.
\end{thm}

\begin{proof}
The operator $\D\ox 1$ is self-adjoint by \cite[Proposition 9.10]{lancesbook}. The commutant of $\cN_\tau(X)$ in $X\ox_AL^2(A,\tau)$ is the algebra $A''$ (acting by right multiplication). Every unitary in $A''$ thus preserves the domain of $\D\ox 1$ and so $\D\ox 1$ is affiliated to $\cN_\tau(X)$. Plainly commutators of $\D\ox 1$ with $\B$ remain bounded.
Since we start with an unbounded Kasparov module, the operator $(1+\D^2)^{-1/2}\in \Ko_A(X)$. So we can approximate $(1+\D^2)^{-1/2}$ in norm by
finite rank operators $\sum_j\Theta_{x_j,y_j}$ and we can take the $x_j,\,y_j\in X$. Hence 
$(1+(\D\ox 1)^2)^{-1/2}=(1+\D^2)^{-1/2}\ox 1$ is in the norm closure of the finite trace operators in $\cN_\tau(X)$, and so $\Tau$-compact.
\end{proof}

A conceptual viewpoint is that 
$(\mathcal{B},L^2(X,\tau),\D\otimes 1_A, (\End^*_A(X)\otimes 1_A)'', \Tr_\tau)$ 
is a semi-finite refinement of the unbounded Kasparov product of 
$(\mathcal{B},X,\D)$ with the Morita morphism $A\to \Ko_A(X)$ and the $*$-homomorphism 
$\Ko_A(X) \to \mathbb{K}_{(\End^*_A(X)\otimes 1_A)''}$.
There is a close relationship between the semifinite index and the Kasparov product, described in \cite{CGRS2,KNR}.

We remark at this stage that there is to date no general theory of 
symmetrically quasi-normed operator ideals in Hilbert $C^*$-modules, 
and, as such, no satisfactory way of describing summability. 
In concrete applications, it is possible to circumvent this problem 
by choosing a frame on $X$, implicitly using the machinery of \cite{nistann}. 
In the examples of most relevance to this paper the spectrum of $\D$ is discrete, 
and we can use the following proposition to study summability.

\begin{prop}
Let $(\mathcal{B},X,\D)$ be a unital unbounded $(B,A)$-Kasparov module, where $\D$ has discrete spectrum $\sigma(\D)\subseteq \R$, and $\tau$ be a positive trace on $A$. Set $P_\lambda:=\chi_{\{\lambda\}}(\D)\in \Ko_A(X)$ for $\lambda\in \sigma(\D)$. Then it holds that 
\begin{enumerate}
\item $(\mathcal{B},L^2(X,\tau),\D\otimes 1_A, \cN_\tau(X), \Tr_\tau)$ is $\mathrm{Li}_s$-summable if and only if 
$$\sum_{\lambda\in \sigma(\D)} \e^{-t|\lambda|^{1/s}}\Tr_\tau(P_\lambda)<\infty,$$ 
for $t$ large enough.
\item $(\mathcal{B},L^2(X,\tau),\D\otimes 1_A, \cN_\tau(X), \Tr_\tau)$ is $p$-summable if and only if 
$$\sum_{\lambda\in \sigma(\D)} (1+\lambda^2)^{-p/2}\Tr_\tau(P_\lambda)<\infty.$$
\end{enumerate}
\end{prop}

The proof follows from the following formula:
$$\Tr_\tau(f(\D))=\sum_{\lambda\in \sigma(\D)} f(\lambda)\Tr_\tau(P_\lambda),$$
which holds for every positive Borel function $f$.

\subsection{Examples}
\label{subsec:examples}

To give some further context before entering into the main construction of this paper, 
let us recall some well known examples that we will further explore later on in the paper. 
The focus in our presentation is on $\mathrm{Li}_1$-summability and heat traces.
We remark that the constructions in this subsection are rather lengthy, and the reader
familiar with the literature can at a first read restrict themself to glancing through 
this subsection.

\subsubsection{Dirac operators on closed manifolds}
\label{diracmfdfirst}

The prototypical example of a spectral triple arises from Dirac operators 
on a closed Riemannian manifold $M$. 
We can work with a rather general type of Dirac operators: if $S\to M$ is 
a Clifford module on $M$ and $\slashed{D}$ is a first order elliptic operator acting on 
$C^\infty(M,S)$ being symmetric in the $L^2$-inner product and 
$\slashed{D}^2$ is a Laplacian type operator\footnote{I.e. the symbol of $\slashed{D}^2$ coincides with the Riemannian metric as a function on $T^*M$.}, we say that $\slashed{D}$ is a Dirac operator. 
In this case, the closure of $\slashed{D}$ in its graph norm defines a 
self-adjoint operator on $L^2(M,S)$ that we by an abuse of notation also denote by $\slashed{D}$. 
It is well-known that $(C^\infty(M),L^2(M,S),\slashed{D})$ is a spectral triple on $C^\infty(M)$. 
We summarize the main properties of its heat traces in the following proposition.

\begin{prop}
\label{heatasumfd}
Let $M$ be an $n$-dimensional Riemannian closed manifold, $\slashed{D}$ a Dirac operator on $M$, and $(C^\infty(M),L^2(M,S),\slashed{D})$ the associated spectral triple. This spectral triple is $(n,\infty)$-summable and for any classical zero-th order pseudo-differential operator $A$ on $S$ with principal symbol $a\in C^\infty(S^*M, \mathrm{End}(S))$ and every $s\in(0,1]$ we have
$$
\Tr_{L^2(M,S)}(A\e^{-t|\slashed{D}|^{1/s}})
=\Gamma(sn+1)c_n^st^{-sn}\int_{S^*M} \mathrm{Tr}_S(a)\,\mathrm{d}V +O(t^{-sn+\epsilon}), \quad\mbox{as $t\to 0$},
$$
for some dimensional constant $c_n>0$ and $\epsilon>0$. Here $\mathrm{Tr}_S(a)\in C^\infty(S^*M)$ denotes the fibrewise trace of $a$.
\end{prop}

The dimensional constant $c_n$ is determined by the Weyl law for $|\slashed{D}|$ describing the ordered sequence $(\lambda_k(|\slashed{D}|))_{k\in \N}$ of eigenvalues as
$$\lambda_k(|\slashed{D}|)=\frac{1}{c_n\mathrm{vol}(M)^{1/n}\mathrm{rank}(S)^{1/n}}k^{1/n}+O(k^{1/n-\epsilon_0}),$$
for some $\epsilon_0>0$. The Weyl law is proven in
many places, for instance \cite{Gilkey}. The general heat trace asymptotics follows from Theorem \ref{heatvszeta} and \cite{Gilkey}. We return to this example below in Example \ref{diracmfdhea} (see page \pageref{diracmfdhea}) and Subsection \ref{diracmfdkms} (see page \pageref{diracmfdkms}).

Later on in the paper, we will make use of a modification of the spectral triple coming from a Dirac operator that is also compatible with non-isometric semigroup actions. Similar constructions were previously considered in \cite{DGMW,GU}.

\begin{defn}
\label{regcariaden}
Let $\psi:[0,\infty)\to (0,\infty)$ be a positive measurable function. 
\begin{itemize}
\item  We say that $\psi$ is regularly varying of index $\rho$ if for all $\lambda>0$
\begin{equation}
\label{varfrho}
\lim_{t\to \infty} \frac{\psi(\lambda t)}{\psi(t)}=\lambda^\rho.
\end{equation}
\item We say that $\psi$ is smoothly regularly varying of index $\rho$ if $\psi\in C^\infty$ and for any $k\in \N$,
\begin{equation}
\label{smoothvarfrho}
\lim_{t\to \infty} \frac{t^k \psi^{(k)}(t)}{\psi(t)}=\rho(\rho-1)\cdots(\rho-k+1).
\end{equation}
\end{itemize}
\end{defn}

A regularly varying function satisfies $\sup_{t>0}\frac{\psi(t)}{\psi(2t)}<\infty$ and there is an associated weak ideal $\mathcal{L}_\psi$ as in Definition \ref{derjugendvonheutebrauchidealen}. By \cite[Lemma 7.1]{GU}, any smoothly regularly varying function $\psi$ satisfies that $\partial_t^k\psi(t)=O(\psi(t)(1+t^2)^{-k/2})$ so if $\psi$ additionally is bounded, $\psi$ belongs to the H\"ormander class $S^0$. 

Smooth regular variation is a strengthening of having regular variation -- a condition used below in Section \ref{kmsanddix} in the context of defining and computing Dixmier traces. See more also in \cite{GU}. By \cite[Theorem 1.8.2]{RegVar}, any regularly varying function asymptotically behaves like a smoothly regularly varying function. Smooth regular variation allows for defining associated classes of pseudo-differential operators and computing Dixmier traces of geometric operators by means of a Connes trace theorem, see \cite[Section 7 and 9]{GU}. 

For a decreasing smoothly varying function $\psi:[0,\infty)\to (0,\infty)$ with $\lim_{t\to 0}\psi(t)=0$, we define the self-adjoint operator 
$$\slashed{D}_\psi:=F_{\slashed{D}}\psi(|\slashed{D}|^n)^{-1}.$$
It follows from \cite[Proposition 10.1]{GU} that $\slashed{D}_\psi\in L^0_{\psi^{-1}}(M,S)$ (see~\cite{GU} for the meaning of this symbol) and its $\psi$-principal symbol is $c_S(\xi)|\xi|^{-1}\psi(|\xi|)^{-1}$, where $c_S:T^*M\to \End(S)$ denotes Clifford multiplication. The next result follows from \cite[Proposition 10.3]{GU}. 

\begin{prop}
\label{heatasumfdwithapsi}
Let $M$ be an $n$-dimensional Riemannian closed manifold, $\slashed{D}$ a Dirac operator on $M$, and $\psi$ as above with 
$$\psi(t)^{-1}=O(t^{1/n}),\quad\mbox{as $t\to \infty$}.$$
Then $(C^\infty(M),L^2(M,S),\slashed{D}_\psi)$ is a $\psi$-summable spectral triple whose associated $K$-homology class coincides with that of $(C^\infty(M),L^2(M,S),\slashed{D})$.
If $\psi(t)^{-1}=o(t^{1/n})$, then for any $a\in C^\infty(M)$, the operator $[\slashed{D}_\psi,a]$ is compact with 
$$\mu(t,[\slashed{D}_\psi,a])=O(t^{-1/n}\psi(t)^{-1}),\quad\mbox{as $t\to \infty$}.$$
\end{prop}

We return to the problem of computing heat traces involving $\slashed{D}_\psi$ below in the Section \ref{kmsanddix} (see page \pageref{kmsanddix}) and Subsection \ref{diracmfdkms} (see page \pageref{diracmfdkms}). 

The spectral triple $(C^\infty(M),L^2(M,S),\slashed{D}_\psi)$ does in some cases extend to a crossed product by a semigroup action. For simplicity, we consider an action of $\N$ by a local diffeomorphism $g:M\to M$. Following \cite{GU}, we say that $g$ acts conformally if there is a function $c_g\in C^\infty(M,\R_{>0})$ such that $g^*g_M=c_gg_M$ where $g_M$ denotes the Riemannian metric on $M$. We say that $g$ lifts to the Clifford bundle $S\to M$ if there is a unitary Clifford linear morphism $u_g:g^*S\to S$. For simplicity, we assume $M$ to be connected and define $N:=\#g^{-1}(\{x\})$ for some $x\in M$. Since $g$ is a local diffeomorphism and $M$ is connected, $N$ is independent of the choice of $x$. 

If $g:M\to M$ is a surjective local diffeomorphism, acting conformally and lifting to $S$, we can define the isometry
\begin{equation}
\label{vgdef}
V_g:L^2(M,S)\to L^2(M,S), \quad V_g\xi:=c_g^{n/4}N^{-1/2} u_g(\xi\circ g).
\end{equation}
The isometry $V_g$ satisfies the following for $a\in C(M)$:
$$V_gaV_g^*=(a\circ g)V_gV_g^*,\quad\mbox{and}\quad V_g^*aV_g=\mathfrak{L}_g(a),$$
where $\mathfrak{L}_g(a)(x):=\sum_{g(y)=x}a(y)$. See more in \cite[Proposition 8.3]{DGMW}. Define $\mathcal{A}$ as the $*$-algebra generated by $C^\infty(M)$ and $V_g$. By  \cite[Proposition 8.6]{DGMW}, the $C^*$-closure of $\A$ is the image in a representation on $L^2(M,S)$ of the Cuntz-Pimsner algebra $O_{E_g}$ defined from $E_g:=C(M)V_g$ (see more in Example \ref{localhomeofirst} (see page \pageref{localhomeofirst}) and Example \ref{localhomeoex} (see page \pageref{localhomeoex})).
The space $E_g=C(M)V_g\subseteq \mathbb{B}(L^2(M,S))$ is considered as a bimodule over $C(M)$ and is an fgp bi-Hilbertian $C(M)$-bimodule because $E_g$ can be identified with $C(M)$ with the bimodule structure $(afb)(x)=a(x)f(x)b(g(x))$ for $a,b,f\in C(M)$, for details, see Example \ref{localhomeofirst} below on page \pageref{localhomeofirst} or \cite[Section 8]{DGMW}.

\begin{defn}
Let $\psi:[0,\infty)\to (0,\infty)$ be a decreasing function and $g:M\to M$ a local diffeomorphism of a Riemannian manifold. 
\begin{itemize}
\item We say that $\psi$ and $g$ are compatible if there is a constant $C\geq 0$ such that for all $x\in M$ and $\xi\in T^*_xM$, the differential $Dg$ satisfies 
$$
\left|\psi(|(Dg)^T_x\xi|)-\psi(|\xi|)\right|\leq C\psi(|\xi|)^2.
$$
\item We say that $g$ acts isometrically if $Dg$ acts isometrically on each fibre.
\end{itemize}
\end{defn}

The following results poses restrictions on a function $\psi$ compatible with a local diffeomorphism which acts non-isometrically.

\begin{prop}
\label{deathorgladiolis}
Let $\psi:[0,\infty)\to (0,\infty)$ be a decreasing function with $\lim_{t\to \infty}\psi(t)=0$ and $g:M\to M$ a compatible local diffeomorphism of a Riemannian manifold. If $\psi$ has regular variation, then either $g$ acts isometrically or $\psi$ has regular variation of index $0$.
\end{prop}

\begin{proof}
Assume that $g$ acts non-isometrically, we shall prove that in this case $\psi$ has regular variation of index $0$. Since $\psi$ is decreasing, it suffices to show that $\lim_{t\to \infty} \frac{\psi(r t)}{\psi(t)}=1$ for some $r\neq 1$ by \cite[Proposition 2.15]{GU}. If $g$ acts non-isometrically, there is a point $x\in M$ and a unit vector $\xi_0\in T^*_xM$ such that $|(Dg)_x^T\xi_0|_{g(x)}\neq 1$. Set 
$$r:=|(Dg)_x^T\xi_0|_{g(x)}\neq 1.$$
Since $\psi$ is compatible with $g$, there is a constant $C\geq 0$ such that 
$$\left|\psi(|(Dg)^T_x\xi|)-\psi(|\xi|)\right|\leq C\psi(|\xi|)^2$$
and by setting $\xi=t\xi_0$, we arrive at the inequality 
$$\left|\psi(rt)-\psi(t)\right|\leq C\psi(t)^2.$$
After dividing by $\psi(t)$, taking the limit $t\to \infty$ and using that $\lim_{t\to \infty}\psi(t)=0$ we arrive at the desired equality $\lim_{t\to \infty} \frac{\psi(r t)}{\psi(t)}=1$.
\end{proof}

A prototypical example of a function with regular variation of index $0$ is $\psi(t):=\log(1+t)$. This function is compatible with any conformal local diffeomorphism and has been considered in the context of constructing spectral triples in \cite{DGMW,GM,GM2,GU,MHaluk}.

\begin{prop}
\label{psidiracprop}
Let $M$ be an $n$-dimensional Riemannian closed manifold, $\slashed{D}$ a Dirac operator on $S\to M$, $g:M\to M$ a surjective local diffeomorphism acting conformally and lifting to $S$ and $\psi:[0,\infty)\to (0,\infty)$ a decreasing function compatible with $g$, having smooth regular variation and satisfying $\lim_{t\to \infty}\psi(t)=0$ and $\psi(t)^{-1}=O(t^{1/n})$ as $t\to \infty$. Let $\A$ denote the $*$-algebra generated by $C^\infty(M)$ and the isometry $V_g$ from Equation \eqref{vgdef} (see page \pageref{vgdef}). Then $(\A,L^2(M,S), \slashed{D}_\psi)$ is a unital $\psi$-summable spectral triple.
\end{prop}

This result follows in the same manner as in \cite[Section 8]{DGMW} 
and \cite[Theorem 10.6]{GU}. Clearly, if $\psi(t)^{-1}=O(\log(t))$ 
as $t\to \infty$, then $(\A,L^2(M,S), \slashed{D}_\psi)$ is 
$\mathrm{Li}_1$-summable. The reader should note that 
the $K$-homology class $[(\A,L^2(M,S), \slashed{D}_\psi)]\in K^*(O_{E_g})$, 
where $O_{E_g}$ is the Cuntz-Pimsner algebra of the module
$E_g=C(M)V_g$, discussed later. The class $[(\A,L^2(M,S), \slashed{D}_\psi)]$ restricts 
to $[(C^\infty(M),L^2(M,S), \slashed{D})]\in K^*(C(M))$. 
Thus the class $[(C^\infty(M),L^2(M,S), \slashed{D})]\in K^*(C(M))$ 
obstructs $[(\A,L^2(M,S), \slashed{D}_\psi)]\in K^*(O_{E_g})$ 
being a Kasparov product with the Cuntz-Pimsner 
boundary extension in $KK^1(O_{E_g},C(M))$ -- we will 
return to study this boundary extension and its associated 
semifinite spectral triples below in Subsection \ref{cpalgexam}.

\subsubsection{Graph $C^*$-algebras}
\label{graphcstarsubsec}

A class of examples carrying interesting noncommutative geometries 
with a well-studied set of KMS-states is that of graph $C^*$-algebras. 
With a finite directed graph $G=(V,E)$, with edge set $E$ and vertex set $V$, 
one associates a $C^*$-algebra $C^*(G)$ \cite{Raeburn}. For simplicity we suppose that we have no sources nor sinks. The $C^*$-algebra 
$C^*(G)$ is generated by partial isometries $(S_e)_{e\in E}$ 
and projections $(p_v)_{v\in V}$ satisfying the relations 
$$
S^*_eS_e=p_{r(e)}, \quad \mbox{and}\quad p_v=\sum_{s(e)=v} S_eS_e^*,
$$
where $r(e)$ denotes the range of the vertex $e$ and $s(e)$ its 
source. The $C^*$-algebra $C^*(G)$ can be described as a 
Cuntz-Pimsner algebra in several ways, a class of $C^*$-algebras 
carrying noncommutative geometries that we will study in the 
next subsection. In this subsection, we focus on a construction 
of noncommutative geometries along an orbit in the infinite 
path space of $G$ -- a construction based in the model of 
$C^*(G)$ as a groupoid $C^*$-algebra. The associated 
noncommutative geometries come from \cite{GM}. As 
explained in \cite{GMR} the groupoid model can be seen 
as a Cuntz-Pimsner model of $C^*(G)$ using the 
one-sided infinite path space 
$$
\Omega_G
:=\{x=e_1e_2\cdots \in E^\N: \,s(e_j)=r(e_{j+1})\ \forall j\}.
$$ 

The path space $\Omega_G$ is a compact Hausdorff 
space in the subspace topology $\Omega_G\subseteq E^\N$. 
It carries a shift mapping $\sigma_G:\Omega_G\to \Omega_G$, 
$\sigma_G(e_1e_2e_3\cdots):=e_2e_3\cdots$. The shift 
mapping is a surjective local homeomorphism. We can 
define an \'etale groupoid $\mathcal{G}_G$ over $\Omega_G$ by 
$$
\mathcal{G}_G:=\big\{(x,n,y)\in \Omega_G\times \Z\times \Omega_G: 
\,\exists k\geq \max(0,-n) \  \mbox{such that}\  
\sigma_G^{n+k}(x)=\sigma_G^k(y)\big\}.
$$
The range mapping is defined by $r(x,n,y)=x$, 
the source mapping as $s(x,n,y):=y$ and the product by 
$$
(x,n,y)(y,m,z):=(x,n+m,z).
$$
The topology of $\mathcal{G}_G$ is uniquely determined 
by declaring the groupoid to be \'etale and the mappings 
$(x,n,y)\mapsto n$ and
$$
\kappa_G(x,n,y):=\min\{k\geq \max(0,-n): \sigma_G^{n+k}(x)=\sigma_G^k(y)\},
$$
to be continuous. 
There is an isomorphism $\pi_G:C^*(G)\to C^*(\mathcal{G}_G)$ 
determined by defining $\pi_G(S_e)$ to be the characteristic 
function of the set $\{(x,1,\sigma_G(x)): x\in C_e\}$ where 
$C_e:=\{e_1e_2\cdots \in \Omega_G: e_1=e\}$. 
See \cite{deaacaneudoaod}.

Define the function $\psi_0:\{(n,k)\in \Z\times \N: n+k\geq 0\}\to \Z$ by 
$$
\psi_0(n,k):=
\begin{cases} 
n, \; &k=0\\
-|n|-k,\; &k>0.
\end{cases}
$$
For a point $y\in \Omega_G$, we define the discrete set 
$$
\mathcal{V}_y:=d^{-1}(\{y\})=\{(x,n): (x,n,y)\in \mathcal{G}_G\}.
$$
The set $\mathcal{V}_y$ is the union of all forward orbits of all backward 
orbits of $y$ under $\sigma_G$ where we keep track of the 
lag $n$. Since $\mathcal{V}_y$ is a fibre of the domain 
mapping, point evaluation in $y$ induces a representation 
$\pi_y:C^*(G)\to \mathbb{B}(\ell^2(\mathcal{V}_y))$. 
If $G$ is primitive $C^*(G)$ is simple and $\pi_y$ is 
faithful. At the level of the generators, 
$$
\pi_y(S_e)\delta_{(x,n)}=
\begin{cases}
\delta_{(ex,n+1)},\; &r(x)=s(e),\\
0,\; &r(x)\neq s(e).\end{cases}
$$

\begin{prop}
\label{localizeckdirinpoint}
Define the operator $\D_y$ densely on 
$\ell^2(\mathcal{V}_y)$ as the self-adjoint operator with 
$$
\D_y\delta_{(x,n)}:=\psi_0(n,\kappa(x,n,y))\delta_{(x,n)}.
$$
The triple $(C_c(\mathcal{G}_G,\ell^2(\mathcal{V}_y),\D_y)$ 
is an $\mathrm{Li}_1$-summable spectral triple. Moreover, 
under the isomorphism $K^1(C^*(G))\cong \Z^E/(1-A_{\rm edge})\Z^E$, 
of odd $K$-homology with the cokernel of the edge adjacency matrix, 
the class $[(C_c(\mathcal{G}_G,\ell^2(\mathcal{V}_y),\D_y)]$ is 
mapped to the element $\delta_e\mod (1-A_{\rm edge})\Z^E$ 
where $\delta_e\in \Z^E$ denotes the basis element 
corresponding to $e\in E$.
\end{prop}

\begin{proof}
It is proven in \cite[Theorem 5.2.3]{GM} that 
$(C_c(\mathcal{G}_G,\ell^2(\mathcal{V}_y),\D_y)$ 
is a spectral triple whose $K$-homology class corresponds 
to the element $\delta_e\mod (1-A_{\rm edge})\Z^E$ under 
$K^1(C^*(G))\cong \Z^E/(1-A_{\rm edge})\Z^E$. It remains 
to prove that $(C_c(\mathcal{G}_G,\ell^2(\mathcal{V}_y),\D_y)$ 
is $\mathrm{Li}_1$-summable. We compute that 
$$
\Tr(\e^{-t|\D_y|})=\sum_{(x,n)\in \mathcal{V}_y} 
\e^{-t(|n|+\kappa_G(x,n,y))}
=\sum_{n\in \Z}\sum_{k=\max(0,-n)}^\infty \#\{(x,n): \kappa_G(x,n,y)=k\}
\e^{-t(|n|+k)}.
$$
However, if $(x,n)$ is such that $\kappa_G(x,n,y)=k$ 
the path $x$ is determined by $y$ except for its first 
$n+k$ steps so 
$\#\{(x,n): \kappa_G(x,n,y)=k\}\leq |E|^{n+k}\leq \e^{\log(|E|)(|n|+k)}$. 
We conclude that 
\[
\Tr(\e^{-t|\D_y|})<\infty \quad\mbox{if}\quad t>\log(|E|).
\qedhere
\]
\end{proof}

We return to this example below in 
Example \ref{diracgraphhea} (see page \pageref{diracgraphhea}) 
and Subsection \ref{diracgraphkms} (see page \pageref{diracgraphkms}).

\subsubsection{Cuntz-Pimsner algebras}
\label{cpalgexam}

In this subsection we consider the construction of semi-finite spectral triples on 
Cuntz-Pimsner algebras -- a broad class 
of examples which include both Cuntz-Krieger algebras and crossed products by $\Z$. 
Quite general techniques for constructing spectral triples for these algebras were developed in a series of papers (in rough chronological order) \cite{GM, RRS,GMR,GM2,RRSmore}.
We consider the set up of \cite{RRS} and \cite{GMR},
which provide a means of lifting data from the (unital) coefficient algebra of a bi-Hilbertian 
bimodule to its Cuntz-Pimsner algebra.

We start with a unital, separable $C^*$-algebra
$A$, and a finitely generated projective (fgp) bi-Hilbertian 
bimodule $E$ over $A$, i.e. a module 
fulfilling the conditions of the following definition.

\begin{defn} 
\label{cond:one}
An fgp bi-Hilbertian bimodule $E$ over $A$ is an $A$-bimodule equipped with the following structures:
\begin{itemize}
\item $E$ has both left and right $A$-valued inner products 
which induce equivalent norms on $E$. 
\item The left and right actions are both injective and adjointable.
\item $E$ is finitely generated and projective 
as both a left and right module. 
\end{itemize}
\end{defn}

To separate the left and the right structures, we write $E_A$ when we want to emphasize the 
right module structure and $(\cdot|\cdot)_A$ for the right inner product. Similarly, ${}_AE$ denotes 
the left module defined from $E$ and ${}_A(\cdot|\cdot)$ the left inner product. 

The algebraic Fock space $\Fock^{\rm alg}$ is the algebraic direct sum of the $A$-modules $E^{\otimes_A k}$. The Fock space $\Fock$ is defined as the right $A$-Hilbert $C^*$-module completion of $\Fock^{\rm alg}$. The Cuntz-Toeplitz algebra $\mathcal{T}_E\subseteq \End^*_A(\Fock)$ is the $C^*$-algebra generated by the creation operators $T_\mu \xi:=\mu\otimes \xi$ for $\mu\in \Fock^{\rm alg}$. The Cuntz-Pimsner algebra $O_E$ is defined from the short exact sequence 
$$0\to \mathbb{K}_A(\Fock)\to \mathcal{T}_E\to O_E\to 0.$$
We call this short exact sequence the defining extension of $O_E$.

A set $(e_j)_{j=1}^N\subset E$ of vectors is a frame 
for the right module
$E_A$ if for all $e\in E$ we have 
$e=\sum_je_j(e_j|e)_A$, and similarly for a left frame $(f_k)_{k=1}^N$.
Since $E$ is finitely generated and projective, there exists left and right frames and 
we can for simplicity assume that they have the same cardinality.
For $e$ and $f$ in the right Hilbert module $E_A$, we denote the associated rank-one 
operator by $\Theta_{e,f}:=e( f|\cdot)$. 
Then the frame condition can be expressed as
$$
\sum_{j=1}^N\Theta_{e_j,e_j}={\rm Id}_E
$$
and similarly for $f_\sigma$. The frame $(e_j)_{j=1}^N$ induces a frame for $E_A^{\ox k}$, namely
$(e_\rho)_{|\rho|=k}$ where $\rho$ is a multi-index and $e_\rho=e_{\rho_1}\ox\cdots\ox e_{\rho_k}$.

We define the right Watatani index of $E^{\otimes k}$ as the element of $A$ given by
\begin{equation}\label{dfn:Watatani}
\e^{\beta_k}=\sum_{|\rho|=k}{}_A(e_\rho|e_\rho)
=\sum_{|\rho'|=k-1}{}_A(e_{\rho'}\e^\beta|e_{\rho'}).
\end{equation}

The right Watatani index is positive, central and since the left action is injective, also invertible.
Therefore $\beta_k$ is a well defined self-adjoint central element in $A$. 
The key assumptions we make concern the asymptotic behaviour of the right Watatani indices.
In \cite[Section 3.2]{RRS}
we define an $A$-bilinear functional $\Phi_\infty:\O_E\to A$. 
This functional gives us an $A$-valued inner 
product on $\O_E$. The construction of $\Phi_\infty$ 
begins by defining 
\begin{equation}\label{Phi_k}
\Phi_k:\End_A^*(E^{\ox k})\to A,
\qquad \Phi_k(T)=\sum_{|\rho|=k}{}_A(Te_\rho|e_\rho).
\end{equation}

Here we use the notation $\End_A^*(E^{\ox k})$ for the $C^*$-algebra of 
$A$-linear adjointable operators on $E^{\ox k}$. It follows from \cite[Lemma 2.16]{KajPinWat}
that $\Phi_k$ does not depend on the choice of frame. We note that 
$\mathrm{e}^{\beta_k}=\Phi_k({\rm Id}_{E^{\ox k}})$. Since $\Phi_k$ is independent of the
choice of frame, so is $\mathrm{e}^{\beta_k}$. We extend the functional $\Phi_k$
to a mapping $\End^*_A(\Fock)\to A$ by compressing along the orthogonally complemented submodule 
$E^{\ox k}\subseteq \Fock$. 
To obtain a good ``limiting functional"  $\Phi_\infty(T):=\lim_{k\to\infty}\Phi_k(T)\mathrm{e}^{-\beta_k}$ on the Cuntz-Toeplitz algebra, we impose the following condition on the 
Watatani indices.

\begin{defn}
\label{ass:one}
Let $E$ be an fgp bi-Hilbertian bimodule over the unital $C^*$-algebra $A$.
We say that $E$ is W-regular if for every $k\in \N$ and
$\nu\in E^{\otimes k}$ there exists a $\tilde{\nu}\in E^{\otimes k}$ satisfying
$$
\Vert \mathrm{e}^{-\beta_n}\nu \mathrm{e}^{\beta_{n-k}}-\tilde\nu\Vert_{E^{\otimes k}} \to 0 \quad\mbox{as $n\to \infty$}.
$$
\end{defn}

In \cite{RRS} the reader can find several examples of Cuntz-Pimsner algebras for 
which a stronger version of W-regularity as defined in Definition \ref{ass:one} holds. There are no known examples of modules that are not W-regular.
When $E$ is W-regular, \cite[Proposition 3.5]{RRS} guarantees that $\lim_{k\to\infty}\Phi_k(T)\mathrm{e}^{-\beta_k}$ is well defined for $T$ from the $*$-algebra generated by the set of 
creation operators $\{T_\nu: \;\nu\in \algFock\}$ and is continuous in the $C^*$-norm. 
In Section \ref{sub:no-left} we shall see that there are ways around W-regularity, and even the existence of an $A$-valued left inner product,  when constructing 
semifinite spectral triples giving rise to KMS-states.

We thus obtain a unital positive $A$-bilinear functional $\Phi_\infty: \T_E\to A$.
The functional $\Phi_\infty$ annihilates the compact endomorphisms, and descends to a well-defined
functional on the Cuntz-Pimsner algebra $\O_E$. By an abuse of notation, we also 
denote this functional by $\Phi_\infty:\O_E\to A$. Since $\Phi_k$ and $\mathrm{e}^{\beta_k}$ 
do not depend on the choice of frame, neither does $\Phi_\infty$. We define the inner product
$$
(S_1|S_2)_A:=\Phi_\infty(S_1^*S_2),\qquad S_1,\,S_2\in \O_E.
$$
When computing these inner products, the following fact is useful.

\begin{lemma}
\label{computeinnerprod} 
Let $E$ be a W-regular fgp bi-Hilbertian bimodule.
For homogeneous elements $\mu,\,\nu\in \algFock$ we have
\begin{equation}\label{eq:mu-nu-beta}\Phi_\infty(S_\mu S_\nu^*)
= \lim_{k\to\infty}{}_A(\mu|\mathrm{e}^{-\beta_k}\nu \mathrm{e}^{\beta_{k-|\nu|}})={}_A(\mu|\tilde{\nu}).
\end{equation}
In particular, if $S\in \O_{E}$ is homogeneous of degree $n\neq 0$, then $\Phi_{\infty}(S)=0$.
\end{lemma}

Completing $\O_E$ modulo the  vectors of zero length (with respect to $\Phi_\infty$) 
yields a right $A$-Hilbert $C^*$-module 
that we denote by $\phimod$. The module $\phimod$ carries a left action of $\O_E$ given by 
extending the multiplication action of $\O_E$ on itself.
By considering the linear span of the image of the generators $S_\nu$, $\nu\in \algFock$, inside
the module $\phimod$, we obtain an isometrically embedded and complemented  copy of the Fock space.
We let $Q$ be the projection on this copy of the Fock space.

\begin{thm}[Proposition 3.14 of \cite{RRS}]
Let $E$ be a W-regular fgp bi-Hilbertian bimodule 
over a unital $C^*$-algebra $A$.
The tuple $(\O_E,\phimod,2Q-1)$ is an odd Kasparov module representing
the class of the defining extension 
\begin{align}
\label{defext}
0\to \Ko_A(F_E)\to \mathcal{T}_E\to \O_E\to 0.
\end{align}
\end{thm}

To construct an unbounded representative of $(\O_E,\phimod,2Q-1)$, we 
will add an additional assumption
regarding the fine structure of the operation $\nu\mapsto \tilde{\nu}$ in 
the definition of W-regularity (see Definition \ref{ass:one}). Assuming W-regularity, 
we can define the 
operator $\topop_k:E^{\otimes k}\to E^{\otimes k}$ by
$$
\topop_k\nu:=\tilde{\nu}=\lim_{n\to\infty}\mathrm{e}^{-\beta_n}\nu \mathrm{e}^{\beta_{n-k}}.
$$

\begin{defn}
\label{ass:two}
Let $E$ be an fpg bi-Hilbertian bimodule over the unital $C^*$-algebra $A$.
We say that $E$ is strictly W-regular if it is W-regular and for any $k$, we can write $\topop_k=c_kP_k=P_{k}c_{k}$ where 
$P_{k}\in \End^{*}_{A}(E^{\otimes k})$ is a (necessarily $A$-bilinear) projection and $c_k$ 
is given by left-multiplication by an element in $A$. 
\end{defn}

\begin{rmk}
\label{justforpage}
As with W-regularity, the reader can in \cite{RRS} find several examples of Cuntz-Pimsner algebras for 
which strict W-regularity holds. There are no known examples of modules that are not strictly 
W-regular.
\end{rmk}

\begin{rmk} 
\label{centralityandasstwo}
If there is a decomposition $\topop_{k}=c_{k}P_{k}$ 
as in the definition of strict W-regularity, \cite[Lemma 3.8]{GMR} shows that
it is unique and of a very specific form. 
Indeed, each $c_k$ is central, invertible and $c_k=\Phi_k(P_k)^{-1}$. 
Strict W-regularity is readily verified in practice using \cite[Lemma 3.8]{GMR}. 
For instance, if $\beta_1$ is 
central for the module action on $E$, 
$c_k=\mathrm{e}^{-\beta_k}=\mathrm{e}^{-k\beta_1}$ 
is central for the module action on $E$ 
and $P_k=1_{E^{\ox k}}$. 
\end{rmk}

When $E$ is strictly W-regular, an unbounded self-adjoint regular operator 
$\D_\psi$ on $\phimod$ is constructed in \cite{GMR} making  
$(\O_E,\phimod,\D_\psi)$ into an unbounded Kasparov module representing the 
$KK$-class of the defining extension \eqref{defext}. 
The operator $\D_\psi$ is of the form 
$$
\D_\psi=\sum_{n\in\Z}\sum_{r\geq\max\{0,n\}}\psi(n,r)P_{n,r}
$$
where
$\psi:\,\Z\times\N\to[0,\infty)$ is a function with certain Lipschitz properties (see \cite[Remark 3.20]{GMR}), and the $P_{n,r}$ are projections on finitely generated projective subspaces,
\cite{GMR}. 
While the particular choice of function $\psi$ does not matter much, we will take the function 
$$
\psi(n,r)=\left\{\begin{array}{ll} n & n=r\\ -(2r-n) & \mbox{otherwise}\end{array}\right..
$$
The projections $P_{n,r}$ form a sequence of mutually orthogonal projections satisfying that 
the direct sum $\oplus_{r=\max(0,n)}^{r_0} P_{n,r}$ is the projection onto the $A$-linear span of 
$\{S_\mu S_\nu^*: |\mu|-|\nu|=n, \max(0,n)\leq r\leq r_0\}$. 
In particular, $P_{n,n}$ is the projection onto $E^{\otimes n}$ for $n\in \N$. 
More precisely, the projections are defined by 
\begin{equation}
\label{onrqnr}
P_{n,r}:=
\begin{cases}
Q_{n,r}-Q_{n,r-1},\; & r>\max\{0,n\}\\ 
Q_{n,r},\; & r=\max\{0,n\}
\end{cases}
\end{equation}
where the projections $Q_{n,r}$ are defined in terms of the right frame $(e_j)_{j=1}^N$ and the left frame $(f_j)_{j=1}^N$ as
\begin{equation}
\label{qnrqnr}
Q_{n,r}:=\sum_{|\rho|-|\sigma|=n,\,|\rho|=r}
\Theta_{W_{e_\rho,c^{-1/2}_{|\sigma|}P_Ff_\sigma},W_{e_\rho,c^{-1/2}_{|\sigma|}P_Ff_\sigma}}
\end{equation}
where $P_F=\oplus P_k$ is the projection on the Fock module coming from Definition \ref{ass:two}.
Here we have written $W_{\xi,\eta}\in \phimod$ for the element 
defined from $S_\xi S_\eta^*\in O_E$ where $\xi\in E^{\otimes r}$ 
and $\eta\in E^{\otimes k}$.
For details, see \cite[Lemma 3.10 and Proposition 3.11]{GMR}.

To obtain a semifinite spectral triple, we localize $(\O_E,\phimod,\D_\psi)$ in 
a positive trace on $A$. Following Proposition \ref{localintra}, we consider the semifinite spectral triple
$$
(\O_E,L^2(\phimod,\tau),\D_\psi,\cN_\tau(\phimod),\Tr_\tau).
$$
Here $L^2(\phimod,\tau):=\Xi_A\otimes_A L^2(A,\tau)$ and $\Tr_\tau$ is the dual trace on $\cN_\tau(\phimod):=(\End^*_A(\phimod)\otimes 1)''$
which satisfies $\Tr_\tau(\Theta_{e,f})=\tau((f|e)_A)$, \cite{LN,Tak}. 

\begin{lemma}
\label{ximodsemi}
Assume that the fgp bi-Hilbertian bimodule $E$ is strictly W-regular 
and that $\tau$ is a positive trace on $A$.
Then the semifinite spectral triple
$$
(\O_E,L^2(\phimod,\tau),\D_\psi,\cN_\tau(\phimod),\Tr_\tau)
$$
is $\mathrm{Li}_1$-summable.
\end{lemma}

\begin{rmk}
The assumptions on the existence of limiting behaviour for 
the Watatani indices are
really just for convenience here. These assumptions 
relate to existence and behaviour of norm limits, but 
we have
passed to the `measurable setting' and so really only need weak limits. We will explore this
point of view in Subsection \ref{sub:no-left}.
\end{rmk}

\begin{proof}
We need to prove that the following expression is finite for $t$ large enough:
$$
\Tr_\tau(\e^{-t|\D_\psi|})
=\sum_{n\in\Z}\sum_{r\geq\max\{0,n\}}\e^{-t|\psi(n,r)|}\Tr_\tau(P_{n,r}).
$$

By definition (see \eqref{onrqnr}), $P_{n,r}=Q_{n,r}-Q_{n,r-1}$ when $r>\max(0,n)$ and $P_{n,r}=Q_{n,r}$ when $r=\max(0,n)$ and $Q_{n,r}$ is defined as in \eqref{qnrqnr}. Using the computations of \cite[Lemma 2.8]{GMR}, we see that 
\begin{align}
\label{tracompqnr}
\Tr_\tau(Q_{n,r})&=\sum_{|\rho|-|\sigma|=n,\,|\rho|=r}\tau\left(( W_{e_\rho,c^{-1/2}_{|\sigma|}P_Ff_\sigma}|W_{e_\rho,c^{-1/2}_{|\sigma|}P_Ff_\sigma})_{A}\right)\\
\nonumber
&=\sum_{|\rho|-|\sigma|=n,\,|\rho|=r}\tau\circ \Phi_\infty\left(S_{c^{-1/2}_{|\sigma|}P_Ff_\sigma} S_{e_\rho}^*S_{e_\rho}S_{c^{-1/2}_{|\sigma|}P_Ff_\sigma} ^*\right)\\
\nonumber
&=\sum_{|\rho|-|\sigma|=n,\,|\rho|=r}\tau\left({}_A(P_Ff_\sigma|P_Ff_\sigma (e_\rho|e_\rho)_A)\right).
\end{align}
Using the fact that the elements of the frame have norm bounded by $1$, we see that 
\begin{align*}
|\Tr_\tau(Q_{n,r})|&\leq \sum_{|\rho|-|\sigma|=n,\,|\rho|=r}|\tau\left({}_A(P_Ff_\sigma|P_Ff_\sigma (e_\rho|e_\rho)_A)\right)|\\
&\leq  \sum_{|\rho|-|\sigma|=n,\,|\rho|=r}\|({}_A(P_Ff_\sigma|P_Ff_\sigma (e_\rho|e_\rho)_A)\|_A\leq  N^{2r-n}\leq \e^{\log(N)|\psi(n,r)|},
\end{align*}
where $N$ is the number of elements in the left frame and the right frame. 
We can now estimate 
\begin{align*}
\left|\Tr_\tau(\e^{-t|\D_\psi|})\right|
=&\left|\sum_{n\in\Z}\sum_{r>\max\{0,n\}}\e^{-t|\psi(n,r)|}\Tr_\tau(Q_{n,r}-Q_{n,r-1})\right|\\
&+\left|\sum_{n\in\Z}\sum_{r=\max\{0,n\}}\e^{-t|\psi(n,r)|}\Tr_\tau(Q_{n,r})\right| \\
\leq& 2\sum_{n\in\Z}\sum_{r\geq\max\{0,n\}}\e^{-t|\psi(n,r)|}|\Tr_\tau(Q_{n,r})| \\
\leq& 2\sum_{n\in\Z}\sum_{r\geq\max\{0,n\}}\e^{-(t-\log(N))|\psi(n,r)|}<\infty,
\end{align*}
if $t>\log(N)$.
\end{proof}

We return to Cuntz-Pimsner algebras below in Example \ref{diraccphea} (see page \pageref{diraccphea}) and Section \ref{diraccpkms} (see page \pageref{diraccpkms}).
Let us discuss a special case of Cuntz-Pimsner algebras arising on a commutative coefficient algebra.

\begin{example}
\label{localhomeofirst}
Let $Y$ be a compact Hausdorff space and $g:Y\to Y$ a surjective local homeomorphism. We consider the module $E_g=C(Y)$ with the bimodule action 
$$
(afb)(x)=a(x)f(x)b(g(x)), \quad a,b\in C(Y),\ f\in E_g.
$$
This is an fgp bi-Hilbertian bimodule in the inner products
$$
{}_{C(Y)}(f_1,f_2)=f_1\overline{f_2}\quad\mbox{and}\quad (f_1|f_2)_{C(Y)}:=\mathfrak{L}_g(\overline{f_1}f_2),
$$
where $\mathfrak{L}_g:C(Y)\to C(Y)$ is the transfer operator 
\begin{equation}\label{transferoperator}
\mathfrak{L}_g(f)(x):=\sum_{y\in g^{-1}(x)}f(y).
\end{equation}
 The Cuntz-Pimsner algebra $O_{E_g}$ can be realized as a groupoid $C^*$-algebra as in \cite{deaacaneudoaod} (see also \cite[Theorem 3.2]{DGMW}) over the solenoid 
$$X=\{x=y_1y_2\cdots \in Y^\N: \,g(y_{k+1})=y_k\ \forall k\}.$$ 
The case that $X$ equipped with the shift mapping is a Smale space was studied in \cite{DGMW}. 

The module $E_g$ has a right frame $(e_j)_{j=1}^N$ where $e_j=\sqrt{\chi_j}$ 
for a partition of unity $(\chi_j)_{j=1}^N$ subordinate to an open 
covering $(U_j)_{j=1}^N$ of $Y$ such that $g|_{U_j}$ is injective for all $j$. Using this partition of unity, one sees that $\beta_k=0$ for all $k$. It follows that $E_g$ is a strictly W-regular module. 

The case that $g:M\to M$ was a surjective local diffeomorphism acting conformally was considered in Subsection \ref{diracmfdfirst} (see page \pageref{diracmfdfirst}). However, the spectral triple considered Proposition \ref{psidiracprop} on $O_{E_g}$ differs greatly from the semifinite spectral triples considered in Lemma \ref{ximodsemi} -- the latter are in the image of the boundary mapping in $KK_1(O_{E_g},C(M))$ defined from Equation \eqref{defext} while the former is not if $[\slashed{D}]\neq 0\in K^*(C(M))$. 

Another class of examples already considered arises from a finite graph $G$ as in Subsection \ref{graphcstarsubsec} (see page \pageref{graphcstarsubsec}) where the shift mapping $\sigma_G:\Omega_G\to \Omega_G$ is a surjective local homeomorphism and $C^*(G)\cong O_{E_{\sigma_G}}$. The spectral triples in Proposition \ref{localizeckdirinpoint} (see page \pageref{localizeckdirinpoint}) arises from the construction of Lemma \ref{ximodsemi} by taking the trace $\tau:C(\Omega_G)\to \C$ to be defined from point evaluation in $y$. 
\end{example}

\subsubsection{Group $C^*$-algebras}
\label{groupcstarexam}

We now turn our attention to examples coming from the reduced group $C^*$-algebra 
of a discrete group. A well known construction associates a spectral triple with a length function 
on the group, we consider this example and a semifinite modification thereof which is possible 
for a-T-menable groups, i.e. groups with the Haagerup property. The methods extend to a-TT-menable groups, a class of groups containing all hyperbolic groups, see more in \cite[Chapter 7.2]{mimurathesis}.

Let $\Gamma$ denote a countable discrete group. Recall that a length function $\ell:\Gamma\to \R_{\geq 0}$ is a function satisfying $\ell(e)=0$ for $e\in \Gamma$ the identity, and $\ell(\gamma \gamma')\leq \ell(\gamma)+\ell(\gamma')$ for all group elements $\gamma,\gamma'\in\Gamma$. 
We say that $\ell$ is a proper 
length function if $\ell$ is a proper function, i.e. the set $\{\gamma\in \Gamma: \ell(\gamma)\leq R\}$ is finite for any $R\geq 0$. If there exists a constant $\beta\geq 0$ such that $\#\{\gamma\in \Gamma: \ell(\gamma)\leq R\}=O(\e^{\beta R})$ as $R\to \infty$ we say that $(\Gamma,\ell)$ has at most exponential growth. 

Define the operator $\D_\ell$ densely on $\ell^2(\Gamma)$ as the self-adjoint operator with 
$$\D_\ell\delta_{\gamma}:=\ell(\gamma)\delta_{\gamma}.$$
The space of compactly supported functions $c_c(\Gamma)$ is a core for $\D_\ell$. We define the $*$-algebra $c_b(\Gamma)\rtimes ^{\rm alg}\Gamma$ as the $*$-algebra generated by multiplication operators (by bounded functions on $\Gamma$) $c_b(\Gamma)\subseteq \mathbb{B}(\ell^2(\Gamma))$ and all left translation operators. 

\begin{prop}
\label{spectripfromlength}
Let $\ell$ be a proper length function on $\Gamma$. The triple $(c_b(\Gamma)\rtimes ^{\rm alg}\Gamma,\ell^2(\Gamma),\D_\ell)$ is a spectral triple defining the trivial class in the $K$-homology of the $C^*$-algebra $c_b(\Gamma)\rtimes_r \Gamma$. Moreover, if  $(\Gamma,\ell)$ has at most exponential growth the spectral triple $(c_b(\Gamma)\rtimes ^{\rm alg}\Gamma,\ell^2(\Gamma),\D_\ell)$ is $\mathrm{Li}_1$-summable.
\end{prop}

\begin{proof}
Since $\ell$ is proper, it is clear that $\D_\ell$ has compact resolvent and if $(\Gamma,\ell)$ has at most exponential growth, then there is $C>0$ such that
$$
\Tr(\e^{-t|\D_\ell|})=\sum_{\gamma\in \Gamma}\e^{-t\ell(\gamma)}=\sum_{n=0}^\infty \#\{\gamma\in \Gamma: \ell(\gamma)=n\}\e^{-tn}\leq C\sum_{n=0}^\infty \e^{-(t-\beta)n}=\frac{C}{1-\e^{\beta-t}}<\infty.
$$

To show that $(c_b(\Gamma)\rtimes ^{\rm alg}\Gamma,\ell^2(\Gamma),\D_\ell)$ is a spectral triple, it remains to show that $c_b(\Gamma)\rtimes ^{\rm alg}\Gamma$ preserves the domain of $\D_\ell$ and has bounded commutators with $\D_\ell$. Domain preservation is clear. For an element $a\lambda_\gamma\in c_b(\Gamma)\rtimes ^{\rm alg}\Gamma$ and a function $f\in c_c(\Gamma)$ we compute that 
$$[\D_\ell,a\lambda_\gamma]f(g)=a(g)(\ell(g)-\ell(\gamma^{-1}g))f(\gamma^{-1}g).$$
It follows that 
\[
\|[\D_\ell,a\lambda_\gamma]\|_{\mathbb{B}(\ell^2(\Gamma)}\leq \|a\|_{c_b(\Gamma)}\sup_{g\in \Gamma}|\ell(g)-\ell(\gamma^{-1}g)|\leq \|a\|_{c_b(\Gamma)}\ell(\gamma).\qedhere
\]
\end{proof}

The $K$-homology class of $(c_b(\Gamma)\rtimes ^{\rm alg}\Gamma,\ell^2(\Gamma),\D_\ell)$ 
is trivial. We shall now consider a topologically more interesting semifinite spectral triple that 
can be constructed on groups with the Haagerup property. We are grateful to Branimir \'{C}a\'{c}i\'{c} for sharing this construction with us. Similar ideas appeared in \cite[Appendix B]{juuunge}.

Let $\Gamma$ be a discrete group with the Haagerup property.
Then there is a proper isometric action of $\Gamma$ on a real Hilbert space $\H_\Gamma$. 
By the Mazur-Ulam theorem there exists an orthogonal representation
$$
\pi_\Gamma:\,\Gamma\to O(\H_\Gamma)
$$
on the Hilbert space $\H_\Gamma$ and a proper cocycle $c_\Gamma$ for $\pi_\Gamma$, meaning that $c_\Gamma:\Gamma\to \H$ is a proper function satisfying the cocycle identity
\begin{equation}
c_\Gamma(\gamma_1\gamma_2)=c_\Gamma(\gamma_1)-\pi_\Gamma(\gamma_1)c_\Gamma(\gamma_2).
\label{eq:cocycle}
\end{equation}
The cocycle identity allows us to define a length function on $\Gamma$ by
$$
\ell(\gamma):=\Vert c_\Gamma(\gamma)\Vert_{\H_\Gamma}.
$$
Since $c_\Gamma$ is proper, so is $\ell$.

\begin{rmk}
The existence of a proper isometric action of a group $\Gamma$ on a Hilbert space is equivalent to $\Gamma$ having the Haagerup property, also known as a-T-menability. Our construction extends to the case when there exists an orthogonal representation $\pi_\Gamma:\,\Gamma\to O(\H_\Gamma)$ and a proper quasi-cocycle $c_\Gamma:\Gamma\to \H_\Gamma$. That is, when 
$$Q(c_\Gamma):=\sup_{\gamma_1,\gamma_2}\|c_\Gamma(\gamma_1\gamma_2)-c_\Gamma(\gamma_1)+\pi_\Gamma(\gamma_1)c_\Gamma(\gamma_2)\|_{\H_\Gamma}<\infty.$$
In this case, $\ell(\gamma):=\Vert c_\Gamma(\gamma)\Vert_{\H_\Gamma}$ could fail to be a length function but still satisfies $\ell(\gamma \gamma')\leq \ell(\gamma)+\ell(\gamma')+Q(c_\Gamma)$ which suffices for our purposes. The existence of a proper quasi-cocycle on a Hilbert space is equivalent to $\Gamma$ being a-TT-menable. Hyperbolic groups are a-TT-menable. For notational simplicity, we restrict our attention to cocycles. 
\end{rmk}

\begin{defn}
Let $\Cl(\H_\Gamma)$ denote the the complex Clifford algebra of $\H_\Gamma$ 
and assume that $\mathfrak{c}_S:\Cl(\H)\to \mathbb{B}(S_\H)$ is a representation of $\Cl(\H)$.
We say that a unitary representation 
$\pi_S:\,\Gamma\to U(S_\H)$ is a lift of $\pi_\Gamma$ to $S_\H$
if for all $v\in\H$ and $g\in\Gamma$ we have
\begin{equation}
\pi_S(g)\mathfrak{c}_S(v)\pi_S(g^{-1})=\mathfrak{c}_S(\pi_\Gamma(g)v).
\label{eq:S-cl}
\end{equation}
\end{defn}

When $\H$ is finite dimensional this is just the well-known Clifford algebra, but when $\H$ is infinite dimensional we refer to \cite{CO,W} for a description of this algebra.

For a representation $S_\H$ of $\Cl(\H)$ we consider the new Hilbert space $\ell^2(\Gamma,S_\H)$. Assuming that $\pi_S$ lifts $\pi_\Gamma$ to $S_\H$ the Hilbert space $\ell^2(\Gamma,S_\H)$ carries a representation of $\Gamma$ defined by
\begin{equation}
\label{tildepidef}
\tilde{\pi}:\,\Gamma\to U(\ell^2(\Gamma,S_\H)),\qquad
\big(\tilde{\pi}(g)f\big)(\gamma)=\pi_S(g)f(g^{-1}\gamma).
\end{equation}
On the Hilbert space $\ell^2(\Gamma, S_\H)$
define a self-adjoint operator $\D_c$ by declaring
\begin{equation}
(\D_cf)(\gamma):=\mathfrak{c}_S(c_\Gamma\gamma))f(\gamma),\quad f\in c_c(\Gamma, S_\H).
\label{eq:cliff}
\end{equation}
Since $\mathfrak{c}_S(v)^2=\|v\|_\H^2$ for all $v\in \H$,  the domain for $\D_c$ can be deduced from 
$$
(\D_c^2f)(\gamma)=\ell(\gamma)^2f(\gamma).
$$
The compatibility requirement Equation \eqref{eq:S-cl}
and cocycle property Equation \eqref{eq:cocycle}
imply that for $g,\gamma\in\Gamma$ we have
$$
\pi_S(g)\mathfrak{c}_S(c_\Gamma(\gamma))=\mathfrak{c}_S(\pi(g)c_\Gamma(\gamma))\pi_S(g)=\mathfrak{c}_S(c_\Gamma(g))\pi_S(g)+\mathfrak{c}_S(c_\Gamma(g\gamma))\pi_S(g).
$$
Then the commutator of $\D_c$ and a group element is
\begin{align*}
([\D_c,\tilde{\pi}(g)]f)(\gamma)&=
\mathfrak{c}_S(c_\Gamma(\gamma))\pi_S(g)f(g^{-1}\gamma)-\pi_S(g)\mathfrak{c}_S(c_\Gamma(g^{-1}\gamma))f(g^{-1}\gamma)\\
&=\mathfrak{c}_S(c_\Gamma(g))\pi_S(g)f(g^{-1}\gamma)=(\mathfrak{c}_S(c_\Gamma(g))\tilde{\pi}(g))(f)(\gamma).
\end{align*}
Hence the commutators between $\D_c$ and group elements are bounded. It is moreover clear that
these commutators lie in $\cN_0\rtimes\Gamma$ where $\cN_0=\mathfrak{c}_S(\Cl(\H))''$. We define $\cN$ as the von Neumann algebraic tensor product $\mathbb{B}(\ell^2(\Gamma))\bar{\otimes}\cN_0$. 

Finally, 
$$
(1+\D_c^2)^{-1}\in \mathbb{K}(\ell^2(\Gamma))\ox 1\subset \mathbb{K}(\ell^2(\Gamma))\ox\mathbb{K}_\tau
$$
where $\mathbb{K}_\tau$ is the compacts in $\cN_0$ for a choice of normalized positive trace $\tau$. 
Let $\Tr_\tau$ be the trace on $\cN$ defined from the trace $\tau$ on $\cN_0$. 
We conclude that $(1+\D_c^2)^{-1}\in \mathbb{K}_{\Tr_\tau}$.
Finally, if $\ell$ has at most exponential growth then $\Tr_\tau(\e^{-t|\D_c|})=\sum_{\gamma\in \Gamma}\e^{-t\ell(\gamma)}<\infty$ for $t$ large enough. As such, $(i\pm \D_c)^{-1}\in \mathrm{Li}_1$ 
if $\ell$ has at most exponential growth. We conclude the following result.

\begin{prop}
\label{cssfst}
Assume that $\mathfrak{c}_S:\Cl(\H)\to \mathbb{B}(S_\H)$ is a representation of $\Cl(\H)$
and that the unitary representation $\pi_S:\,\Gamma\to U(S_\H)$ lifts $\pi_\Gamma$ to $S_\H$. 
Let $c_b(\Gamma)$ be the (continuous) bounded functions on $\Gamma$, and define a representation of $c_b(\Gamma)\rtimes \Gamma$ on $\ell^2(\Gamma, S_\H)$ by 
$$\hat{\pi}_S(a\lambda_g)f(\gamma)=a(\gamma) [\tilde{\pi}(g)f](\gamma),$$
where $\tilde{\pi}$ is as in \eqref{tildepidef}. Then the triple $(c_b(\Gamma)\rtimes^{\rm alg} \Gamma, \ell^2(\Gamma, S_\H), \D_c, \cN, \Tr_\tau)$ is a semifinite spectral triple which is $\mathrm{Li}_1$-summable if $\ell$ has at most exponential growth.
\end{prop}

We return to the example of this subsubsection in Example \ref{diracgrouphea} (see page \pageref{diracgrouphea}) and Subsection \ref{diracgroupkms} (see page \pageref{diracgroupkms}).

\begin{example}
The following example of a proper group cocycle shows the construction's geometric advantage compared to only using a length function. Consider the trivial action of the discrete group $\Gamma=\Z^n$ on $\H_\Gamma=\R^n$. The inclusion $\Z^n\hookrightarrow\R^n$ is additive and proper, and therefore a proper group cocycle for the trivial action. The semifinite spectral triple associated with a finite dimensional Clifford representation $\mathfrak{c}_S:\R^n\to \End_\C(S)$ can when restricted to $C^*(\Z^n)\cong C(\mathbb{T}^n)$ be identified with the semifinite spectral triple $(C^\infty(\mathbb{T}^n),L^2(\mathbb{T}^n,S),\slashed{D}_{\mathbb{T}^n}\otimes 1_S,\mathbb{B}(L^2(\mathbb{T}^n))\otimes \C\ell_n,\Tr_\tau)$ using Fourier theory on the dual torus $\mathbb{T}^n=\widehat{\Z^n}$. We observe that to extract $K$-homological content from this construction we need to specify a grading if $n$ is even. In particular, it is unclear how to interpret the construction above in $K$-homology when $\H_\Gamma$ is infinite-dimensional.
\end{example}

\section{KMS states constructed from $\mathrm{Li}_1$-summable spectral triples}
\label{sec:KMS}

This section contains the fundamental technical construction of the paper. 
Starting from a semifinite $\mathrm{Li}_1$-summable spectral triple, 
we use the associated algebra of Toeplitz operators to construct an action 
from the operator $\D$ and a KMS-state from the operator $|\D|$.

\subsection{The positive part of the spectrum and heat traces}

Throughout this section we suppose that $(\A,\H,\D,\cN,\Tau)$ is 
a unital semifinite
spectral triple. The spectral triple can be semifinite, 
in which case we let $\Tau$ denote the given positive faithful normal semifinite trace. 
In general $\Tau$ is not unique, and coincides with a non-zero multiple of 
the operator trace in the ``usual" non-semifinite case $\cN=\B(\H)$.
We write $\mathbb{K}_\cN$ for the compacts for $\Tau$. Again, 
$\Ko_\cN$ coincides with the usual compacts in the case of the type I factor $\cN=\B(\H)$.

We write $\cN^+$ for the 
von Neumann algebra $P_\D\cN P_\D$. By an abuse of notation, 
we write $\Tau$ also for the induced faithful normal semifinite trace on $\cN^+$. 
The $\Tau$-compacts on $\cN^+$ will be denoted by $\Ko_\cN^+$.

\begin{defn}
\label{ass:minus-one}
The operator $\D$ is said to have positive $\Tau$-essential spectrum if for some $\beta\in [0,\infty)$, we have $\Tau(P_\D\e^{-t\D})<\infty$ for $t>\beta$ 
and $\Tau(P_\D\e^{-t\D})\nearrow \infty$ as $t\searrow\beta$. 

\end{defn}

\begin{prop}
\label{converatbeta}
Let $\D$ be a self-adjoint densely defined operator affiliated with $\cN$ satisfying that $(1+\D^2)^{-1/2}\in \mathrm{Li}_1(\Tau)$. We define the number
\begin{equation}
\label{invtempish}
\beta_\D:=\inf\{t>0: \Tau(P_\D\e^{-t\D})<\infty\}.
\end{equation}
Then $\beta_\D\in [0,\infty)$ and $\D$ has positive $\Tau$-essential spectrum if and only if
$$
\lim_{t\searrow\beta_\D}\Tau(P_\D\e^{-t\D})=\infty.
$$
In particular, if $\beta_\D=0$ then $\D$ has positive $\Tau$-essential spectrum if and only if $\Tau(P_\D)=\infty$.
\end{prop}

\begin{proof}
By definition, if $(1+\D^2)^{-1/2}\in \mathrm{Li}_1$ then $\Tau(\e^{-t|\D|})<\infty$ for $t$ large enough. Therefore, $\Tau(P_\D\e^{-t\D})=\Tau(P_\D\e^{-t|\D|})<\infty$ for $t$ large enough and $\beta_\D:=\inf\{t>0: \Tau(P_\D\e^{-t\D})<\infty\}$ will be a number in $[0,\infty)$. By definition, $\Tau(P_\D\e^{-t\D})<\infty$ for $t>\beta_\D$ and if $
\lim_{t\searrow\beta_\D}\Tau(P_\D\e^{-t\D})=\infty$ then $\D$ has positive $\Tau$-essential spectrum with $\beta=\beta_\D$. Conversely, if $\D$ has positive $\Tau$-essential spectrum there is a $\beta\in [0,\infty)$ with $\Tau(P_\D\e^{-t\D})<\infty$ for $t>\beta$ 
and $\Tau(P_\D\e^{-t\D})\nearrow \infty$ as $t\searrow\beta$, and in this case it is clear that $\beta=\beta_\D$.
\end{proof}

\begin{rmk}
We will often impose the assumption of positive $\Tau$-essential spectrum. 
If for some $\beta\in [0,\infty)$, $\Tau((1-P_\D)\e^{t\D})<\infty$ for $t>\beta$ and 
$\Tau((1-P_\D)\e^{t\D})\nearrow \infty$ as $t\searrow\beta$, 
we can equally well use $-\D$ in our construction. 
\end{rmk}

\begin{prop}
Let $\D$ be a self-adjoint densely defined operator affiliated with $\cN$ satisfying that $(1+\D^2)^{-1/2}\in \mathrm{Li}_1(\Tau)$. Then $\D$ has positive $\Tau$-essential spectrum if and only if both of the following conditions fail:
\begin{enumerate}
\item $P_\D$ has finite $\Tau$-trace.
\item There exists a $p>0$ such that $P_\D\e^{-|\D|}\in \mathcal{L}^p(\Tau)\setminus \cap_{q>p}  \mathcal{L}^q(\Tau)$.
\end{enumerate}
\end{prop}

The reader should note that conditions 1. and 2. are mutually exclusive.

\begin{proof}
If $\D$ has positive $\Tau$-essential spectrum, then clearly 1. fails. Also 2. fails if $\D$ has positive $\Tau$-essential spectrum because condition 2. is equivalent to $\beta_\D=p$ and $\lim_{t\searrow p}\Tau(P_\D\e^{-t\D})$ being finite. 

Conversely, if Condition 1. and 2. fails, then either $\beta_\D=0$ and $\lim_{t\searrow 0}\Tau(P_\D\e^{-t\D})$ must be infinite not to violate $\Tau(P_\D)$ being infinite or $\beta_\D>0$ and the set $\{p>0: P_\D\e^{-|\D|}\in \mathcal{L}^p(\Tau)\}$ is open (due to condition 2. failing) showing that $\Tau(P_\D\e^{-t\D})\nearrow \infty$ as $t\searrow\beta_\D$. 
\end{proof}

\begin{example}
\label{diracmfdhea}
Dirac operators on closed manifolds, as considered in Subsection \ref{diracmfdfirst} (see page \pageref{diracmfdfirst}), have positive essential spectrum. In this case, we can compute $\beta_\D=0$ and the leading term in the heat trace asymptotics using Proposition \ref{heatasumfd}. If $(C^\infty(M),L^2(M,S),\slashed{D})$ is the spectral triple associated with a Dirac operator, Proposition \ref{heatasumfd} implies that
$$
\Tr_{L^2(M,S)}(P_\D \e^{-t|\slashed{D}|})
=n!c_nt^{-n}\int_{S^*M} \mathrm{Tr}_S(p_{\slashed{D}})\mathrm{d}V 
+O(t^{-n+\epsilon}), \quad\mbox{as $t\to 0$},
$$
where $p_{\slashed{D}}$ is the principal symbol of the 
zeroth order pseudo-differential operator $P_\slashed{D}$. 
A direct computation shows that 
$p_{\slashed{D}}(x,\xi)=\frac{1}{2}\left(c_S(\xi)+1\right)$ 
for $x\in M$ and $\xi\in S^*_xM$. Here $c_S:T^*M\to \End(S)$ 
denotes Clifford multiplication. More generally, 
Proposition \ref{heatasumfd} allows us to conclude that for $a\in C^\infty(M)$,
\begin{align*}
\Tr_{L^2(M,S)}(P_\D a\e^{-t|\slashed{D}|})
&=n!c_nt^{-n}\int_{S^*M} \mathrm{Tr}_S(p_{\slashed{D}})a\,\mathrm{d}V 
+O(t^{-n+\epsilon})=\\
&=n!c_nt^{-n}\int_{M}\int_{S^*_xM} a(x)\mathrm{Tr}_S(p_{\slashed{D}}(x,\xi))\,\mathrm{d}V_{S^*_xM}(\xi)\mathrm{d}V(x) 
+O(t^{-n+\epsilon})=\\
&=\tilde{c}_nt^{-n}\int_{M}a\,\mathrm{d}V 
+O(t^{-n+\epsilon}), \quad\mbox{as $t\to 0$},
\end{align*}
for a new constant $\tilde{c}_n>0$, depending 
only on the dimension of $M$ and the rank of $S$. In the last equality we 
used that $p_{\slashed{D}}-1/2$ is an antisymmetric 
function under the involution $(x,\xi)\mapsto (x,-\xi)$ of $S^*M$ 
and therefore 
\begin{align*}
\int_{S^*_xM}\mathrm{Tr}_S(p_{\slashed{D}}(x,\xi))\,\mathrm{d}V_{S^*_xM}&=\int_{S^*_xM}\mathrm{Tr}_S(p_{\slashed{D}}(x,-\xi))\,\mathrm{d}V_{S^*_xM}=\\
&=\frac{\mathrm{rank}(S)}{2}\int_{S^*_xM} \,\mathrm{d}V_{S^*_xM}=\frac{\mathrm{rank}(S)\pi^{\dim(M)/2}}{\Gamma(\dim(M)/2)}.
\end{align*}
\end{example}

\begin{example}
\label{diracgraphhea}
The spectral triples for graph $C^*$-algebras 
from Proposition \ref{localizeckdirinpoint} (see page \pageref{localizeckdirinpoint})
also have positive essential spectrum. 
We compute $\beta_\D$ and the heat trace asymptotics 
assuming that $G$ is primitive. 
In this example, $P_{\D_y}$ is the projection onto the subspace 
$\ell^2(\mathcal{V}_y\cap \kappa_G^{-1}(0))$. We use the notation 
\begin{equation}
\label{vplusdef}
\mathcal{V}_y^+:=\mathcal{V}_y\cap \kappa_G^{-1}(0)
=\{(x,n)\in \mathcal{V}_y: n\geq 0, \; \sigma_G^n(x)=y\}.
\end{equation}
The space
$P_{\D_y}\ell^2(\mathcal{V}_y)$ is therefore spanned by the 
orthonormal basis $(\delta_{(x,n)})_{(x,n)\in \mathcal{V}_y^+}$. 
Note that if $\sigma_G^n(x)=y$ then $x$ is uniquely determined 
by $y$ and a finite path $\sigma=\sigma_1\sigma_2\cdots \sigma_n$ 
with $x=\sigma y$. Note that paths of the form 
$x=\sigma y$, with $s(\sigma_n)=r(y)$, 
exhaust all possible $x\in \sigma_G^{-n}(\{y\})$. 
Using that $P_{\D_y} \D_y\delta_{(x,n)}=n \delta_{(x,n)}$ 
for $(x,n)\in \mathcal{V}_y^+$, we compute that 
$$
\Tr(P_{\D_y}\e^{-t|\D_y|})
=\sum_{n=0}^\infty \sum_{x\in \sigma_G^{-n}(\{y\})} \e^{-tn}
=\sum_{n=0}^\infty \#\{\sigma\in E_n: s(\sigma_n)=r(y)\} \e^{-tn}.
$$
Let $A$ denote the edge adjacency matrix of $G$ and 
$r_\sigma(A)$ its spectral radius. If $G$ is primitive, (i.e. all entries of $A^k$ are positive for some integer $k>0$)
we let $w\in \C^E$ its $\ell^2$-normalized Perron-Frobenius vector. 
It follows from \cite[Lemma 3.7]{RRS} that there is an 
$\alpha_0\in [0,1)$ such that 
$$
\#\{\sigma\in E_n: s(\sigma_n)=r(y)\}
=\|w\|_{\ell^1}w_{r(y)} r_\sigma(A)^{n}+O((\alpha_0 r_\sigma(A))^n),
$$
as $n\to \infty$. We can conclude that there is a function 
$f$ holomorphic in $\mathrm{Re}(t)>\log r_\sigma(A)+\log\alpha_0$ such that 
$$
\Tr(P_{\D_y}\e^{-t|\D_y|})=\frac{\|w\|_{\ell^1}w_{r(y)}}{1-r_\sigma (A) \e^{-t}}+f(t).
$$
Therefore, 
$\Tr(P_{\D_y}\e^{-t|\D_y|})-\frac{\|w\|_{\ell^1}w_{r(y)}}{t-\log(r_\sigma (A))}$ 
has a holomorphic extension to 
$\mathrm{Re}(t)>\log r_\sigma(A)+\log\alpha_0$  whenever $G$ is primitive. 

More generally, if $G$ is primitive, the method above 
shows that for two finite paths $\mu$ and $\nu$ we can compute that 
\begin{align*}
\Tr(P_{\D_y}S_\mu S_\nu^*\e^{-t|\D_y|})
=&\delta_{\mu,\nu}\sum_{n=0}^\infty \sum_{x\in \sigma_G^{-n}(\{y\})} \|S_\mu^*\delta_{(x,n)}\|_{\ell^2}^2\e^{-tn}\\
=&\sum_{n=|\mu|}^\infty \#\big\{\sigma\in E_{n-|\mu|}: r(\sigma_1)=s(\mu), \, s(\sigma_{n-|\mu|})=r(y)\big\} \e^{-tn}\\
&+\delta_{\mu,\nu}\sum_{n=0}^{|\mu|-1} \sum_{x\in \sigma_G^{-n}(\{y\})} \|S_\mu^*\delta_{(x,n)}\|_{\ell^2}^2\e^{-tn}\\
=&\delta_{\mu,\nu}w_{s(\mu)}w_{r(y)} \frac{\e^{-t|\mu|}}{1-r_\sigma(A)\e^{-t}}+\delta_{\mu,\nu}f_{\mu,y}(t),
\end{align*}
for a function $f_{\mu,y}$ holomorphic in 
$\mathrm{Re}(t)>\log r_\sigma(A)+\log\alpha_0$. We conclude that 
$$
\Tr(P_{\D_y}S_\mu S_\nu^*\e^{-t|\D_y|})-\delta_{\mu,\nu}w_{d(\mu)}w_{r(y)}\frac{r_\sigma(A)^{-|\mu|}}{t-\log(r_\sigma (A))},
$$ 
has a holomorphic extension to 
$\mathrm{Re}(t)>\log r_\sigma(A)+\log\alpha_0$  
whenever $G$ is primitive. As such, 
$\beta_{\D_y}=\log(r_\sigma(A))$ and $\D_y$ has positive essential spectrum.
\end{example}

\begin{example}
\label{diraccphea}
Let $O_E$ be a Cuntz-Pimsner algebra  defined from a strictly W-regular (recall Definition \ref{ass:two}) finitely generated and projective bi-Hilbertian bimodule $E_A$ and a positive trace $\tau$ on the coefficient algebra $A$. The semifinite spectral triple considered in Lemma \ref{ximodsemi} (see page \pageref{ximodsemi}) also has positive $\Tr_\tau$-essential spectrum assuming a criticality condition on $\tau$ that we formulate below (see Definition \ref{defn:critical} on page \pageref{defn:critical}). The heat trace asymptotics are slightly more involved, and we compute these explicitly in Subsection \ref{subsec:T-vs-LN} under a condition on $\tau$ previously studied by Laca-Neshveyev \cite{LN} in the context of KMS-states. However, for a general $\tau$ we can proceed as in the proof of Lemma \ref{ximodsemi} to deduce the following.

\begin{prop}
\label{posheatcomp}
For any strictly $W$-regular fgp bi-Hilbertian bimodule $E$ over the unital $C^*$-algebra and a positive trace $\tau$ on $A$, 
$$\Tr_\tau(P_{\D_\psi}\e^{-t\D_\psi})=\sum_{n=0}^\infty \e^{-tn} \tau_*(E^{\otimes_A n}),$$
where $\tau_*:K_0(A)\to \R$ denotes 
the map induced by $\tau$ on $K$-theory. In particular, $\Tr_\tau(P_{\D_\psi}\e^{-t\D_\psi})$ does not depend on the choice of inner products on $E$ but only on $\tau$ and the bimodule structure on $E$.
\end{prop}

\begin{proof}
We compute that 
$$
\Tr_\tau(P_\D \e^{-t|\D_\psi|})=\sum_{n=0}^\infty \e^{-t|\psi(n,n)|}\Tr_\tau(P_{n,n})=\sum_{n=0}^\infty \e^{-tn}\Tr_\tau(Q_{n,n})=\sum_{n=0}^\infty \sum_{|\rho|=n}\e^{-tn}\tau((e_\rho|e_\rho)_A).
$$
On the other hand, $\tau_*(E^{\otimes_A n})=(\tau\otimes \Tr_{M_{N(n)}})(p_{E^{\otimes_A n}})$ where $p_{E^{\otimes_A n}}\in M_{N(n)}(A)$ is a projection representing $E^{\otimes_A n}$. Using the choice of frame $(e_j)_{j=1}^N$, we can take $N(n):=N^n$ and $p_{E^{\otimes_A n}}:=((e_\mu|e_\nu)_A)_{|\mu|=|\nu|=n}$. In this choice of representing projection, 
\[
\tau_*(E^{\otimes_A n})=(\tau\otimes \Tr_{M_{N(n)}})(p_{E^{\otimes_A n}})=\Tr_{M_{N(n)}}((\tau((e_\mu|e_\nu)_A)_{|\mu|=|\nu|=n}))=\sum_{|\rho|=n}\tau((e_\rho|e_\rho)_A).\qedhere
\]
\end{proof}

This computation shows that it is in general difficult to compute $\beta_\D$. In this case $\beta_\D$ 
depends on the asymptotic properties of the sequence $(\tau_*(E^{\otimes_A n}))_{n\in \N}$ as $n\to \infty$.

For a simple tensor $\sigma\in E^{\ox m}$ write $\sigma=\underline{\sigma}\overline{\sigma}$, where the initial segment
$\underline{\sigma}$ will be of a length understood from context ($|\underline{\sigma}|=|\mu|$ in the next computation).
With this  notation, we can compute our functional on a typical $S_\mu S_\nu^*\in O_E$, 
where $\mu\in E^{\otimes k}$, $\nu\in E^{\otimes l}$ are simple tensors. We find 
\begin{align}
\label{tracpcocomcme}
\Tr_\tau(P_\D S_\mu S_\nu^* \e^{-t|\D_\psi|})&=\delta_{|\mu|,|\nu|}\sum_{n=0}^\infty \e^{-tn}\Tr_\tau(S_\mu S_\nu^* Q_{n,n})\\
\nonumber
&=\delta_{|\mu|,|\nu|}\sum_{n=|\mu|}^\infty\sum_{|\sigma|=n} \e^{-tn}\tau ((S_\mu^*e_{\sigma}|S_\nu^*e_{\sigma})_{E^{\otimes (n-|\mu|)}})\\
\nonumber
&=\delta_{|\mu|,|\nu|}\sum_{n=|\mu|}^\infty\sum_{|\sigma|=n} \e^{-tn}\tau (\,(\,(\mu|e_{\underline{\sigma}})_{E^{|\mu|}}e_{\overline{\sigma}}\,|\,(\nu|e_{\underline{\sigma}})_{E^{|\mu|}}e_{\overline{\sigma}})_{E^{\otimes (n-|\mu|)}}).
\end{align}
\end{example}

\begin{example}
\label{diracgrouphea}
Consider a length function $\ell$ on a countable group $\Gamma$ as in Subsection \ref{groupcstarexam}.

\begin{defn}
\label{criticaldefn}
We define the critical value of $(\Gamma,\ell)$ as
$$\beta(\Gamma, \ell):=\inf\{t\geq 0:\sum_{\gamma\in \Gamma} \e^{-t\ell(\gamma)}<\infty \}.$$
If $\sum_{\gamma\in \Gamma} \e^{-t\ell(\gamma)}\nearrow \infty$ as $t\searrow \beta(\Gamma,\ell)$, we say that $\ell$ is critical. 
\end{defn}

It follows directly from Definition~\ref{ass:minus-one} that the operator $\D_\ell$ appearing in Proposition \ref{spectripfromlength} (see page \pageref{spectripfromlength}) has positive essential spectrum as long as $\ell$ is critical. The heat trace of an element $a\lambda_g\in c_b(\Gamma)\rtimes^{\rm alg}\Gamma$ is given by 
$$\Tr(P_{\D_\ell} a\lambda_g \e^{-t|\D_\ell|})=\Tr(a\lambda_g \e^{-t|\D_\ell|})=\delta_{e,g}\sum_{\gamma\in \Gamma} a(\gamma)\e^{-t\ell(\gamma)}.$$

Similar computations can be carried out for the semifinite spectral triple constructed in Proposition \ref{cssfst} using a Hilbert space valued cocycle $c_\Gamma$ (see page \pageref{cssfst}). Note that 
$$
P_{\D_c}f(g)=\frac{1}{2}\left(\frac{\mathfrak{c}_S(c_\Gamma(g))}{\|c_\Gamma(g)\|_{\H_\Gamma}}+1\right).
$$ 
Therefore, the heat trace of an element 
$a\lambda_g\in c_b(\Gamma)\rtimes^{\rm alg}\Gamma$ is given by 
\begin{align*}
\Tr_\tau(P_{\D_c} a\lambda_g \e^{-t|\D_c|})
&=\frac{1}{2}\sum_{\gamma\in \Gamma}\left\langle \delta_\gamma, \tau\left(\frac{\mathfrak{c}_S(c_\Gamma(g))}{\|c_\Gamma(g)\|_{\H_\Gamma}}+1\right)a(g^{-1}\gamma)\delta_{g\gamma}\right\rangle \e^{-t\ell(\gamma)}=\\
&=\frac{1}{2}\delta_{e,g}\sum_{\gamma\in \Gamma} a(\gamma)\e^{-t\ell(\gamma)}=\frac{1}{2}\Tr(P_{\D_\ell} a\lambda_g \e^{-t|\D_\ell|}).
\end{align*}
Here we use that $\tau(\mathfrak{c}_S(v))=0$ for any $v\in \H_\Gamma$ which holds due to the fact that we can pick a $w\in \H_\Gamma$ orthogonal to $v$ and compute that 
$$\tau(\mathfrak{c}_S(v))=\tau(\mathfrak{c}_S(w)\mathfrak{c}_S(v)\mathfrak{c}_S(w))=-\tau(\mathfrak{c}_S(w)^2\mathfrak{c}_S(v))=-\tau(\mathfrak{c}_S(v)).$$
 If the length function $\ell(\gamma):=\Vert c_\Gamma(\gamma)\Vert_{\H_\Gamma}$ associated with $c_\Gamma$ is critical, we say that $c_\Gamma$ is critical.
We conclude that the semifinite spectral triple from Proposition \ref{cssfst} has positive essential spectrum if $c_\Gamma$ is critical.
\end{example}

\subsection{To-plitz or not To-plitz}
\label{subsec:toplitz}

We proceed under the same working conditions as in the previous section to construct states from spectral triples. Recall that $\cN^+=P_\D\cN P_\D$
and that $\Ko_\cN^+=P_\D\Ko_\cN P_\D$.

\begin{defn}
\label{defn:to-plitz}
Let $(\A,\H,\D,\cN,\Tau)$ be a semifinite spectral triple. 
We define the Toeplitz algebra of $(\A,\H,\D,\cN,\Tau)$ as 
$$
T_A:=P_\D A P_\D+\Ko_\cN^+\subseteq \cN^+.
$$
The saturated Toeplitz algebra of $(\A,\H,\D,\cN,\Tau)$ is defined by 
$$
T_{A,\D}:=C^*\left(\bigcup_{s\in\R}\e^{is\D}T_A\e^{-is\D}\right)=C^*\left(\bigcup_{s\in\R}\e^{is\D}P_\D AP_\D\e^{-is\D}+\Ko_\cN^+\right)\subseteq \cN^+.
$$
\end{defn}

\begin{prop}
\label{toepiscstar}
The Toeplitz algebra $T_A$ of a unital semifinite spectral triple $(\A,\H,\D,\cN,\Tau)$ is a $C^*$-algebra.
\end{prop}

\begin{proof}
It follows from Lemma \ref{bddcommsum} that $[P_\D,a]\in \Ko_\cN$ for all $a\in A$. Therefore, the mapping 
$$
\beta_\D:A\to \cN^+/\Ko_\cN^+, \quad a\mapsto P_\D aP_\D\mod\Ko_\cN^+,
$$
is a $*$-homomorphism and $\beta_\D(A)\subseteq \cN^+/\Ko_\cN^+$ is a closed $C^*$-subalgebra. By definition, $T_A$ is the preimage of $\beta_\D(A)$ under the quotient mapping $\cN^+\to \cN^+/\Ko_\cN^+$ and is therefore a $C^*$-algebra.
\end{proof}

The reader should note that Proposition \ref{toepiscstar} also holds in the non-unital setting because it only relies on Lemma \ref{bddcommsum} which holds non-unitally. 

\begin{prop}
The saturated Toeplitz algebra of $(\A,\H,\D,\cN,\Tau)$ carries an $\R$-action $\sigma^+:\R\to \mathrm{Aut}(T_{A,\D})$ defined by 
$$
\sigma_s^+(T):=P_\D \e^{is\D}T\e^{-is\D}P_\D=P_\D \e^{is|\D|}T\e^{-is|\D|}P_\D,\qquad s\in \R, \; T\in T_{A,\D}.
$$
\end{prop}

This proposition is a consequence of that $T_{A,\D}$ is constructed as the saturation of $T_A$ under the action $\sigma^+$ extended to $\cN^+$. We note that if we identify $T_{A,\D}$ with a subalgebra of $\cN$, we can also write $\sigma_s^+(T):= \e^{is\D}T\e^{-is\D}= \e^{is|\D|}T\e^{-is|\D|}$.

Define a one-parameter family of states $(\phi_{t,0})_{t>\beta}$ on $T_{A,\D}$ by
$$
\phi_{t,0}(T):=\frac{\Tau(P_\D T\e^{-t\D})}{\Tau(P_\D\e^{-t\D})}.
$$
We shall compose the family $(\phi_{t,0})_{t>\beta}$ with an ``extended limit" as $t\to \beta$:

\begin{defn}
\label{extendedlimitdef}
An extended limit as $t\to \beta$ is a state $\omega\in L^\infty(\beta,\infty)^*$ such that $\omega(f)=0$ whenever $\lim_{t\to \beta}f(t)=0$. For an extended limit $\omega$ and $f\in L^\infty(\beta,\infty)$, we write 
$$\lim_{t\to \omega} f:=\omega(f).$$

\end{defn}

Let $\omega$ be any extended limit and define
$$
\phi_{\omega,0}:\,T_{A,\D}\to\C,\qquad \phi_{\omega,0}(T):=\omega\circ \phi_{t,0}(T)=\lim_{t\to \omega} \phi_{t,0}(T).
$$

\begin{lemma}
\label{vanishoncomp}
Let $(\A,\H,\D,\cN,\Tau)$ be a semifinite $\mathrm{Li}_1$-summable spectral triple with positive $\Tau$-essential spectrum (see Definition \ref{ass:minus-one} on page \pageref{ass:minus-one}).
For any extended limit $\omega$, the functional $\phi_{\omega,0}$ is a state on $T_{A,\D}$. 
Moreover $\phi_{\omega,0}(T)=0$ for all $T\in \mathbb{K}_\cN^+$. 
\end{lemma}
\begin{proof}
It is immediate that $\phi_{\omega,0}$ is a state. 
For the statement that $\phi_{\omega,0}(T)=0$ for all $T\in \mathbb{K}_\cN^+$, 
we observe that  since $\phi_{\omega,0}$ is a state, it is also norm-continuous. 
It therefore suffices to prove that $\phi_{\omega,0}(T)=0$ for all projections $T\in\cN^+$ with 
$\Tau(T)<\infty$. For such $T$, we can estimate that 
$$
\phi_{t,0}(T)=\frac{\Tau(P_\D T\e^{-t\D})}{\Tau(P_\D\e^{-t\D})}\leq \frac{\Tau(T)}{\Tau(P_\D\e^{-t\D})}.
$$
Since, $\Tau(P_\D\e^{-t\D})\nearrow \infty$ as $t\searrow \beta$, it follows that $\lim_{t\to \beta} \phi_{t,0}(T)=0$. We conclude that for all projections $T\in\cN^+$, with 
$\Tau(T)<\infty$, and any extended limit $\omega$ as $t\to \beta$, $\omega\circ \phi_{t,0}(T)=0$.
\end{proof}

Due to Lemma \ref{vanishoncomp}, we can make the following definition.

\begin{defn}
Define the $C^*$-algebra $A_\D:=T_{A,\D}/\Ko^+_\cN$ and the state $\phi_\omega$ on $A_\D$ as 
$$\phi_\omega(T\!\!\!\mod \Ko^+_\cN):=\phi_{\omega,0}(T).$$
\end{defn}

The state $\phi_{\omega,0}$ also restricts to a state on $T_A$,
and Lemma \ref{vanishoncomp} implies that $\phi_{\omega,0}|_{T_A}$ descends 
to a state on $A$ via the $*$-epimorphism $\beta_\D:A\to T_A/\Ko^+_\cN$ 
(see the proof of Proposition \ref{toepiscstar} on page \pageref{toepiscstar}). 
To analyse the situation of the state on $A_\D$ versus that on $A$,
we consider the following ideal 
$$
I:=\{a\in A:\,P_\D aP_\D\in \mathbb{K}_\cN^+\}
$$
so that $P_\D IP_\D=P_\D AP_\D\cap\mathbb{K}_\cN^+$.
Since $T_A\subseteq T_{A,\D}$, we obtain a commuting diagram
$$
\xymatrix{0\ar[r]& \mathbb{K}_\cN^+\ar[r]\ar[d]^{=}&T_A\ar[r]\ar@{^{(}->}[d]&A/I\ar@{^{(}->}[d]\ar[r]&0\\
0\ar[r]&\mathbb{K}_\cN^+\ar[r]&T_{A,\D}\ar[r]&A_\D\ar[r]&0,
}
$$
with exact rows. The mapping $A/I\to A_\D$ is indeed injective by the four lemma. We identify $A/I$ with a subalgebra of $A_\D$. The induced mapping $\gamma: A\to A_\D$ is compatible with the states $\phi_\omega$ and $\phi_{\omega,0}$ in the sense that $\phi_{\omega,0}(P_\D aP_\D)=\phi_\omega(\gamma(a))$ for $a\in A$.

The $\R$-action $\sigma_s^+(T):=\e^{is\D}T\e^{-is\D}$ on $T_{A,\D}$ induces an $\R$-action on $A/I$. The following proposition follows from the construction of $T_{A,\D}$ as the saturation of $T_A$ under $\sigma^+$.

\begin{prop}\label{prop:3.13}
Let $\beta\in \R$. The algebra $A_\D$ carries an $\R$-action $\sigma:\R\to \mathrm{Aut}(A_\D)$ defined by declaring the quotient mapping $T_{A,\D}\to A_\D$ to be equivariant. The $C^*$-algebra $A_\D$ is the saturation of $A/I$ under the action $\sigma$, i.e. $A_\D$ is generated in $\cN^+/\Ko_\cN^+$ by $\cup_{s\in \R}\sigma_s(A/I)$.
\end{prop}

The  aim of our construction is to obtain a KMS state
on $A$, or, failing that, on $A/I$. As a first step we introduce conditions
ensuring that we at least get a KMS state on $T_{A,\D}$, and so on $A_\D$. 
To this end we make the following assumption

\begin{defn}
\label{ass:zero}
Let $(\A,\H,\D,\cN,\Tau)$ be a semifinite spectral triple. We say that a subset $S\subset\A$ is analytically generating at $\beta$ if the set $P_\D SP_\D+\mathbb{K}^+$ generates the Toeplitz algebra $T_A$ as a $C^*$-algebra and satisfies that $\e^{\beta\D}P_\D SP_\D\e^{-\beta\D}\subset \cN^+$.

We say that the semifinite spectral triple $(\A,\H,\D,\cN,\Tau)$ is $\beta$-analytic if it admits an analytically generating set at $\beta$.
\end{defn}

We note that this condition is empty if $\beta=0$. The condition of being $\beta$-analytic is just requiring that we have enough analytic elements in $T_{A,\D}$ to verify the KMS condition. Indeed we have the following.

\begin{prop}
\label{equivalencesbeta}
Let $(\A,\H,\D,\cN,\Tau)$ be a semifinite spectral triple and $\beta\in \R$. Consider the following statements:
\begin{enumerate}
\item[i)] The semifinite spectral triple $(\A,\H,\D,\cN,\Tau)$ is $\beta$-analytic.
\item[ii)] There is a dense $\sigma^+$-invariant subspace $T^0_{A,\D}\subseteq T_{A,\D}$ of elements satisfying that for any $T\in T^0_{A,\D}$ the function $f_T:\R\to \cN$, $f_T(t):=\sigma_t^+(T)$ has a bounded holomorphic extension to the strip $\{z\in \C: \; \mathrm{Im}(z)\in (-\beta,0)\}$. 
\item[iii)] There is a dense $\sigma$-invariant  subspace $A^0_{\D}\subseteq A_\D$ of elements satisfying that for any $a\in A_\D$ the function $f_a:\R\to \cN$, $f_a(t):=\sigma_t(a)$ has a bounded holomorphic extension to the strip $\{z\in \C: \; \mathrm{Im}(z)\in (-\beta,0)\}$. 
\end{enumerate}
It holds that i) implies ii) which implies iii). If $I:=\{a\in A:\,P_\D aP_\D\in \mathbb{K}_\cN^+\}=0$, iii) implies i).
\end{prop}

\begin{proof}
It is clear that i) implies ii) since for an analytically generating set $S$ at $\beta$ we can take $T^0_{A,\D}$ to be the $\sigma^+$-invariant $*$-algebra generated by $P_\D SP_\D\cup\mathcal{F}_\D$, where  $\mathcal{F}_\D\subseteq \mathbb{K}^+_\cN$ is the dense two sided ideal in $\cN$ generated by the spectral projections of $\D$ over compact intervals in $\R$. The implication ii)$\Rightarrow$iii) is seen from taking $A^0_\D:=T^0_{A,\D}/ (T^0_{A,\D}\cap \Ko_\cN^+)$. 

If $I=0$, the set $S=A^0_\D\cap A$ is dense in $A$, and for every $s\in S$
we find that $\e^{\beta |\D|}s\e^{-\beta|\D|}$ is a bounded operator in $\cN$. 
Thus $\e^{\beta |\D|}P_\D SP_\D \e^{-\beta|\D|}\subset\cN^+$ and $P_\D S P_\D+\Ko_\cN^+$ plainly generates $T_A$. We conclude that $S$ is analytically generating at $\beta$ and that iii) implies i) if $I=0$.
\end{proof}

\begin{prop}
\label{prop:no-work}
Let $(\A,\H,\D,\cN,\Tau)$ be a $\beta_\D$-analytic $\mathrm{Li}_1$-summable spectral triple with positive $\Tau$-essential spectrum. For any extended limit $\omega$ as $t\to \beta_\D$, the state $\phi_{\omega,0}$ is a 
$\beta_\D$-KMS state on $T_{A,\D}$ for the one-parameter group $\sigma^+$.
\end{prop}

\begin{proof}
It follows from Proposition \ref{equivalencesbeta} that the dense subalgebra $T^0_{A,\D}\subseteq T_{A,\D}$ consists of $\beta_\D$-analytic elements of $T_{A,\D}$. The twisted trace property relative to the one-parameter group $\sigma^+$ holds on $T^0_{A,\D}$ by direct computation: for all $T_1,\,T_2\in T^0_{A,\D}$
$$
\phi_\omega(T_1T_2)=\lim_{t\to \omega}\frac{\Tau(T_1T_2\e^{-t|\D|})}{\Tau(P_\D\e^{-t|\D|})}=\lim_{t\to \omega}\frac{\Tau(\e^{t|\D|}T_2\e^{-t|\D|}T_1\e^{-t|\D|})}{\Tau(P_\D\e^{-t|\D|})}=\phi_\omega(\sigma_{-i\beta}(T_2)T_1).
$$

By definition, $\phi_{\omega,0}$ is a $\beta_\D$-KMS-state for $\sigma^+$.
\end{proof}

\begin{corl}
\label{cor:phi-omega} 
Let $(\A,\H,\D,\cN,\Tau)$ be a $\beta_\D$-analytic $\mathrm{Li}_1$-summable spectral triple with positive $\Tau$-essential spectrum. For any extended limit $\omega$ as $t\to \beta_\D$, the state $\phi_{\omega}$ is a 
$\beta_\D$-KMS state on $A_\D$ for the one-parameter group $\sigma$, defined in Proposition~\ref{prop:3.13}.
\end{corl}

\begin{proof}
This follows from Proposition \ref{equivalencesbeta} and Proposition \ref{prop:no-work} because $\phi_{\omega,0}$ vanishes on the compacts and the fact that $\phi_\omega$ is induced from $\phi_{\omega,0}$.
\end{proof}

In practice, for a $\beta_\D$-analytic 
$\mathrm{Li}_1$-summable spectral triple with positive 
$\Tau$-essential spectrum, we will want to check that in fact
$T_{A,\D}=T_A$. In this case $A_\D=A/I$ and $\phi_\omega$ induces a KMS-state on $A/I$. 
In the often occuring special case $P_\D AP_\D\cap\mathbb{K}_\cN^+=0$, we 
obtain a KMS state on the algebra $A$. In practice, these things are all checkable and we will do 
so in several examples in the subsequent sections.

The special case $\beta=0$ has been addressed by Voiculescu, \cite[Proposition 4.6]{Voics} under the assumption that $\D$ is positive. Exponential $\beta$-compatibility is superfluous when $\beta_\D=0$. By Proposition \ref{converatbeta}, when $\beta_\D=0$, positive $\Tau$-essential spectrum is equivalent to $\Tau(P_\D)=\infty$.

\begin{thm}
\label{thm:voics}
Let $(\A,\H,\D,\cN,\Tau)$ be a unital $\mathrm{Li}_{1}$-summable semifinite spectral triple 
with $\beta_\D=0$ and $\Tau(P_\D)=\infty$.
Then $A$ has a tracial state. Indeed, for any extended limit $\omega$ as $t\to 0$, 
$$
\phi_\omega(a)
:=\lim_{t\to \omega}\frac{\Tau(P_\D a\e^{-t\D})}{\Tau(P_\D\e^{-t\D})}
$$
is a tracial state.
\end{thm}

Theorem \ref{thm:voics} applies to unital $\mathrm{Li}_{(0),1}$-summable semifinite spectral triples with $\Tau(P_\D)=\infty$ since $\mathrm{Li}_{(0),1}$-summability implies $\beta_\D=0$.

\begin{rmk}
If the heat trace has an asymptotic expansion as in Theorem \ref{heatvszeta}, then Theorem \ref{thm:voics} can be further simplified. Assume that there is a $p>0$ and that for any $a\in \mathcal{A}$, there is a $\phi_0(a)\in \C$ such that $\phi_0(1)\neq 0$ and 
$$\Tau(P_\D a\e^{-t|\D|})=\phi_0(a)t^{-p}+O(t^{-p+\epsilon}), \quad\mbox{as $t\to 0$},$$
for some $\epsilon>0$ (which can depend on $a$). Since $\phi_0(1)\neq 0$, it follows that $\phi_0$ is continuous in the $C^*$-norm on $\A$ and $\phi_\omega(a)=\frac{\phi_0(a)}{\phi_0(1)}$ for all $a\in \mathcal{A}$. In fact, $\phi_0(a)=\Gamma(p)\mathrm{Res}_{z=p} \zeta(z; P_\D a, |\D|)$.
\end{rmk}

The construction of the state $\phi_\omega$ in Corollary \ref{cor:phi-omega} involves the operator $P_\D$. We shall now provide a result allowing us to remove $P_\D$ from the definition of $\phi_\omega$ in the presence of certain symmetries on the spectral triple. The result provides a checkable set of conditions to compute $\phi_\omega$ by means of asymptotics of $\e^{-t|\D|}$.

\begin{lemma}
\label{gammajlem}
Let $(\A,\H,\D,\cN,\Tau)$ be a unital $\mathrm{Li}_{1}$-summable semifinite spectral triple and $\beta\geq 0$ a number such that $\Tau(\e^{-t|D|})<\infty$ for $t>\beta$ and $\Tau(\e^{-t|\D|})\nearrow \infty$ as $t\searrow\beta$. Assume that there exists self-adjoint operators $\gamma_1,\ldots, \gamma_N\in \cN\cap \A'$ such that 
\begin{enumerate}
\item $\sum_{j=1}^N \gamma_j^2$ is invertible;
\item $\gamma_j\Dom(\D)\subseteq \Dom(\D)$ and $[\D,\gamma_j]_+:=\D\gamma_j+\gamma_j \D$ has a bounded extension to $\H$;
\item For $j=1,\ldots, N$ and some $\epsilon>0$, the function $t\mapsto \e^{-t|\D|}\gamma_j\e^{t|\D|}$ extends to a norm continuous function from the interval $[\beta, \beta+\epsilon)$ to $\cN$ with 
$$\lim_{t\to \beta}\e^{-t|\D|}\gamma_j\e^{t|\D|}=\gamma_j.$$
\end{enumerate}
Then $\beta_\D=\beta$ and $\D$ has positive $\Tau$-essential spectrum. Moreover, for any extended limit $\omega$ as $t\to\beta_\D$, 
$$\phi_\omega(a)= \lim_{t\to \omega} \frac{\Tau(a\e^{-t|\D|})}{\Tau(\e^{-t|\D|})}.$$
\end{lemma}

\begin{example}
Before proceeding with the proof of the lemma, we give some examples of how the operators $\gamma_1,\ldots, \gamma_N\in \cN\cap \A'$ can arise. The most trivial instance is when $(\A,\H,\D,\cN,\Tau)$ is even, in which case the grading $\gamma$ will satisfy the conditions of Lemma \ref{gammajlem}.

Here is a more geometric example. Let $(C^\infty(M),L^2(M,S),\slashed{D})$ denote the spectral triple defined from a Dirac operator on a closed Riemannian manifold $M$ as in Proposition \ref{heatasumfd} (see page \pageref{heatasumfd}). If we take a collection $X_1,\ldots, X_N\in C^\infty(M,TM)$ of vector fields spanning the tangent bundle $TM$ in all points, the collection of Clifford multiplication operators $\gamma_j:=c_S(X_j)$, $j=1,\ldots,N$ is readily verified to satisfy the conditions of Lemma \ref{gammajlem}. This construction extends to semi-finite spectral triples defined from the fibrewise Dirac operator of a Riemannian spin$^c$ submersion $\pi:M\to B$ (see more in \cite{kaaaaadsuij}) and a measure on $B$ by taking $X_1,\ldots, X_N$ to be vertical vector fields spanning the vertical tangent bundle $\ker\mathrm{d}\pi$ in all points of $M$.
\end{example}

\begin{proof}
Define $C:=\|(\sum_{j=1}^N \gamma_j^2)^{-1}\|$ and recall that $F_\D:=\D|\D|^{-1}$ modulo a finite trace projection. For any self-adjoint $a\in \A$, we estimate
\begin{align*}
\Tau(F_\D a\e^{-t|\D|})=&\frac{1}{2}\Tau\left((F_\D a+aF_\D)\e^{-t|\D|}\right)\leq \frac{C}{2}\Tau\left(\sum_{j=1}^N\gamma_j(F_\D a+aF_\D)\e^{-t|\D|}\gamma_j\right)=\\
=&\frac{C}{2}\Tau\left(\sum_{j=1}^N\gamma_j(F_\D a+aF_\D)\gamma_j\e^{-t|\D|}\right)+\\
&+\frac{C}{2}\Tau\left(\sum_{j=1}^N\gamma_j(F_\D a+aF_\D)(\e^{-t|\D|}\gamma_j\e^{t|\D|}-\gamma_j)\e^{-t|\D|}\right)=\\
=&-\frac{C}{2}\Tau\left(\sum_{j=1}^N\gamma_j^2(F_\D a+aF_\D)\e^{-t|\D|}\right)+\\
&-\frac{C}{2}\Tau\left(\sum_{j=1}^N\gamma_j([F_\D,\gamma_j]_+ a+a[F_\D,\gamma_j]_+)\gamma_j\e^{-t|\D|}\right)\\
&+\frac{C}{2}\Tau\left(\sum_{j=1}^N\gamma_j(F_\D a+aF_\D)(\e^{-t|\D|}\gamma_j\e^{t|\D|}-\gamma_j)\e^{-t|\D|}\right)\leq \\
\leq &-\Tau\left(F_\D a\e^{-t|\D|}\right)-\frac{C}{2}\Tau\left(\sum_{j=1}^N\gamma_j([F_\D,\gamma_j]_+ a+a[F_\D,\gamma_j]_+)\gamma_j\e^{-t|\D|}\right)\\
&+\frac{C}{2}\Tau\left(\sum_{j=1}^N\gamma_j(F_\D a+aF_\D)(\e^{-t|\D|}\gamma_j\e^{t|\D|}-\gamma_j)\e^{-t|\D|}\right)
\end{align*}
Since $[\D,\gamma_j]_+$ is bounded, $[F_\D,\gamma_j]_+$ is compact and an approximation argument by finite $\Tau$-rank operators shows that 
$$\tau\left(\sum_{j=1}^N\gamma_j([F_\D,\gamma_j]_+ a+a[F_\D,\gamma_j]_+)\gamma_j\e^{-t|\D|}\right)=o(\Tau(\e^{-t|\D|})),$$
as $t\searrow \beta$. By norm continuity of $t\mapsto \e^{-t|\D|}\gamma_j\e^{t|\D|}$ we can also deduce that 
$$\tau\left(\sum_{j=1}^N\gamma_j(F_\D a+aF_\D)(\e^{-t|\D|}\gamma_j\e^{t|\D|}-\gamma_j)\e^{-t|\D|}\right)=o(\Tau(\e^{-t|\D|})),$$
as $t\searrow \beta$. In conclusion, for a self-adjoint $a$,
$$\Tau(F_\D a\e^{-t|\D|})=-\Tau(F_\D a\e^{-t|\D|})+o(\Tau(\e^{-t|\D|})).$$
We can conclude that $\tau(F_\D a\e^{-t|\D|})=o(\Tau(\e^{-t|\D|}))$ as $t\searrow \beta$. Since $2P_\D=F_\D+1$, we have for any $a\in \A$ that
$$\Tau(a\e^{-t|\D|})=2\Tau(P_\D a\e^{-t|\D|})+o(\Tau(\e^{-t|\D|})), \quad\mbox{as $t\searrow \beta$.}$$
In particular $\beta=\beta_\D$ and $\Tau(P_\D\e^{-t|\D|})\nearrow \infty$ as $t\searrow \beta$. We compute that 
$$\frac{\Tau(a\e^{-t|\D|})}{\Tau(\e^{-t|\D|})}= \frac{\Tau(P_\D a\e^{-t|\D|})}{\Tau(P_\D\e^{-t|\D|})}+o(1), \quad\mbox{as $t\searrow \beta$.}$$
In particular, for any extended limit $\omega$ as $t\to\beta$, 
$$ \lim_{t\to \omega}\frac{\Tau(a\e^{-t|\D|})}{\Tau(\e^{-t|\D|})}=  \lim_{t\to \omega}\frac{\Tau(P_\D a\e^{-t|\D|})}{\Tau(P_\D\e^{-t|\D|})}.$$
This concludes the proof of the lemma.
\end{proof}

\subsection{Modular spectral triples and modular index theory}
\label{modularsubsec}

Modular spectral triples and their (equivariant) index theory were considered in
\cite{CPR2,CRT,CNNR}, with the definition
laid out most clearly in \cite[Definition 2.1]{RenSen}. These were defined
in order to study the (equivariant) index theory of KMS weights associated to
periodic flows, so one might wonder how modular spectral triples fit into our scheme.

Given a KMS state $\psi:B\to\C$ with inverse temperature $\beta$ 
for a one-parameter group
$\sigma:\R\to {\rm Aut}(B)$ on a unital $C^*$-algebra and a 
faithful expectation onto the fixed point algebra $\Phi:\,B\to B^\sigma$,
we can emulate the constructions that inspired the definition of modular spectral triples.

First we construct the right $C^*$-module $X$ over $B^\sigma$
by completing $B$ in the norm coming from the inner product
$$
(x|y)_{B^\sigma}=\Phi(x^*y).
$$
Then we can use \cite{LN} to construct $\Tr_\psi:\Ko_{B^\sigma}(X)\to\C$ 
the trace dual to $\psi|_{B^\sigma}$. The action $\sigma$ induces
a one parameter unitary group  on $L^2(X,\psi)$, and we let $\D$ 
be the generator of this one parameter group.
By \cite{CNNR}, when the action $\sigma$ is periodic, the data
$$
(B,L^2(X,\psi),\D,\Ko_{B^\sigma}(X)'',\phi_\D),
$$
where $\phi_\D(a)
:=\Tr_\psi(e^{-\beta \D/2}a e^{-\beta \D/2})$, 
gives us a modular spectral triple. The operator $\D$ is affiliated to
the semifinite algebra $ (\Ko(X)^{\sigma^{\phi_\D}})''$

\begin{prop}
\label{prop:Li-mod}
Let $(\A,\H,\D,\cN,\Tau)$ be a semifinite spectral triple
such that 
that for all $t>\beta$
$$
\Tau(P_\D \e^{-t|\D|})<\infty
$$
and $\lim_{t\searrow \beta}\Tau(P_\D \e^{-t\D})=\infty$. 
Assume that $P_\D AP_\D\cap\Ko_\cN=\{0\}$ 
and the spectral triple is $\beta$-analytic.
Then we obtain the KMS state $\phi_\omega:\,A_\D\to\C$,
and provided that $\D$ has discrete spectrum we obtain
an expectation $\Phi:\,A_\D\to A_\D^\sigma$ onto the fixed point subalgebra.

Provided that  the group $\sigma$ is periodic
we then obtain a finitely summable modular spectral triple
$$
(\A_\D,L^2(A_\D,\phi_\omega),\D,(\Ko_{A_\D^\sigma}(A_\D))'',\phi_\D),
$$
where the operator $\D$ generates the one-parameter group 
induced by $\sigma$ on $L^2(A_\D,\phi_\omega)$ 
and $\phi_\D:=\Tr_{\phi_\omega}(e^{-\beta \D/2}\cdot e^{-\beta \D/2})$. 
The spectral dimension is 1.
\end{prop}

\begin{proof}
The existence of the KMS state $\phi_\omega$ comes from
Corollary \ref{cor:phi-omega}.

In general the action $\sigma$ is real, but assuming that the operator 
$\D$ has discrete spectrum, the action will 
factor through a (compact) torus. To see this, 
one takes a rational basis of the eigenvalues 
(possibly an infinite basis), and takes a product 
over the circles corresponding to these individual actions.

Consequently, by averaging over this torus,
there is an expectation $\Phi:\,\cN\to \cN^{\sigma}$ onto
the fixed point algebra for 
$t\mapsto (T\mapsto \e^{it\D}T\e^{-it\D})$. 
Of course $\D$ is affiliated to the fixed-point algebra.

Finally if the action $\sigma$ is periodic then \cite{CNNR}
proves that we have a modular spectral triple, and that
$\phi_\D((1+\D^2)^{-s/2})<\infty$ for $s>1$.
\end{proof}

For circle actions there is a 
local index formula in twisted cyclic theory, 
but for real actions factoring through a torus there is not. 

One serious issue that comes up 
in this more general setting is the compactness of the resolvent
of $\D$, and determining summability. We leave this issue to 
another place.

\section{The KMS-state $\phi_\omega$ and Dixmier traces}
\label{kmsanddix}

In the present section we discuss a relation 
between the trace $\phi_\omega$ from 
Theorem~\ref{thm:voics} and Dixmier traces.
For a decreasing function $\psi: [0,\infty) \to (0,\infty)$ we denote $\Psi(t):=\int_0^t \psi(s) ds.$
Let $\mathcal{L}_\psi(\Tau)$ be the principal ideal defined as in Definition \ref{derjugendvonheutebrauchidealen} (see page \pageref{derjugendvonheutebrauchidealen}).
Let $E_T$ be the spectral projection of an operator $T$ affiliated with $\cN$ and let $n_\Tau(s,T):=\Tau(E_{|T|}(s,\infty))$ be its distribution function.

The following result  extends~\cite[Lemma 12.6.3]{LSZ}.

\begin{lemma}\label{HTas} Let $\psi$ be a regularly varying function of index $-1$.
Let $T\in \mathcal{L}_\psi(\Tau)$ be strictly positive and $\mu_\Tau(t,T) \sim \psi(t)$, $t\to \infty$. For every $q>0$ we have
$$\Tau (\e^{-{T^{-q}}/t}) \sim \Gamma(1+\frac{1}{q}) \psi^{-1}(t^{-\frac{1}{q}}), \ t\to \infty.$$
\end{lemma}

\begin{proof}
The assumption $\mu_\Tau(t,T) \sim \psi(t)$ implies $\mu_\Tau(t,T^{q}) \sim [\psi(t)]^{q}$, $t\to \infty$. Since the distribution function is an inverse of the singular values function, it follows that $n_\Tau(s, {T^{q}})\sim \psi^{-1}(s^{\frac{1}{q}}),$ $s\to0+$.
Next we have $\Tau(E_{T^{-q}}(t)) \sim n_\Tau(1/t, {T^{q}})$, $t\to \infty$. Thus, 
$$\Tau(E_{T^{-q}}(t)) \sim \psi^{-1}(t^{-\frac{1}{q}}), \ t\to \infty.$$

Since $\psi$ is varying regularly with index $-1$, \cite[Theorem 1.5.12]{RegVar} implies that $\psi^{-1}$ varies regularly with index $-1$, too. Thus, $\Tau(E_{T^{-q}})$ varies regularly with index $\frac{1}{q}$.
Writing the heat trace as a Laplace transform
$$
\Tau (\e^{-{T^{-q}}/t}) = \int_0^\infty \e^{-{z}/{t}}\,\mathrm{d} \Tau(E_{T^{-q}}(z))
$$
and using the Karamata theorem~\cite[Chapter IV, Theorem 8.1]{Korevaar} we obtain
\[
\Tau (\e^{-{T^{-q}}/t}) \sim \Gamma(1+\frac{1}{q})\psi^{-1}(t^{-\frac{1}{q}}), \ t\to \infty.\qedhere
\]
\end{proof}

Let $P_a : L^\infty(0,\infty) \to L^\infty(0,\infty)$ 
be the exponentiation operator defined by
$(P_a f)(t) := f(t^a), \ t>0.$

\begin{defn}\label{dfn:invariance}\cite{CPS2}
An extended
limit $\omega$ as $t\to\infty$ on $L^\infty(0,\infty)$ is said to be
 exponentiation invariant if 
$$\lim_{t\to\omega} (P_a f)(t) = \lim_{t\to\omega} f(t)$$
for every $f\in L^\infty(0,\infty)$ and every $a>0$.
\end{defn}

\begin{defn}
\label{Dixmier traces}
Let $\psi: [0,\infty) \to (0,\infty)$ be a regularly varying function of index $-1$. For any extended limit $\omega$ as $t\to\infty$ on $L^\infty(0,\infty)$ a linear extension of the weight
\begin{equation*}
\Tau_{\omega, \psi}(T):= \lim_{t\to\omega} \frac{1}{\Psi(t)}\int_0^t \mu_\Tau(s,T)\,\mathrm{d}s, \quad 0\le T\in \mathcal L_\psi(\Tau)
\end{equation*}
is said to be a Dixmier trace on $\mathcal L_\psi(\Tau)$.
\end{defn}

\begin{rmk}
Usually Dixmier traces are defined on Lorentz ideals corresponding to the function $\Psi$ (which are strictly larger than $\mathcal L_\psi(\Tau)$) by exactly the same formula as in Defninition~\ref{Dixmier traces} (see e.g.~\cite{Dixmier, BRB, LSZ}). Then, Dixmier traces on $\mathcal L_\psi(\Tau)$ are restrictions of those on Lorentz ideal to $\mathcal L_\psi(\Tau)$. Since we do not deal with Lorentz ideals here, it is convenient to define Dixmier traces directly on $\mathcal L_\psi(\Tau)$.

Also, it should be pointed out that on Lorentz ideals to define Dixmier traces one needs an additional assumption on $\omega$: either dilation invariance \cite{Dixmier, LSZ} or exponentiation invariance \cite{GaSu, SUZ3}. As it was shown in \cite[Theorem 17]{Sed_Suk} these requirements are redundant on $\mathcal L_\psi(\Tau)$.
\end{rmk}

The proof of the following theorem is the same as that of~\cite[Theorem 8.5.1]{LSZ} and thus omitted. Note however, that in~\cite{LSZ} the result was proved for Lorentz ideals and required dilation invariance of the extended limit $\omega$. For the case of $\mathcal L_\psi(\Tau)$ one can refer to~\cite[Lemma 15]{Sed_Suk} to remove this assumption.

\begin{thm}\label{zeta}
Let $f\in C^2[0,\infty)$ be a bounded function such that $f(0)=f'(0)=0$. Let $T\in \mathcal{L}_\psi(\Tau)$ be positive and let $B\in \mathcal N$. For every extended
limit $\omega$ as $t\to\infty$ on $L^\infty(0,\infty)$ we have
$$\lim_{t\to\omega}\left(\frac1{\Psi(t)} \int_1^t \Tau (f(sT)B) \frac{\mathrm{d}s}{s^2}\right) =
\int_0^\infty f(s) \frac{\mathrm{d}s}{s^2} \cdot \lim_{t\to\omega}\left(\frac1{\Psi(t)} \int_1^t \Tau (\e^{-(sTB)^{-1}}) \frac{\mathrm{d}s}{s^2}\right).$$
\end{thm}

Below we will need the relation between generalised heat kernels and Dixmier traces on $\mathcal{L}_\psi(\Tau)$, which was proved in~\cite{GaSu} under the additional assumption that
\begin{equation}\label{exp1}
A_\Psi(\alpha):=\lim_{t\to\infty} \frac{\Psi(t^\alpha)}{\Psi(t)}\quad\mbox{exists for every $\alpha> 0$.}
\end{equation}
Recall the notation $\Psi(t)=\int_0^t\psi(s) ds$. The corresponding (natural) assumption on $\psi$ is  that
\begin{equation}
\label{exp2}
\alpha \cdot \lim_{t\to\infty} \frac{\psi(t^\alpha) t^{\alpha-1}}{\psi(t)} \quad\mbox{exists for every $\alpha> 0$.}
\end{equation}
The number appearing in Equation \eqref{exp2} equals $A_\Psi(\alpha)$. Note that condition~\eqref{exp2} implies the condition~\eqref{exp1} and regular variation of $\psi$ with index $-1$ \cite[Proposition 1.7]{GaSu}. 

Let $H: L^\infty(0,\infty) \to L^\infty(0,\infty)$ be the Cesaro mean defined as follows:
$$
(Hf)(t) := \frac1t \int_0^t f(s) \,\mathrm{d}s, \ t>0.
$$ 
Let $M_\Psi: L^\infty(0,\infty) \to L^\infty(0,\infty)$ be the Cesaro mean twisted by $\Psi$, that is
\begin{align*}
(M_\Psi f)(t):&= [(H (f\circ \Psi^{-1}))\circ \Psi](t)
=\frac1{\Psi(t)} \int_0^t f(s) \psi(s) \,\mathrm{d}s,  t>0.
\end{align*}

\begin{lemma}
\label{MP}
If $\psi$ satisfies condition~\eqref{exp2}, then 
$$M_\Psi \circ P_a - P_a \circ M_\Psi : L^\infty(0,\infty) \to C_0(0,\infty)$$
for every $a> 0.$
\end{lemma}

\begin{proof}
For $f\in L^\infty(0,\infty)$ we have
$$(M_\Psi \circ P_a f)(t)= \frac1{\Psi(t)} \int_0^t f(s^a) \psi(s) \,\mathrm{d}s= \frac1{\Psi(t)} \int_0^{t^a} f(s) \psi(s^{1/a}) \frac{s^{1/a-1}}a \,\mathrm{d}s.$$
Using conditions~\eqref{exp1} and~\eqref{exp2} we obtain
$$(M_\Psi \circ P_a f)(t)\in A_\Psi(a)A_\Psi(\frac1a) \frac1{\Psi(t^a)} \int_0^{t^a} f(s) \psi(s) \,\mathrm{d}s + C_0.$$

Direct calculations show that $A_\Psi(a)A_\Psi(\frac1a)=1$.
Thus,
\[
(M_\Psi \circ P_a  - P_a \circ M_\Psi )f \in C_0.\qedhere
\]
\end{proof}

\begin{lemma}
\label{ologdlemma}
If $\psi:[0,\infty)\to (0,\infty)$ is a decreasing function with regular variation of index $-1$. Then for any $d>0$, as $t\to \infty$, 
$$\psi(t)=o((\log(t))^{-d}).$$
\end{lemma}

\begin{proof}
Since $\psi$ is decreasing and $\lim_{t\to\infty} \frac{\psi(2t)}{\psi(t)}=2^{-1}$, we can for any $\epsilon\in (0,1)$ find a constant $C_\epsilon>0$ such that for $t\in (2^k,2^{k+1}]$,
$$\psi(t)\leq C_\epsilon (2-\epsilon)^{-k}.$$
For details, see \cite[Proposition 2.21]{GU}. Therefore, $\psi(t)=o(t^{-1/2})$ and the lemma follows.
\end{proof}

For the next result we need to assume the  (stronger) condition,
that the function $\psi$ satisfies
\begin{equation}
\label{invas}
\lim_{t\to\infty} \frac{t^2 \psi(t)}{\psi^{-1}(1/t)}=c,
\end{equation}
for some constant $c>0$.

\begin{rmk}
For every $k\in\Z$, the functions $\psi(t)= \frac{\log^k t}{t}$ and $\psi(t)= \frac{\log^k (\log t)}{t \cdot \log t}$ satisfy condition~\eqref{invas}. The functions $\psi=\Psi'$ with $\Psi(t)= \e^{\log^\beta t}$ do not satisfy~\eqref{invas} for any $\beta>0$.
\end{rmk}

In the following result we use a singular values function of an unbounded operator affiliated with $\cN$. 
For such operators the formula~\eqref{mu} cannot be used. The singular values function of an operator $T$ affiliated with $\cN$ is defined \cite{FK} as follows:
$$
\mu_\Tau(t,T):=\inf\{s\ge0 : n_\Tau(s,T)\le t\}.
$$

\begin{thm}
\label{Frohlich}
Let $d>0$ and $\psi:[0,\infty)\to (0,\infty)$ be a decreasing function satisfying conditions~\eqref{exp2} and~\eqref{invas}. Assume that $(\A, \H,\D,\cN,\Tau)$ is a unital semifinite spectral triple such that 
\begin{enumerate}
\item $\Tau$ is infinite;
\item $\D$ is positive;
\item $\mu_\Tau(t,\D)\sim \psi(t)^{-1/d}$.
\end{enumerate}
Then $(\A, \H,\D,\cN,\Tau)$ is $\mathrm{Li}_{(0),1}$-summable with positive $\Tau$-essential spectrum. Moreover, for every $a\in A_\D$ and every exponentiation invariant extended limit $\omega$ as $t\to\infty$ we have
$$\phi_{\tilde{\omega}}(a)  = \Tau_{\omega, \psi}(a (1+\D^2)^{-d/2}),$$
where $\tilde{\omega}:=\omega\circ (JM_\Psi J)$ and $J:L^\infty(0,\infty)\to L^\infty(0,\infty)$ is defined as the pullback along $t\mapsto t^{-1}$.
\end{thm}

\begin{proof}
Condition \eqref{exp2} on $\psi$ implies that $\psi$ has regular variation of index $-1$ so $\psi(t)=o((\log(t))^{-d})$ by Lemma \ref{ologdlemma}. Therefore, $\mu_\Tau(t,\D)=o(\log(t))$ and the semifinite spectral triple $(\A, \H,\D,\cN,\Tau)$ is $\mathrm{Li}_{(0),1}$-summable. Since $\Tau$ is infinite, $\Tau(\e^{-t\D})\nearrow \infty$ as $t\searrow 0$. The operator $\D$ therefore has positive $\Tau$-essential spectrum. Thus, 
$$\phi_{\tilde{\omega}}(a)= \lim_{t\to\tilde{\omega}}\frac{\Tau (a \e^{-t\D})}{\Tau (\e^{-t\D})}$$ 
is a trace on $\A$ by Theorem~\ref{thm:voics}.
The assumption on $\D$ implies that 
$\mu_\Tau(t, (1+\D^2)^{-d/2}) \sim \psi(t)$, $t\to\infty$. Applying Lemma~\ref{HTas} with $T=(1+\D^2)^{-d/2}$ and $q=1/d$ yields
$$
\Tau (\e^{-\D/t}) \sim \Gamma(1+d) \psi^{-1}(t^{-d}), \ t\to \infty.
$$
Using the properties of extended limits, the definitions of $P_a$ and $J$ and Lemma~\ref{MP}, we obtain
\begin{align}
\phi_{\tilde{\omega}}(a) &= \frac1{\Gamma(1+d)}\lim_{t\to\omega}(J\circ M_\Psi)\left(\frac{\Tau (a \e^{-\D/t})}{\psi^{-1}(t^{-d})} \right)\nonumber\\
&= \frac1{\Gamma(1+d)}\lim_{t\to\omega}(J\circ M_\Psi \circ P_{d})\left(\frac{\Tau (a \e^{-(t \D^{-d})^{-1/d}})}{\psi^{-1}(1/t)} \right)\nonumber\\
&= \frac1{\Gamma(1+d)} \lim_{t\to\omega}(J\circ P_{d} \circ M_\Psi )\left(\frac{\Tau (a \e^{-(t \D^{-d})^{-1/d}})}{\psi^{-1}(1/t)} \right).
\label{eq10}
\end{align}

Using the definition of $M_\Psi$ and assumption~\eqref{invas} we obtain
\begin{align}
M_\Psi \left(\frac{\Tau (a \e^{-(t \D^{-d})^{-1/d}})}{\psi^{-1}(1/t)} \right)&=
\frac1{\Psi(t)} \int_0^t \frac{\Tau (a \e^{-(s \D^{-d})^{-1/d}})}{\psi^{-1}(1/s)} \psi(s) \,\mathrm{d}s\nonumber\\
&\in
\frac1{\Psi(t)} \int_0^t \Tau (a \e^{-(s \D^{-d})^{-1/d}}) \frac{\mathrm{d}s}{s^2} + C_0(0,\infty).
\label{eq11}
\end{align}

Since $\omega$ is exponentiation invariant extended limit, combining~\eqref{eq10} and~\eqref{eq11} we obtain
$$\phi_{\tilde{\omega}}(a) =\frac1{\Gamma(1+d)}  \cdot \lim_{1/t\to\omega} \left(\frac1{\Psi(t)} \int_1^t \Tau (a \e^{-(s \D^{-d})^{-1/d}}) \frac{\mathrm{d}s}{s^2} \right).$$

Now we apply Lemma~\ref{zeta} twice with $T=\D^{-d}$, $f(x)=\e^{-x^{-1/d}}$ and then with $f(x)=\e^{-x^{-1}}$. Since 
$$
\int_0^\infty \e^{-x^{-q}}\frac{\mathrm{d}x}{x}=\Gamma(1+\frac{1}{q}),
$$ 
we obtain
\begin{align*}
\phi_{\tilde{\omega}}(a)&= \lim_{1/t\to\omega} J\left(\frac1{\Psi(t)} \int_1^t \Tau (a\e^{-(s\D^{-d})^{-1}}) \frac{\mathrm{d}s}{s^2}\right)=\Tau_{\omega, \psi}\left(a (1+\D^2)^{-d/2}\right),
\end{align*}
where the last equality follows from~\cite[Theorem 4.7]{GaSu}.
\end{proof}

We can now provide sufficient conditions on a general $\mathrm{Li}_{(0),1}$-summable unital semifinite spectral triples ensuring a relation between the trace $\phi_\omega$ of Theorem \ref{thm:voics} and Dixmier traces.

\begin{corl}
\label{dixmiercorforphiom}
Let $(\A,\H,\D,\cN,\Tau)$ be a unital semifinite spectral triple with $\Tau(P_\D)=\infty$. Assume that there is a number $d>0$ and a decreasing function $\psi:[0,\infty)\to (0,\infty)$ with regular variation of index $-1$ satisfying conditions~\eqref{exp2} and~\eqref{invas} and 
$$\mu(t,P_\D\D)\sim \psi(t)^{-1/d}.$$
Then, $\beta_\D=0$ and for any exponentiation invariant extended limit $\omega$ as $t\to\infty$ and $a\in A_\D$,
$$\phi_{\tilde{\omega}}(a)  = \Tau_{\omega, \psi}(P_\D a (1+\D^2)^{-d/2}),$$
where $\tilde{\omega}$ is as in Theorem \ref{Frohlich}.
\end{corl}

The corollary follows immediately from Theorem \ref{Frohlich} applied to the unital semifinite spectral triple $(\mathfrak{T},P_\D\H,P_\D\D,\cN^+,\Tau)$ where $\mathfrak{T}$ is the $*$-algebra generated by $P_\D \A P_\D$.

\begin{example}
Let us revisit the spectral triple constructed in Proposition \ref{heatasumfdwithapsi} (see page \pageref{heatasumfdwithapsi}).
If $\psi:[0,\infty)\to (0,\infty)$ is a smoothly varying function with $\lim_{t\to 0}\psi(t)=0$ and $\psi(t)^{-1}=O(t^{1/n})$ as $t\to \infty$, and $\slashed{D}$ a Dirac operator on a Riemannian closed $n$-dimensional manifold $M$, a $\psi$-summable spectral triple $(C^\infty(M),L^2(M,S),\slashed{D}_\psi)$ was constructed in Proposition \ref{heatasumfdwithapsi}, where $\slashed{D}_\psi:=F_{\slashed{D}} \psi(|\slashed{D}|^n)^{-1}$. The Weyl law for $|\slashed{D}|$ and Theorem \ref{heatvszeta} applied to $B=P_{\slashed{D}}$ shows that the order of the spectral asymptotics of $|\slashed{D}|$ coincides with the order of the spectral asymptotics of $P_\slashed{D}\slashed{D}$ so $\mu(t,P_{\slashed{D}}\slashed{D}_\psi)\sim c\psi(t)^{-1}$ for some constant $c>0$. 

If $\psi$ has regular variation of index $-d$ for a $d>0$, Corollary \ref{dixmiercorforphiom} shows that the tracial state on $C(M)$ defined from the spectral triple $(C^\infty(M),L^2(M,S),\slashed{D}_\psi)$ takes the form
$$\phi_\omega(a)= c'\Tr_{\omega, \psi}(P_\slashed{D} a \psi(|\slashed{D}|^n)^{1/d}),$$
for some proportionality constant $c'$. Applying Connes' trace theorem as in \cite[Theorem 9.1]{GU}, it follows that 
$$\phi_\omega(a)=\displaystyle\stackinset{c}{}{c}{}{-\mkern4mu}{\displaystyle\int_M} a\mathrm{d}V,$$
where $\mathrm{d}V$ denotes the Riemannian volume measure on $M$ and $\displaystyle\stackinset{c}{}{c}{}{-\mkern4mu}{\displaystyle\int_M}$ the normalized integral. 

The computation above requires $\psi$ to have strictly negative index of regular variation. We note that by Proposition \ref{deathorgladiolis} (see page \pageref{deathorgladiolis}), the computation above can only extend to the spectral triple from Proposition \ref{psidiracprop} (see page \pageref{psidiracprop}) on a crossed product by a local diffeomorphism if the local diffeomorphism acts isometrically.
\end{example}

\section{The KMS-state $\phi_\omega$ in examples}
\label{kmsinexamplesec}

We are now in a state where we can compute the KMS-states associated with the spectral triples considered in Subsection \ref{subsec:examples} (see page \pageref{subsec:examples}). The computations for the KMS-states associated with the unbounded Kasparov cycle on Cuntz-Pimsner algebras considered in Subsection \ref{cpalgexam} (see page \pageref{cpalgexam}) are more involved and dedicated a separate section, Section \ref{diraccpkms} (see page \pageref{diraccpkms}).

\subsection{Dirac operators on closed manifolds}
\label{diracmfdkms}

For a closed manifold $M$ with a Dirac operator $\slashed{D}$ acting on a Clifford bundle $S\to M$, we consider the spectral triple $(C^\infty(M),L^2(M,S),\slashed{D})$ as in Proposition \ref{heatasumfd} (see page \pageref{heatasumfd}).  
The following theorem can be deduced immediately from Example \ref{diracmfdhea} (see page \pageref{diracmfdhea}) or from Corollary \ref{dixmiercorforphiom} and Connes' trace theorem for pseudo-differential operators.

\begin{thm}
\label{simpleacse}
Let $(C^\infty(M),L^2(M,S),\slashed{D})$ be the spectral triple associated with a Dirac operator on a closed manifold, $\omega$ an extended limit as $t\to0$ and $\phi_\omega$ the associated tracial state from Theorem \ref{thm:voics} (see page \pageref{thm:voics}). The trace $\phi_\omega$ is the normalized volume integral on $M$, i.e. for $a\in C(M)$, 
$$
\phi_\omega(a)=\displaystyle\stackinset{c}{}{c}{}{-\mkern4mu}{\displaystyle\int_M} a\,\mathrm{d} V.
$$
\end{thm}

For completeness, let us describe the Toeplitz algebras $T_{C(M)}$, $T_{C(M),\slashed{D}}$ and the flow $\sigma$ on $C(M)_{\slashed{D}}$ in this example. We remark that since $\phi_\omega$ is a trace in this case, the flow induced from our construction is irrelevant for the study of $\phi_\omega$ but it could nevertheless serve to clarify the constructions of Subsection \ref{subsec:toplitz}. The relevant algebras are all contained in the $C^*$-algebra $\Psi^0_{C^*}(M,S)$ defined as the $C^*$-closure of the $*$-algebra $\Psi^0_{\rm cl}(M,S)$ of zeroth order classical pseudo-differential operators acting on $L^2(M,S)$. It is well known that $\Psi^0_{C^*}(M,S)$ fits into a short exact sequence 
$$0\to \Ko(L^2(M,S))\to \Psi^0_{C^*}(M,S)\xrightarrow{{\rm symb}}C(S^*M,\End(S))\to 0,$$
where ${\rm symb}$ denotes the continuous extension of the principal symbol mapping $\Psi^0_{\rm cl}(M,S)\to C^\infty(S^*M,\End(S))$. Since $P_{\slashed{D}}$ is a projection in $\Psi^0_{\rm cl}(M,S)$, we can consider the $C^*$-algebra $\Psi^0_{C^*,+}(M,S):=P_{\slashed{D}}\Psi^0_{C^*}(M,S)P_{\slashed{D}}$ and we obtain a short exact sequence 
$$0\to \Ko(P_{\slashed{D}}L^2(M,S))\to \Psi^0_{C^*,+}(M,S)\xrightarrow{{\rm symb}}C(S^*M,\End(p_{\slashed{D}}S))\to 0,$$
where $p_{\slashed{D}}\in C^\infty(S^*M,\End(S))$ denotes the principal symbol of $P_{\slashed{D}}$. 
The algebras $T_{C(M)}$ and $T_{C(M),\slashed{D}}$ are characterized by the following commuting diagram with exact rows
$$
\xymatrix{
0\ar[r]& \Ko(P_{\slashed{D}}L^2(M,S))\ar[r]\ar[d]^{=}&T_{C(M)}\ar[rr]^{{\rm symb}}\ar@{^{(}->}[d]&&C(M)\ar@{^{(}->}[d]\ar[r]&0\\
0\ar[r]& \Ko(P_{\slashed{D}}L^2(M,S))\ar[r]\ar[d]^{=}&T_{C(M),\slashed{D}}\ar[rr]^{{\rm symb}}\ar@{^{(}->}[d]&&C(M)_{\slashed{D}}\ar@{^{(}->}[d]\ar[r]&0\\
0\ar[r]&\Ko(P_{\slashed{D}}L^2(M,S))\ar[r]&\Psi^0_{C^*,+}(M,S)\ar[rr]^{{\rm symb}}&&C(S^*M,\End(p_{\slashed{D}}S))\ar[r]&0,
}
$$
The composition of the mappings in the right most column coincides with the pull back homomorphism $C(M)\to C(S^*M)$ composed with the inclusion $C(S^*M)\subseteq C(S^*M,\End(p_{\slashed{D}}S))$.

To describe the flow $\sigma$ on $C(M)_{\slashed{D}}$, 
we describe it on $C^\infty(S^*M,\End(p_{\slashed{D}}S))$. 
Surjectivity of the principal symbol mapping implies that any 
$a\in C^\infty(S^*M,\End(p_{\slashed{D}}S))$ is the symbol of an operator 
$A\in P_{\slashed{D}}\Psi^0_{\rm cl}(M,S)P_{\slashed{D}}$. 
By Egorov's theorem \cite{egorovref} (see also \cite[Section IV]{horacta} and \cite{duissing}), $\e^{is\slashed{D}}A\e^{-is\slashed{D}}$ 
is again an element of $P_{\slashed{D}}\Psi^0_{\rm cl}(M,S)P_{\slashed{D}}$ 
and the expression
$\sigma_s(a):={\rm symb}(\e^{is\slashed{D}}A\e^{-is\slashed{D}})$ 
gives a well defined flow on $C^\infty(S^*M,\End(p_{\slashed{D}}S))$. 
Again by Egorov's theorem, using that 
$\sigma_s(a)={\rm symb}(\e^{is|\slashed{D}|}A\e^{-is|\slashed{D}|})$ for 
$A\in P_{\slashed{D}}\Psi^0_{\rm cl}(M,S)P_{\slashed{D}}$, 
we have that $\sigma_s(a)=g_s^*(a)$ where $g_s:S^*M\to S^*M$ is the 
Hamiltonian flow associated with 
the symbol $|\xi|$ of $|\slashed{D}|$. On the cosphere bundle, this Hamiltonian flow 
coincides with the geodesic flow. 
We conclude that the flow $\sigma$ is induced by geodesic flow and 
$C(M)_{\slashed{D}}\subseteq C(S^*M)$ is a closed subalgebra invariant under geodesic flow. This discussion should be compared with that in \cite{CSG}.

As in Subsection \ref{diracmfdfirst}, we consider a local diffeomorphism $g:M\to M$. 
Assuming that $g$ acts conformally and lifts to $S\to M$, it is readily verified that $g$ is compatible 
with the decreasing function $\psi(t):=\frac{1}{\log(1+t^{2/n})}$. We use the notation $\slashed{D}_{\rm log}:=\slashed{D}_\psi$ for this particular choice of $\psi$. Note that 
$$
\slashed{D}_{\rm log}=F_\slashed{D}\log(1+\slashed{D}^2)
\quad\mbox{and}\quad \e^{-t|\slashed{D}_{\rm log}|}=(1+\slashed{D}^2)^{-t}.
$$ 
By Proposition \ref{psidiracprop}, we arrive 
at a spectral triple $(\A,L^2(M,S),\slashed{D}_{\rm log})$ where $\A$ 
is the $*$-algebra generated by $C^\infty(M)$ and an 
isometry $V_g$. Let us compute KMS-state 
constructed from $(\A,L^2(M,S),\slashed{D}_{\rm log})$. 
Before diverting into this computation, 
we recall that the $C^*$-closure of $\A$ coincides 
with the image of a representation of $O_{E_g}$. 
As such, we can write elements of $\A$ as linear span 
of elements of the form $S_\mu S_\nu^*$ where 
$\mu=a_1V_g\cdots a_k V_g$ and 
$\nu=b_1V_g\cdots b_l V_g$, where 
$a_1,\ldots, a_k,b_1,\ldots, b_l\in C^\infty(M)$.

\begin{prop}
Set $S:=C^\infty(M)\cup C^\infty(M)V_g\subseteq \A$. 
For any $\beta\in \R$, the set $S$ is an analytically generating set at $\beta$ for 
$(\A,L^2(M,S),\slashed{D}_{\rm log})$. 
\end{prop}

\begin{proof}
For notational convenience, write $\D=\slashed{D}_{\rm log}$. 
The set $S$ generates $\A$, so $P_\D SP_\D+\Ko$ 
generates $T_A$. For any $\beta\in \R$, and $a\in S$
$$
\e^{-\beta\D}P_\D aP_\D\e^{\beta\D}
=(1+\slashed{D}^2)^{-\beta}P_\slashed{D} 
aP_\slashed{D}(1+\slashed{D}^2)^{\beta}.
$$
The proposition follows from that 
$\Dom((1+\slashed{D}^2)^\beta)=H^{2\beta}(M,S)$ 
as Banach spaces and any $a\in S$ extends to a 
continuous operator on the Sobolev spaces 
$H^{2\beta}(M,S)$ for all $\beta$.
\end{proof}

\begin{thm}
\label{kmsforconfdiff}
Let $M$ be a connected $n$-dimensional Riemannian manifold, 
$\slashed{D}$ a Dirac operator on $S\to M$, and 
$g:M\to M$ a local diffeomorphism acting conformally 
and lifting to $S$. Then the spectral triple 
$(\A,L^2(M,S),\slashed{D}_{\rm log})$ is 
$\mathrm{Li}_1$-summable, has positive 
essential spectrum with $\beta_\D=n/2$ and 
is $n/2$-analytic. 

Assume for all $m\in \N_+$, that the set of fixed points
$$
\{x\in M: g^m(x)=x\},
$$
has measure zero with respect to the volume measure. Then the state $\phi_\omega$ 
on $A$ constructed from Corollary \ref{cor:phi-omega} 
(see page \pageref{cor:phi-omega}) is independent of 
$\omega$ and takes the form 
\begin{equation}
\label{formforphiomforbg}
\phi_\omega(S_\mu S_\nu^*)
=\delta_{|\mu|,|\nu|}\displaystyle\stackinset{c}{}{c}{}{-\mkern4mu}{\displaystyle\int_M}\mathfrak{L}_g(c_g^{n/2}b_k^*\mathfrak{L}_g(c_g^{n/2}b_{k-1}^*\mathfrak{L}_g(\cdots \mathfrak{L}_g(c_g^{n/2}b_1^*a_1)a_2)\cdots a_{k-1})a_k) \,\mathrm{d}V,
\end{equation}
For $\mu=a_1V_g\cdots a_k V_g$ and 
$\nu=b_1V_g\cdots b_l V_g$. Here 
$\mathrm{d}V$ denotes the Riemannian volume form. 

The state $\phi_\omega$ 
viewed as a state on $O_{E_g}$ via its representation 
on $L^2(M,S)$ is KMS for the action defined by 
$\gamma_t(aV_g):=c_g^{itn/2}aV_g$  with inverse temperature $1$.
\end{thm}

We note that the state Theorem \ref{kmsforconfdiff} is a KMS-state on a Cuntz-Pimsner algebra, but not for its gauge action. If $c_g(x)<1$ for some $x\in M$, the generator of $\gamma$ is not positive on $E_g$ and the Laca-Neshveyev correspondence \cite{LN} does not apply in the context of $O_{E_g}$ with the action $\gamma$.

In the case $k=l=0$, the formula \eqref{formforphiomforbg} 
should be interpreted as 
$\phi_\omega(a)=\displaystyle\stackinset{c}{}{c}{}{-\mkern4mu}{\displaystyle\int_M} a\,\mathrm{d}V$ for $a\in C^\infty(M)$. 
This special case follows from Connes' trace theorem.

\begin{proof}
For $t>n/2$, standard techniques of pseudo-differential operators show that the integral kernel $K_t$ of the trace class operator $P_\slashed{D} (1+\slashed{D}^2)^{-t}$ belongs to $C(M\times M,\mathrm{END}(S))\cap C^\infty(M\times M\setminus \Delta_M,\mathrm{END}(S))$. Here $\Delta_M\subseteq M\times M$ denotes the diagonal and $\mathrm{END}(S)$ denotes the big endomorphism bundle defined by $\mathrm{END}(S)_{(x,y)}:=\mathrm{Hom}(S_x,S_y)$ for $(x,y)\in M\times M$. By \cite[Proposition 8.3]{DGMW}, $V_g^*$ takes the form
$$V_g^*=N^{1/2}\mathfrak{L}_{S,g}c_g^{-n/4},$$
where $\mathfrak{L}_{S,g}$ is the $L^2$-extension of the operator 
$$\mathfrak{L}_{S,g}:C(M,S)\to C(M,S), \quad \mathfrak{L}_{S,g}\xi(x):=\sum_{g(y)=x} u_{g}(y)^{-1}\xi(y).$$

Take $a_1,\ldots, a_k,b_1,\ldots, b_l\in C^\infty(M)$ and write $\mu=a_1V_g\cdots a_k V_g$ and $\nu=b_1V_g\cdots b_l V_g$. We introduce the notation $\tilde{a}_j:=a_jc_g^{n/4}$ and $\tilde{b}_j=b_j c_g^{-n/4}$. We can compute for $t>n/2$ that 
\small
\begin{align*}
\Tr&_{L^2(M,S)}(P_\slashed{D}S_\mu S_\nu^*\e^{-t|\slashed{D}_{\rm log}|})=\Tr_{L^2(M,S)}(a_1V_g\cdots a_k V_gV_g^*b_l^*V_g^*\cdots V_g^*b_1^*P_\slashed{D} (1+\slashed{D}^2)^{-t})\\
&{}\\
&=N^{-(k-l)/2}\Tr_{L^2(M,S)}(a_1c_g^{n/4}u_g g^*\cdots a_kc_g^{n/4}u_g g^*\mathfrak{L}_{S,g}c_g^{-n/4}b_l^*\cdots c_g^{-n/4}\mathfrak{L}_{S,g} b_1^*P_\slashed{D} (1+\slashed{D}^2)^{-t})\\
&{}\\
&=N^{-(k-l)/2}\int_{M}\sum_{(x_1,\ldots,x_{k+l})\in M(x,k,l)}\left(\prod_{j=1}^{k} \tilde{a}_j(x_j)\right)\left(\prod_{j=k+1}^{k+l} \tilde{b}_{j-k}(x_j)^*\right)K_t(x_{k+1},x_1) \mathrm{d}V(x),
\end{align*}
\normalsize
where $M(x,k,l)\subseteq M^{k+l}$ is defined as the $k+l$-tuples $(x_1,\ldots,x_{k+l})$ such that for  $j=1,\ldots, k$, $x_j=g^{j-1}(x)$, $x_k=g(x_{k+l})$ and for $j=k+1,\ldots k+l-1$, $g(x_j)=x_{j+1}$.
Note that $M(x,k,l)$ is finite for all $x$, because $g$ is a local homeomorphism, and that $x_1=x$ for $(x_1,\ldots, x_{k+l})\in M(x,k,l)$. 

Define $M^0(x,k,l)\subseteq M(x,k,l)$ as the $k+l$-tuples $(x_1,\ldots,x_{k+l})$ where $x_{k+1}=x_1$. The set $M^0(x,k,l)$ can be characterized as the $k+l$-tuples $(x_1,\ldots,x_{k+l})$ with $x=x_1=x_{k+1}$ and $x_k=g^l(x_{k+l})$ such that for $j=1,\ldots, k$, $x_j=g^{j-1}(x)$, and for $j=k+1,\ldots k+l-1$, $x_{j+1}=g(x_j)$. In particular, if $M^0(x,k,l)$ is non-empty then $g^k(x)=g^l(x)$.
In other words, $(x_1,\ldots,x_{k+l})\in M^0(x,k,l)$ if and only if $g^k(x)=g^l(x)$, and $x_j=g^{j-1}(x)$ for $j=1,\ldots, k$ and $x_{k+j}=g^{j-1}(x)$ for $j=1,\ldots, l$.
Therefore, $M^0(x,k,l)$ contains at most one element.
In particular, since $M^0(x,k,l)$ is non-empty then $g^k(x)=g^l(x)$ and the set of fixed points $\{x\in M: g^m(x)=x\}$ has measure zero for all $m\in \N_+$ by assumptions, then 
$$
M^0(x,k,l)=\emptyset\quad\mbox{ if $k\neq l$ a.e. in $x$.}
$$ 

As $t$ approaches $n/2$, the integral kernel $K_t$ localizes (up to lower order term) to the diagonal and the leading order terms come from the sum over $M^0(x,k,l)$. The Weyl law for $\slashed{D}^2$ and an explicit pseudo-differential computation of the principal symbol of $P_\slashed{D}(1+\slashed{D}^2)^{-t}$ implies that there is a constant $c$ and an $\epsilon>0$ only depending on $\slashed{D}$ such that 
\small
\begin{align*}
\Tr_{L^2(M,S)}&(P_\slashed{D}S_\mu S_\nu^*\e^{-t|\slashed{D}_{\rm log}|})=\\
=&c\delta_{k,l}(t-n/2)^{-1}\int_{M}\sum_{(x_1,\ldots,x_{k+l})\in M^0(x,k,l)}\left(\prod_{j=1}^{k} \tilde{a}_j(x_j)\right)\left(\prod_{j=k+1}^{k+l} \tilde{b}_{j-k}(x_j)^*\right) \mathrm{d}V(x)+f_{\mu,\nu}(t),
\end{align*}
\normalsize
where $f_{\mu,\nu}$ is holomorphic on a neighborhood of the intervall $[n/2-\epsilon,n/2]$. 

Recall the notation $\tilde{a}_j:=a_jc_g^{n/4}$ and $\tilde{b}_j:=b_jc_g^{-n/4}$. 
For $k=l$ and $(x_1,\ldots,x_{k+l})\in M^0(x,k,k)$ we write 
\begin{align*}
\Big(\prod_{j=1}^{k} &\tilde{a}_j(x_j)\Big)\left(\prod_{j=k+1}^{2k} \tilde{b}_{j-k}(x_j)^*\right)=\prod_{j=1}^{k} \tilde{a}_j(g^{j-1}(x))\tilde{b}_{j}^*(g^{j-1}(x))=\\
&=\prod_{j=1}^{k} a_j(g^{j-1}(x))b_{j}^*(g^{j-1}(x))=\left([a_1b_1^*][g^*(a_2b_2^*)][(g^2)^*(a_2b_2^*)]\cdots [(g^k)^*(a_kb_k^*)]\right)(x).
\end{align*}
By the same argument that $V_g^*=N^{-1/2}\mathfrak{L}_{S,g}c_g^{-n/4}$, we can partially integrate $\int_M ag^*(b)\mathrm{d}V=\int_M \mathfrak{L}_g(c_g^{n/2}a)b\mathrm{d}V$ for $a,b\in C(M)$. By partially integrating $k-1$ times we deduce that for some function $f_{\mu,\nu}$ holomorphic on a neighborhood of the intervall $[n/2-\epsilon,n/2]$. 
\small
\begin{align*}
\Tr_{L^2(M,S)}&(P_\slashed{D}S_\mu S_\nu^*\e^{-t|\slashed{D}_{\rm log}|})=\\
=&c\delta_{k,l}(t-n/2)^{-1}\int_{M}[a_1b_1^*][g^*(a_2b_2^*)][(g^2)^*(a_2b_2^*)]\cdots [(g^k)^*(a_kb_k^*)] \mathrm{d}V+f_{\mu,\nu}(t)=\\
&=c\delta_{k,l}(t-n/2)^{-1}\int_{M}\mathfrak{L}_g(c_g^{n/2}b_k^*\mathfrak{L}_g(c_g^{n/2}b_{k-1}^*\mathfrak{L}_g(\cdots \mathfrak{L}_g(c_g^{n/2}b_1^*a_1)a_2)\cdots a_{k-1})a_k) \mathrm{d}V+f_{\mu,\nu}(t)
 \end{align*}
\normalsize
We conclude that formula \eqref{formforphiomforbg} holds.

It remains to show that $\phi_\omega$ defines a KMS-state on $O_{E_g}$ for the action defined by $\gamma_t(aV_g):=c_g^{itn/2}aV_g$. Let $\tau$ denote the tracial state on $C(M)$ defined from integrating against the volume form and $L\in \End^*_{C(M)}(E_g)$ the generator of $\gamma_t$, i.e. $L=\frac{n}{2}\log(c_g)$. Some yoga with inner products shows that for $\mu=a_1V_g\otimes \cdots \otimes a_k V_g, \nu=b_1V_g\otimes \cdots \otimes b_k V_g\in E_g^{\otimes_{C(M)} k}$, we have the computation
\begin{align*}
\phi_\omega(S_\mu S_\nu^*)=\int_{M}&\mathfrak{L}_g(c_g^{n/2}b_k^*\mathfrak{L}_g(c_g^{n/2}b_{k-1}^*\mathfrak{L}_g(\cdots \mathfrak{L}_g(c_g^{n/2}b_1^*a_1)a_2)\cdots a_{k-1})a_k) \mathrm{d}V=\\
&=\tau\left(b_1V_g\otimes \cdots \otimes b_k V_g|\e^{-L}(a_1V_g)\otimes \cdots \otimes \e^{-L}(a_k V_g)\right)_{E_g^{\otimes_{C(M)} k}}=\phi_\omega(S_\nu^*\gamma_i(S_\mu)).
\end{align*}
We conclude that $\phi_\omega(ab)=\phi_\omega(b\gamma_i(a))$ for $a,b\in \A$ and $\phi_\omega$ is a KMS-state on $O_{E_g}$ in the action $\gamma$.
\end{proof}

\subsection{Graph $C^*$-algebras}
\label{diracgraphkms}

For a finite graph $G$ we consider the spectral triple on $C^*(G)$ constructed in 
Proposition \ref{localizeckdirinpoint} from the choice of a point $y\in \Omega_G$ 
in the infinite path space. 
We will assume that $G$ is primitive, in which case $C^*(G)$ is simple. 
For an element $(x,n)\in \mathcal{V}_y$ and a finite path $\mu$, we compute that
$$\e^{is\D_y}P_{\D_y} S_\mu P_{\D_y}\e^{-is\D_y}\delta_{(x,n)}=
\e^{is|\mu|}P_{\D_y} S_\mu P_{\D_y}\delta_{(x,n)}.$$
It follows that $\sigma_t(S_\mu S_\nu^*)=\e^{is(|\mu|-|\nu|)}S_\mu S_\nu^*$ and 
$\sigma$ coincides with the gauge action on the graph $C^*$-algebra $C^*(G)$. 
We can conclude that $C^*(G)=C^*(G)_{\D_y}$ is closed under the flow $\sigma$. 
The reader can recall from Example \ref{diracgraphhea} (see page \pageref{diracgraphhea}) 
that $P_{\D_y}\ell^2(\mathcal{V}_y)=\ell^2(\mathcal{V}_y^+)$ where $\mathcal{V}_y^+$ 
is defined in Equation \eqref{vplusdef}.
Moreover, $T_{C^*(G)}\subseteq \mathbb{B}(\ell^2(\mathcal{V}_y^+))$ is the 
$C^*$-algebra generated by the creation operators 
$$T_e \delta_{(x,n)}:=\begin{cases}
\delta_{(ex,n)},\; &r(e)=s(x),\\
0,\; &r(e)\neq s(x).\end{cases}$$

\begin{prop}
Let $\beta\in \R$. The set $S=\{S_e: e\in E\}\subseteq C_c(\mathcal{G}_G)$ is an analytically generating set at $\beta$ for $(C_c(\mathcal{G}_G),\ell^2(\mathcal{V}_y),\D_y)$.
\end{prop}

\begin{proof}
Since $P_{\D_y}S_eP_{\D_y}=T_e$, it is clear that $S$ satisfies that $P_{\D_y} SP_{\D_y}+\mathbb{K}(\ell^2(\mathcal{V}_y^+)$ generates the Toeplitz algebra $T_{C^*(G)}$ as a $C^*$-algebra. Moreover, $\e^{\beta\D_y}P_\D S_eP_\D\e^{-\beta\D_y}=\e^{\beta}P_\D S_eP_\D$ is bounded and the proposition follows.
\end{proof}

We conclude the following theorem from Example \ref{diracgraphhea}.

\begin{thm}
\label{kmsforgrpah}
Let $G$ be a finite primitive graph with edge adjacency matrix $A$ and $y\in \Omega_G$. 
For any extended limit $\omega$ as $t\to\log r_\sigma(A)$, 
the KMS-state $\phi_\omega$ on $C^*(G)$ associated with 
the spectral triple $(C_c(\mathcal{G}_G),\ell^2(\mathcal{V}_y),\D_y)$ 
(see Proposition \ref{localizeckdirinpoint} on page \pageref{localizeckdirinpoint}) 
as in Corollary \ref{cor:phi-omega} (see page \pageref{cor:phi-omega}) is given by 
$$\phi_\omega(S_\mu S_\nu^*)=\delta_{\mu,\nu}\frac{w_{s(\mu)}}{\|w\|_{\ell^1}} r_\sigma(A)^{-|\mu|},$$
where $\nu$ and $\mu$ are finite paths and 
$w\in \C^E$ is the $\ell^2$-normalized Perron-Frobenius vector. 
The state $\phi_\omega$ is KMS for the gauge action and 
its inverse temperature is $\log(r_\sigma(A))$.
\end{thm}

The KMS-state $\phi_\omega$ on $C^*(G)$ in 
Theorem \ref{kmsforgrpah} is the unique KMS-state by \cite{EFW}. 
Numerous authors present constructions of this state and for
more general graphs, eg \cite{aHLRS, aHR, CT, KPgraph}.

\subsection{Group $C^*$-algebras}
\label{diracgroupkms}

For the reduced group $C^*$-algebra of a countable discrete group we considered two types of (semifinite) spectral triples in Subsection \ref{groupcstarexam} (see page \pageref{groupcstarexam}). We now compute the associated KMS-states. 

We fix a length function $\ell$ on the discrete countable group $\Gamma$.
For technical simplicity, we assume that $\Gamma$ is an exact group 
ensuring that $\Gamma$ acts amenably on its Stone-Cech boundary $\partial_{SC}\Gamma$ (see \cite{ozawexact}).
We assume that $(\Gamma,\ell)$ is of at most exponential growth and that $\ell$ is critical (see Definition \ref{criticaldefn} on page \pageref{criticaldefn}). 
Let $(c_b(\Gamma)\rtimes^{\rm alg} \Gamma, \ell^2(\Gamma), \D_\ell)$ denote the associated 
$\mathrm{Li}_1$-summable spectral triple as in Proposition \ref{spectripfromlength}. 
Since $\D_\ell\geq 0$, we have that 
$$T_{c_b(\Gamma)\rtimes_r \Gamma}=c_b(\Gamma)\rtimes_r \Gamma+\Ko(\ell^2(\Gamma))=c_b(\Gamma)\rtimes_r \Gamma.$$ 
The last equality follows from that $c_0(\Gamma)\rtimes_r\Gamma=\Ko(\ell^2(\Gamma))$. We conclude that we have a short exact sequence
$$0\to \Ko(\ell^2(\Gamma))\to T_{c_b(\Gamma)\rtimes_r \Gamma}\to C(\partial_{SC}\Gamma)\rtimes_r \Gamma\to 0.$$
The flow $\sigma^+$ on $T_{c_b(\Gamma)\rtimes_r \Gamma}=c_b(\Gamma)\rtimes_r \Gamma$ is given on an element $a\lambda_g\in c_b(\Gamma)\rtimes \Gamma$ by 
\begin{equation}
\label{sigmapluonsc}
\sigma^+_s(a\lambda_g)=\e^{is(\ell(\cdot)-\ell(g^{-1}\cdot)}a\lambda_g,
\end{equation}
and we conclude that $T_{c_b(\Gamma)\rtimes_r \Gamma}$ is invariant under $\sigma^+$. Therefore $T_{c_b(\Gamma)\rtimes_r \Gamma}=T_{c_b(\Gamma)\rtimes_r \Gamma,\D_\ell}$.

\begin{prop}
Let $\beta\in \R$. The $*$-algebra $c_b(\Gamma)\rtimes^{\rm alg} \Gamma$ is an analytically generating set at $\beta$ for $(c_b(\Gamma)\rtimes^{\rm alg} \Gamma, \ell^2(\Gamma), \D_\ell)$.
\end{prop}

\begin{proof}
Note that $P_{\D_\ell}=1$ because $\D_\ell$ is positive. For $a\lambda_g\in c_b(\Gamma)\rtimes^{\rm alg} \Gamma$, we compute that 
$$\e^{\beta\D_\ell}a\lambda_g\e^{-\beta\D_\ell}=\e^{-\beta(\ell(\cdot)-\ell(g^{-1}\cdot))}a\lambda_g.$$
Since $\|\e^{-\beta(\ell(\cdot)-\ell(g^{-1}\cdot)}a\|_{c_b(\Gamma)}\leq \e^{|\beta|\ell(g^{-1})}\|a\|_{c_b(\Gamma)}$, it holds that $\e^{\beta\D_\ell}a\lambda_g\e^{-\beta\D_\ell}\in c_b(\Gamma)\rtimes^{\rm alg} \Gamma$ and the proposition follows. 
\end{proof}

\begin{prop}
Let $\beta\in \R$. The $*$-algebra $c_b(\Gamma)\rtimes^{\rm alg} \Gamma$ is an analytically generating set at $\beta$ for  $(c_b(\Gamma)\rtimes^{\rm alg} \Gamma, \ell^2(\Gamma, S_\H), \D_c, \cN, \Tr_\tau)$.
\end{prop}

\begin{proof}
For $a\lambda_g\in c_b(\Gamma)\rtimes^{\rm alg} \Gamma$, we compute for $f\in \ell^2(\Gamma,S_\H)$ that 
$$\e^{\beta|\D_c|}\hat{\pi}_S(a\lambda_g)\e^{-\beta|\D_c|}f(\gamma)=\e^{-\beta(\ell(\gamma)-\ell(g^{-1}\gamma))}a(\gamma)[\tilde{\pi}(g)f](\gamma).$$
The estimate $\|\e^{-\beta(\ell(\cdot)-\ell(g^{-1}\cdot)}a\|_{c_b(\Gamma)}\leq \e^{|\beta|\ell(g^{-1})}\|a\|_{c_b(\Gamma)}$ shows that $\e^{\beta|\D_c|}\hat{\pi}_S(a\lambda_g)\e^{-\beta|\D_c|}\in \hat{\pi}\left(c_b(\Gamma)\rtimes^{\rm alg} \Gamma\right)$. Therefore, 
\begin{align*}
&\e^{\beta\D_c}P_{\D_c} (c_b(\Gamma)\rtimes^{\rm alg} \Gamma) P_{\D_c}\e^{-\beta\D_c}\\&=P_{\D_c}\e^{\beta|\D_c|} (c_b(\Gamma)\rtimes^{\rm alg} \Gamma)\e^{-\beta|\D_c|}P_{\D_c}\subseteq P_{\D_c} (c_b(\Gamma)\rtimes^{\rm alg} \Gamma) P_{\D_c}\subseteq \cN^+,
\end{align*} 
and the proposition follows.
\end{proof}

\begin{defn}
If $\omega$ is an extended limit as $t\to\beta(\Gamma,\ell)$, and $\ell$ is critical, we define the Patterson-Sullivan measure $\mu_\omega$ on 
the Stone-\v{C}ech boundary $\partial_{SC}\Gamma$ as
$$
\int_{\partial_{SC}\Gamma} a\,\mathrm{d}\mu_\omega:=\lim_{t\to \omega}\frac{\sum_{\gamma\in \Gamma} \tilde{a}(\gamma)\e^{-t\ell(\gamma)}}{\sum_{\gamma\in \Gamma} \e^{-t\ell(\gamma)}},
$$
for a function $a\in C(\partial_{SC}\Gamma)$ and where $\tilde{a}\in c_b(\Gamma)$ is any function with $a=\tilde{a}\mod c_0(\Gamma)$.
\end{defn}

\begin{rmk}
It is possible to define the Patterson-Sullivan measure $\mu_\omega$ as an extended weak*-limit of the family of probability measures on $\Gamma$
$$\mu_t=\frac{\sum_{\gamma\in \Gamma} \delta_\gamma\e^{-t\ell(\gamma)}}{\sum_{\gamma\in \Gamma} \e^{-t\ell(\gamma)}}.$$
In the literature, Patterson-Sullivan measures are usually defined as 
weak* accumulation points of $(\mu_t)_{t>\beta(\Gamma, \ell)}$ but 
we allow for a slightly more general construction with extended limits. 
A priori, $\mu_\omega$ is a probability measure on the Stone-\v{C}ech compactification of $\Gamma$. 
Since the support of $\mu_\omega$ is contained in the closed subspace 
$\partial_{SC}\Gamma$ we consider 
$\mu_\omega$ as a measure on  $\partial_{SC}\Gamma$.
\end{rmk}

\begin{thm}
\label{themforgamma}
Let $\Gamma$ be a discrete group and $\phi_\omega$ 
the KMS-state on $C(\partial_{SC}\Gamma)\rtimes_r \Gamma$ constructed as in Corollary \ref{cor:phi-omega} (see page \pageref{cor:phi-omega}) using an extended limit $\omega$ as $t\to\beta(\Gamma,\ell)$ and either of the following two semifinite spectral triples:
\begin{itemize}
\item The spectral triple 
$$(c_b(\Gamma)\rtimes^{\rm alg} \Gamma, \ell^2(\Gamma), \D_\ell)$$
associated with a critical  length function of at most exponential growth as in Proposition \ref{spectripfromlength} (see page \pageref{spectripfromlength}).
\item The semifinite spectral triple 
$(c_b(\Gamma)\rtimes^{\rm alg} \Gamma, \ell^2(\Gamma, S_\H), \D_c, \cN, \Tr_\tau)$ 
associated with a critical proper Hilbert space valued cocycle of at most exponential growth 
as in Proposition \ref{cssfst} (see page \pageref{cssfst}).
\end{itemize}
Then $\phi_\omega$ is given in terms of the Patterson-Sullivan measure $\mu_\omega$ by 
$$
\phi_\omega(a\lambda_g)=\delta_{e,g}\int_{\partial_{SC}\Gamma} a\,\mathrm{d}\mu_\omega.
$$
The state $\phi_\omega$ is KMS at inverse temperature $\beta(\Gamma,\ell)$ for the flow on $C(\partial_{SC}\Gamma)\rtimes_r \Gamma$ induced by the action $\sigma^+$ on $c_b(\Gamma)\rtimes_r\Gamma$ given in Equation \eqref{sigmapluonsc}. Moreover, $\phi_\omega$ extends to a KMS-state at inverse temperature $1$ on the von Neumann algebra $L^\infty(\partial_{SC}\Gamma,\mu_\omega)\overline{\rtimes}\Gamma$ with its Radon-Nikodym flow
$$\sigma^{RN}_s(a\lambda_g)=\left(\frac{\mathrm{d}g_*\mu_{\omega}}{\mathrm{d}\mu_{\omega}}\right)^{is} a\lambda_g.$$
\end{thm}

\begin{proof}
By the computations of Example \ref{diracgrouphea} (see page \pageref{diracgrouphea}), the spectral triple associated with a length function as in Proposition \ref{spectripfromlength} have the same heat traces as the semifinite spectral triple associated with a proper Hilbert space valued cocycle as in Proposition \ref{cssfst}. In both cases, Example \ref{diracgrouphea} shows that for $\tilde{a}\lambda_g\in c_b(\Gamma)\rtimes^{\rm alg}\Gamma$ we have 
$$
\phi_{t,0}(\tilde{a}\lambda_g)=\delta_{e,g}\frac{\sum_{\gamma\in \Gamma} \tilde{a}(\gamma)\e^{-t\ell(\gamma)}}{\sum_{\gamma\in \Gamma} \e^{-t\ell(\gamma)}}.
$$
It follows that $\phi_\omega(a\lambda_g)=\delta_{e,g}\int_{\partial_{SC}\Gamma} a\,\mathrm{d}\mu_\omega$ in both cases.

To relate $\phi_\omega$ to the Radon-Nikodym flow, we first show that $\mu_\omega$ is strictly positive, i.e. that $\mu_\omega(U)>0$ for any open set $U\subset \partial_{SC}\Gamma$. For any open set $U\subset \partial_{SC}\Gamma$, the translates $(\gamma U)_{\gamma\in \Gamma}$ cover $\partial_{SC}\Gamma$. If $\mu_\omega(U)=0$, then by quasi-invariance $\mu_\omega(\gamma U)=0$ which contradicts $\mu_\omega$ being a probability measure.

The fact that $\mu_{\omega}$ is strictly positive ensures that the Radon-Nikodym derivatives $\frac{\mathrm{d}g_*\mu_{\omega}}{\mathrm{d}\mu_{\omega}}$ are well defined and strictly positive.
The mapping $g\mapsto \frac{\mathrm{d}g_*\mu_{\omega}}{\mathrm{d}\mu_{\omega}}$ is a cocycle, i.e. for $h,g\in \Gamma$, 
$$\frac{\mathrm{d}(gh)_*\mu_{\omega}}{\mathrm{d}\mu_{\omega}}=\frac{\mathrm{d}g_*\mu_{\omega}}{\mathrm{d}\mu_{\omega}}(g^{-1})^*\left[\frac{\mathrm{d}h_*\mu_{\omega}}{\mathrm{d}\mu_{\omega}}\right].$$
The proof is completed by computing that for $a\lambda_g, b\lambda_h\in C(\partial_{SC}\Gamma)\rtimes^{\rm alg}\Gamma$, we have the identity
\begin{align*}
\phi_\omega(a\lambda_gb\lambda_h)&=\delta_{h,g^{-1}} \phi_\omega(ah^*(b))=\delta_{h,g^{-1}}\int_{\partial_{SC}\Gamma} ah^*(b)\mathrm{d}\mu_{\omega}=\\
&=\delta_{h,g^{-1}}\int_{\partial_{SC}\Gamma} h^*\left(g^*(a)b\right)\mathrm{d}\mu_{\omega}=\delta_{h,g^{-1}}\int_{\partial_{SC}\Gamma} g^*(a)b\mathrm{d}(h_*\mu_{\omega})=\\
&=\delta_{h,g^{-1}}\int_{\partial_{SC}\Gamma} bg^*(a)\frac{\mathrm{d}h_*\mu_{\omega}}{\mathrm{d}\mu_{\omega}}\mathrm{d}\mu_{\omega}=\\
&=\delta_{h,g^{-1}}\int_{\partial_{SC}\Gamma} b(h^{-1})^*\left(a\left(\frac{\mathrm{d}g_*\mu_{\omega}}{\mathrm{d}\mu_{\omega}}\right)^{-1}\right)\mathrm{d}\mu_{\omega}=\\
&=\phi_\omega\left(b\lambda_h\left(\frac{\mathrm{d}g_*\mu_{\omega}}{\mathrm{d}\mu_{\omega}}\right)^{-1}a\lambda_g\right)=\phi_\omega\left(b\lambda_h\sigma^{RN}_{s=i}(a\lambda_g)\right)
\end{align*}
In the third last identity we used the cocycle identity implying that if $hg=e$, then 
\[
(h^{-1})^*\left(\frac{\mathrm{d}g_*\mu_\omega}{\mathrm{d}\mu_\omega}\right)\frac{\mathrm{d}h_*\mu_\omega}{\mathrm{d}\mu_\omega}=1.\qedhere
\]
\end{proof}

\begin{rmk}
The reader should note that $\phi_\omega|_{C^*_r(\Gamma)}$ coincides with the $\ell^2$-trace.
\end{rmk}

\section{KMS-states on Cuntz-Pimsner algebras with their gauge action}
\label{diraccpkms}

In this section we consider the constructions of Corollary \ref{cor:phi-omega} in a broad class 
of examples which include both Cuntz-Krieger algebras and crossed products by $\Z$.
Here we use the techniques from Section \ref{sec:KMS} 
in conjunction with those from Subsection \ref{cpalgexam} to analyse
the KMS states on Cuntz-Pimsner algebras, and compare them to the Laca-Neshveyev correspondence establishing a bijection between KMS-states on Cuntz-Pimsner algebras and 
tracial states on its coefficient algebra.  

\subsection{KMS-states on Cuntz-Pimsner algebras from traces on the coefficient algebra}

Firstly, we shall show that a critical positive trace on $A$ (see Definition \ref{defn:critical} below) gives rise to a KMS-state on the Cuntz-Pimsner algebra 
$O_E$ assuming that $E$ is strictly W-regular. Recall Lemma \ref{ximodsemi} (see page \pageref{ximodsemi}) giving the construction of the $\mathrm{Li}_1$-summable semifinite spectral triple $(\O_E,L^2(\phimod,\tau),\D_\psi,\cN_\tau(\Xi_A),\Tr_\tau)$, where $\cN_\tau(\Xi_A):=(\End^*_A(\phimod)\otimes_A1)''$. 

\begin{lemma}
\label{betacrit}
Let $E$ be a strictly W-regular fgp bi-Hilbertian bimodule over $A$, $\beta\in \R$ and $\tau$ a positive trace on $A$.
The set $S=\{S_{e}:\,e\in E\}\subset\O_E$ is an analytically generating set at $\beta$ for $(\O_E,L^2(\phimod,\tau),\D_\psi,\cN_\tau(\Xi_A),\Tr_\tau)$.
\end{lemma}

\begin{proof}
Since $T_{\O_E,\D}$ is precisely the Toeplitz algebra
$T_E$, it is immediate that the set of operators $\{ PS_eP : e\in E\}$
generates $T_{\O_E,\D}$. The analyticity condition
on the Fock space is likewise obvious from the computation 
\[
\e^{\beta\D_\psi}P_{\D_\psi}S_eP_{\D_\psi}\e^{-\beta\D_\psi}=
\e^\beta P_{\D_\psi}S_eP_{\D_\psi}.\qedhere
\]
\end{proof}

\begin{defn}
\label{defn:critical}
Let $E$ be an $A$-bimodule which is fgp from the right and $\tau$ a positive trace on $A$. 
We define the critical value of $(E,\tau)$ as
$$\beta(E,\tau):=\inf\{t\geq 0: \sum_{n=0}^\infty  \tau_*(E^{\otimes_An})\e^{-tn}<\infty\}.$$
We say that $\tau$ is critical for $E$ if 
$$\lim_{t\searrow \beta(E,\tau)} \sum_{n=0}^\infty  \tau_*(E^{\otimes_An})\e^{-tn}=\infty.$$
\end{defn}

\begin{rmk}
Note that the critical value of a trace and the notion of it being critical only depends on $[E]\in KK_0(A,A)$, in fact only on the sequence $(\mathrm{ch}_0(E^{\otimes_An}))_{n\in \N}\subseteq HC_0(A)$ in cyclic homology. Just as in the proof of Lemma \ref{ximodsemi} we obtain the estimate $0\leq \beta(E,\tau)\leq \log(N)$, where $N$ is the number of elements in the left frame and the right frame.

It follows from Definition~\ref{ass:minus-one} and Proposition~\ref{posheatcomp} that the $\mathrm{Li}_1$-summable semifinite spectral triple $(\O_E,L^2(\phimod,\tau),\D_\psi,\cN_\tau(\Xi_A),\Tr_\tau)$ of Lemma \ref{ximodsemi} has positive $\Tr_\tau$-essential spectrum if and only if $\tau$ is critical.
\end{rmk}

By construction, the projection $P_{\D_\psi}$ is the projection onto the Fock module $\Fock$ and therefore $\cN_\tau(\Xi_A)^+= (\End^*_A(\Fock)\otimes_A1)''$ is a subalgebra of $\mathbb{B}(L^2(\Fock,\tau))$. We let $N$ denote the number operator on $L^2(\Fock,\tau)$ --  the self-adjoint operator defined from $N|_{E^{\otimes n}\otimes_AL^2(A,\tau)}=n\mathrm{Id}_{E^{\otimes n}\otimes_AL^2(A,\tau)}$. The next proposition follows from the definition of the Toeplitz algebra of a semifinite spectral triple. 

\begin{prop}
\label{toepliz}
The Toeplitz algebra $T_{O_E}$ of the semifinite spectral triple 
$$(\O_E,L^2(\phimod,\tau),\D_\psi,\cN_\tau(\Xi_A),\Tr_\tau)$$ 
is given by 
$$T_{O_E}=\mathcal{T}_E\otimes_A 1_A+\Ko_{(\End^*_A(\Fock)\ox_A1)''},$$
where $\mathcal{T}_E\subseteq \End^*_A(\Fock)$ is the Cuntz-Toeplitz algebra of $E$. Moreover, the action $\sigma^+$ preserves $T_{O_E}$ and is generated by the number operator in the sense that for $\mu\in E^k$, $\nu\in E^{\otimes l}$ and $K\in \Ko_{(\End^*_A(\Fock)\ox_A1)''}$, we have 
$$\sigma^+_s(T_\mu T_\nu^*+K)=\e^{is(|\mu|-|\nu|)}T_\mu T_\nu^*+\e^{isN}K\e^{-isN}.$$
\end{prop}

An immediate consequence of Proposition \ref{toepliz} is that $O_E=T_{O_E}/ \Ko_{(\End^*_A(\Fock))''}$ and that the action $\sigma$ on $\O_E$ coincides with the gauge action $\sigma_s(S_\mu S_\nu^*)=\e^{is(|\mu|-|\nu|)}S_\mu S_\nu^*$. The following theorem is readily deduced from the computations of Example \ref{diraccphea} (see page \pageref{diraccphea}).

\begin{thm}
\label{kmscp}
Let $E$ be a strictly W-regular fgp bi-Hilbertian bimodule over $A$.
For any positive trace $\tau$ on $A$ which is critical for $E$ and any extended limit $\omega$ as $t\to\beta(E,\tau)$, 
the KMS-state $\phi_\omega$ on $O_E$ associated with 
the semifinite spectral triple $(\O_E,L^2(\phimod,\tau),\D_\psi,\cN_\tau(\Xi_A),\Tr_\tau)$ 
as in Theorem \ref{cor:phi-omega} (see page \pageref{cor:phi-omega}) is given by 
$$\phi_{\tau,\omega}(S_\mu S_\nu^*)=\delta_{|\mu|,|\nu|}\e^{-\beta(E,\tau)|\mu|}\lim_{t\to \omega}\frac{\sum_{n=0}^\infty\Tr^{E^{\ox n}}_\tau\left((\nu|\mu)_{E^{|\mu|}}\right)\e^{-tn}}{\sum_{n=0}^\infty \Tr^{E^{\ox n}}_\tau(1)\e^{-tn}} ,$$
where $\nu\in E^{\otimes k}$ and $\mu\in E^{\otimes l}$. 
The state $\phi_{\tau, \omega}$ is KMS for the gauge action on $O_E$ and 
its inverse temperature is $\beta(E,\tau)$.
\end{thm}

\begin{proof}
By Lemma \ref{betacrit} the semifinite spectral triple $(\O_E,L^2(\phimod,\tau),\D_\psi,\cN_\tau(\Xi_A),\Tr_\tau)$ is $\beta$-analytic for any $\beta\in \R$. By Definition~\ref{ass:minus-one} and Proposition~\ref{posheatcomp} it has positive $\Tr_\tau$-essential spectrum. Thus the state $\phi_{\tau, \omega}$ as in Theorem \ref{cor:phi-omega} is KMS for the gauge action on $O_E$ with inverse temperature $\beta(E,\tau)$. Using the computation in Equation \eqref{tracpcocomcme}, we see that 
$$\phi_{\tau,\omega}(S_\mu S_\nu^*)=\delta_{|\mu|,|\nu|}\lim_{t\to \omega}\frac{\sum_{n=|\mu|}^\infty\sum_{|\sigma|=n} \e^{-tn}\tau \big((\,(\mu|e_{\underline{\sigma}})_{E^{|\mu|}}\,e_{\overline{\sigma}}\,|\,(\nu|e_{\underline{\sigma}})_{E^{|\mu|}}\,e_{\overline{\sigma}})_{E^{\otimes (n-|\mu|)}}\big)}{\sum_{n=0}^\infty \sum_{|\rho|=n} \tau((e_\rho|e_\rho)_A)\e^{-tn}} ,$$
for $\nu\in E^{\otimes k}$ and $\mu\in E^{\otimes l}$. 
However, this expression can be vastly simplified using that $\phi_{\tau,\omega}$ is KMS. Using
$$
\phi_{\tau,\omega}(S_\mu S_\nu^*)=\delta_{|\mu|,|\nu|}\e^{-\beta(E,\tau)|\mu|}\phi_{\tau,\omega}(S_\nu^*S_\mu)=\delta_{|\mu|,|\nu|}\e^{-\beta(E,\tau)|\mu|}\phi_{\tau,\omega}((\nu|\mu)_{E^{\otimes |\mu|}}),
$$
we obtain
\begin{align*}
\phi_{\tau,\omega}(S_\mu S_\nu^*)&=\delta_{|\mu|,|\nu|}\e^{-\beta(E,\tau)|\mu|}\phi_\omega((\nu|\mu)_{E^{\otimes |\mu|}})\\
&=\delta_{|\mu|,|\nu|}\e^{-\beta(E,\tau)|\mu|}\lim_{t\to \omega}\frac{\sum_{n=0}^\infty\sum_{|\sigma|=n} \e^{-tn}\tau \big((\,e_{\sigma}\,|\,(\nu|\mu)_{E^{|\mu|}}\,e_{\sigma})_{E^{\otimes n}}\big)}{\sum_{n=0}^\infty \sum_{|\rho|=n} \tau((e_\rho|e_\rho)_A)\e^{-tn}}\\
&=\delta_{|\mu|,|\nu|}\e^{-\beta(E,\tau)|\mu|}\lim_{t\to \omega}\frac{\sum_{n=0}^\infty\Tr^{E^{\ox n}}_\tau\left((\nu|\mu)_{E^{|\mu|}}\right)\e^{-tn}}{\sum_{n=0}^\infty \Tr^{E^{\ox n}}_\tau(1)\e^{-tn}}.
\end{align*}
\end{proof}

The reader should note that in the formula computing $\phi_{\tau,\omega}$, it is only the right inner product on $E$ that appears.

\begin{example}
\label{localhomeoex}
Let us consider the construction from Theorem \ref{kmscp} in a specific example. As in Example \ref{localhomeofirst} (see page \pageref{localhomeofirst}), we consider a compact Hausdorff space $Y$, a surjective local homeomorphism $g:Y\to Y$ and the associated bimodule $E_g$. 
Let us compute $\phi_{\tau,\omega}$ starting from a positive trace $\tau$ on $C(Y)$, i.e. a positive measure $\lambda$ on $Y$. By the argument of Theorem \ref{kmscp}, the KMS-condition on $\phi_{\tau,\omega}$ guarantees that it suffices to describe $\phi_{\tau,\omega}(a)$ for $a\in C(Y)$. By the Riesz representation theorem, 
$$
\phi_{\tau,\omega}(a)=\int_Y a\,\mathrm{d}\lambda_\omega,
$$
for a probability measure $\lambda_\omega$. 
We compute that for $a\in C(Y)$,
$$
\Tr^{E^{\ox n}}_\tau(a)=\sum_{|\sigma|=n}\tau( (e_\sigma|ae_\sigma)_{C(Y)})
=\int_Y \mathfrak{L}_g^n(a)\,\mathrm{d}\lambda
=\int_Y a\,\mathrm{d}[(\mathfrak{L}_g^n)_*\lambda].$$
From Theorem \ref{kmscp}, we conclude that $\lambda_{\omega}$ is given by an extended weak* limit of measures
$$\lambda_\omega=\lim_{t\to \omega}\frac{\sum_{n=0}^\infty \e^{-tn}(\mathfrak{L}_g^n)_*\lambda}{\sum_{n=0}^\infty \e^{-tn}[(\mathfrak{L}_g^n)_*\lambda](Y)}.$$
The KMS-condition on $\phi_{\tau,\omega}$ translates into $(\mathfrak{L}_g)_*\lambda_\omega=\e^{\beta(E,\tau)}\lambda_\omega$ which is readily verified for the measure $\lambda_\omega$.

We remark that the construction above is reminiscent of the method in \cite{waltersruelle} to construct equilibrium measures. In the case that $g$ is mixing, i.e. for all open subsets $U,V\subseteq Y$ there is an $N\geq 0$ such that $g^n(U)\cap V\neq \emptyset$ for all $n\geq N$, then there exists a unique KMS-state on $O_{E_g}$ (see \cite[Theorem 6.1]{DGMW}). In particular, for mixing $g$, the KMS-state $\phi_{\tau,\omega}$ on $O_{E_g}$ does not depend on the choice of trace $\tau$.
\end{example}

\subsection{The Toeplitz construction vs the Laca-Neshveyev correspondence}
\label{subsec:T-vs-LN}

In the previous subsection we saw that there is a mapping from the set of positive critical traces 
on a $C^*$-algebra $A$ to the set of KMS-states on the Cuntz-Pimsner algebra 
$O_E$ when $E$ is strictly W-regular. 
As Example \ref{localhomeoex} shows, this mapping is not injective in general, but 
it is surjective in some cases (e.g. when $g$ is mixing). 
We now compare our construction to the bijective correspondence between 
a certain set of tracial states on $A$ with KMS-states on the Cuntz-Pimsner algebra $O_E$ 
first discovered by Laca-Neshveyev \cite{LN}.

\begin{defn}
\label{ass:3.5}
The positive trace $\tau:A\to\C$ satisfies the Laca-Neshveyev condition for $\alpha\geq 0$ if 
$$
\Tr^E_\tau(L_a)=\e^\alpha\tau(a),
$$
where $L_a$ denotes the left action of $a$ on $E$.
\end{defn}

For notational simplifity we often write $\Tr^E_\tau(a)$ instead of $\Tr^E_\tau(L_a)$. Given a positive trace $\tau:A\to\C$ satisfying the Laca-Neshveyev condition, it was proven by Laca-Neshveyev \cite[Theorem 2.1 and 2.5]{LN} that the expression
$$
\phi_{LN,\tau}(S_\mu S_\nu^*):=
\delta_{|\mu|,|\nu|}\mathrm{e}^{-\alpha|\mu|}\tau((\nu|\mu)_A), 
$$
defines an $\alpha$-KMS state on $O_E$. Moreover, Laca-Neshveyev proved that the construction $\tau\mapsto \phi_{LN,\tau}$ is a bijection between tracial states on $A$ satisfying the Laca-Neshveyev condition for $\alpha\geq 0$ and $\alpha$-KMS states on $O_E$.

The work of Laca-Neshveyev \cite{LN} gives more context to the construction in Theorem \ref{kmscp} (see page \pageref{kmscp}). For a unital $C^*$-algebra $A$, we let $\mathfrak{T}(A)$ denote the set of positive traces on $A$. If $E$ is an $A-A$-correspondence which is finitely generated and projective as a right module, we can also define $\mathfrak{CT}_{E,\alpha}(A)$ as the set of positive critical traces $\tau$ with $\beta(E,\tau)=\alpha$. We also define $\mathfrak{LN}_{E,\alpha}(A)$ as the set of positive traces satisfying the Laca-Neshveyev condition.

Following \cite[Discussion proceeding Definition 2.3]{LN}, we define 
\begin{equation}\label{eq:F}
F_{E,\alpha}:\mathfrak{T}(A)\to \mathfrak{T}(A), \quad F_{E,\alpha}\tau(a)=\Tr^E_\tau(a)\e^{-\alpha}.
\end{equation}

\begin{prop}
\label{tracecompforf}
For any positive trace $\tau$ on a unital $C^*$-algebra and an $A$-$A$-correspondence $E$ which is fgp from the right, it holds that 
$$F^n_{E,\alpha}\tau(a)=\e^{-\alpha n}\Tr^{E^{\ox n}}_\tau(a), \quad n\in \N_+.$$
\end{prop}

\begin{proof}
By definition, $F_{E,\alpha}\tau(a)=\e^{-\alpha}\sum_{j=1}^N \tau ( e_j|ae_j)_A$ for a right frame $(e_j)_{j=1}^N$ of $E$. A direct computation shows that 
\[
F^n_{E,\alpha}\tau(a)=\e^{-\alpha n}\sum_{|\sigma|=n}\tau (e_\sigma|ae_\sigma)_A=\e^{-\alpha n}\Tr^{E^{\ox n}}_\tau(a).
\qedhere
\]
\end{proof}

\begin{prop}
\label{lnequitofixed}
Let $E$ be an $A$-$A$-correspondence which is fgp from the right. Then $\tau\in \mathfrak{T}(A)$ is a fixed point of $F_{E,\alpha}$ if and only if $\tau$ satisfies the Laca-Neshveyev condition from Definition \ref{ass:3.5} (see page \pageref{ass:3.5}).
\end{prop}

Proposition \ref{lnequitofixed} is a direct consequence of Definition~\ref{ass:3.5} and the formula~\eqref{eq:F}.

\begin{prop}
\label{quasicompu}Let $E$ be an $A$-$A$-correspondence which is fgp from the right. 
If $\tau\in \mathfrak{T}(A)$ is a tracial state satisfying the Laca-Neshveyev condition for $\alpha\geq 0$, then 
$$
\tau_*(E^{\otimes_A n})=\e^{\alpha n} F_{E,\alpha}^n\tau(1)=\e^{\alpha n}.
$$
\end{prop}

\begin{proof}
The proposition follows from the computation 
$$
\tau_*(E^{\otimes_A n})=\mathrm{Tr}_\tau^{E^{\otimes n}}(1)=\e^{\alpha n}F_{E,\alpha}^n(\tau)(1),
$$
and Proposition \ref{lnequitofixed}.
\end{proof}

\begin{prop}
\label{someconforf}
Let $E$ be an $A$-$A$-correspondence which is fgp from the right and $\alpha\geq 0$. Then the following holds:
\begin{enumerate}
\item A positive trace $\tau\in \mathfrak{T}(A)$ is critical for $\alpha$ if and only if the positive trace 
$$S_{E,\alpha}^t\tau:=\sum_{n=0}^\infty \e^{-tn}F_{E,\alpha}^n\tau,$$ 
is finite for $t>0$ and satisfies $S_{E,\alpha}^t\tau(1)\to \infty$ as $t\to 0$. In particular, we have an inclusion of sets
$$\mathfrak{LN}_{E,\alpha}(A)\subseteq \mathfrak{CT}_{E,\alpha}(A).$$
\item For any extended limit $\omega$ as $t\to 0$, the mapping 
$$S_{E,\alpha}^\omega:\mathfrak{CT}_{E,\alpha}(A)\to \mathfrak{LN}_{E,\alpha}(A), \quad S_{E,\alpha}^\omega\tau:=\lim_{t\to \omega}\frac{S_{E,\alpha}^t\tau}{S_{E,\alpha}^t\tau(1)},$$
is well defined. Moreover, $(S_{E,\alpha}^\omega)^2\tau=\frac{S_{E,\alpha}^\omega\tau}{S_{E,\alpha}^\omega\tau(1)}$ and $S^\omega_{E,\alpha}$ surjects onto the set of tracial states in $\mathfrak{LN}_{E,\alpha}(A)$.
\end{enumerate}
\end{prop}

\begin{proof}
Statement 1 is an immediate consequence of Proposition \ref{quasicompu}. The first part of statement 2 follows from the computation 
$$S_{E,\alpha}^t\tau(1)-F_{E,\alpha}S_{E,\alpha}^t\tau(1)=\tau(1)=o(S_{E,\alpha}^t\tau(1)), \quad \mbox{as $t\to 0$ for $\tau \in  \mathfrak{CT}_{E,\alpha}(A)$}.$$
The second part of statement 2 follows from the fact that Proposition \ref{lnequitofixed} implies that for $\tau\in \mathfrak{LN}_{E,\alpha}(A)$, 
\[
S_{E,\alpha}^t\tau:=(1-\e^{-t})^{-1}\tau.
\qedhere
\]
\end{proof}

\begin{rmk}
For Example \ref{localhomeoex}, the mapping $F_{E,\alpha}$ takes the form $F_{E,\alpha}\tau=\e^{-\alpha}\mathfrak{L}_g^*\tau$. In particular, the computations of Example \ref{localhomeoex} is a special case of the constructions in Proposition \ref{someconforf}. We shall see that this holds in general below in Proposition \ref{restricprop}.
\end{rmk}

Using our previous results, Proposition \ref{posheatcomp} (see page \pageref{posheatcomp}) and Proposition \ref{quasicompu} (see page \pageref{quasicompu}), we can deduce a computation of heat traces.

\begin{prop}
\label{specasuln}
Let $E$ be a strictly W-regular fpg bi-Hilbertian bimodule over a unital $C^*$-algebra $A$. If $\tau$ is a tracial state on $A$ satisfying the Laca-Neshveyev condition for $\alpha\geq 0$, 
then 
$$
\Tr_\tau(P_\D\e^{-t|\D_\psi|})=\frac{1}{1-\e^{\alpha-t}}.
$$
In particular, for any tracial state $\tau$ satisfying the Laca-Neshveyev condition for $\alpha\geq 0$ the semifinite spectral triple $(\O_E,L^2(\phimod,\tau),\D_\psi,\cN_\tau(\Xi_A),\Tr_\tau)$ has positive $\Tr_\tau$-essential spectrum with $\beta_\D=\beta(E,\tau)=\alpha$.
\end{prop}

\begin{proof}
We compute that 
\begin{align*}
\Tr_\tau(P_\D\e^{-t|\D_\psi|})&=\sum_{n=0}^\infty \tau_*(E^{\otimes_A n})\e^{-tn}=\sum_{n=0}^\infty \e^{-(t-\alpha)n}=\frac{1}{1-\e^{\alpha-t}}.
\end{align*}
In the first step we used Proposition \ref{posheatcomp} and in the second step we used Proposition \ref{quasicompu}.
\end{proof}

We can now reformulate Theorem \ref{kmscp} in terms of the map $F_{E,\alpha}$ and the constructions of Proposition \ref{someconforf}.

\begin{prop}
\label{restricprop}
Assume that $E$ is a strictly W-regular fgp bi-Hilbertian bimodule over $A$. Let $\alpha\geq 0$, $\omega$ be an extended limit as $t\to \alpha$ and $\tau\in \mathfrak{CT}_{E,\alpha}(A)$ a critical trace. 
The KMS-state $\phi_{\tau,\omega}$ defined from Theorem \ref{kmscp} takes the following form:
$$\phi_{\tau,\omega}(S_\mu S_\nu^*)=\delta_{|\mu|,|\nu|} \e^{-\alpha|\mu|}S_{E,\alpha}^{\omega_\alpha}\tau((\nu|\mu)_A),\quad \mu,\nu \in \Fock^{\rm alg},$$
where $\omega_\alpha$ is the extended limit at $0$ obtained from translating $\omega$ by $\alpha$.
\end{prop}

\begin{proof}
The KMS-condition on $O_E$ reduces the proof to showing that $\phi_{\tau,\omega}(a)=S_{E,\alpha}^{\omega_\alpha}\tau(a)$ for $a\in A$, just as in the proof of Theorem \ref{kmscp}. Using Proposition \ref{tracecompforf} we can compute for $a\in A$ that
\begin{align*}
\phi_{\tau,\omega}(a)
&=\lim_{t\to \omega_0}
\frac{\sum_{n=0}^\infty \Tr_\tau^{E^{\ox n}}(a)\e^{-tn}}{\sum_{n=0}^\infty\Tr_\tau^{E^{\ox n}}(1)\e^{-tn}}=\lim_{t\to \omega}
\frac{\sum_{n=0}^\infty F_{E,\alpha}^n\tau(a)\e^{-(t-\alpha)n}}{\sum_{n=0}^\infty F_{E,\alpha}^n\tau(1)\e^{-(t-\alpha)n}}=S_{E,\alpha}^{\omega_\alpha}\tau(a).
\qedhere
\end{align*}
\end{proof}

Let $KMS_\alpha(O_E)$ denote the set of $\alpha$-KMS states on $O_E$ for the gauge action and $\mathfrak{LNS}_{E,\alpha}(A)$ for the set of tracial states satisfying the Laca-Neshveyev condition.

\begin{thm}
\label{prop:tau-phi=LN}
Let $\alpha\geq 0$ and let $\omega$ be an extended limit as $t\to \alpha$. 
Assume that $E$ is a strictly W-regular fgp bi-Hilbertian bimodule over $A$. 
The mapping 
$$\mathfrak{LNS}_{E,\alpha}(A)\to KMS_\alpha(O_E), \quad \tau\mapsto \phi_{\tau,\omega},$$ 
defined from Theorem \ref{kmscp}, and revisited in Proposition \ref{restricprop}, is a well defined bijection of sets.
More precisely, $\phi_{\tau,\omega}:O_E\to\C$ is a $KMS_\alpha$ state
for the gauge action, which is independent of 
$\omega$ and coincides with $\phi_{LN,\tau}$. 
\end{thm}

\begin{proof}
By Theorem \ref{kmscp} and Proposition \ref{specasuln}, the mapping $\tau\mapsto \phi_{\tau,\omega}$ is a well defined mapping from the set of positive traces on $A$ satisfying the Laca-Neshveyev condition for $\alpha$ and 
$\alpha$-KMS-states on $O_E$ for the gauge action. By the Laca-Neshveyev correspondence, the KMS-state $\phi_{\tau,\omega}$ is uniquely determined by the trace $\phi_{\tau,\omega}|_A$. We can therefore deduce that $\phi_{\tau,\omega}=\phi_{LN,\tau}$ and the Theorem upon proving the identity $\phi_{\tau,\omega}|_A=\tau$. This statement follows immediately from Proposition \ref{restricprop} and the second part of Proposition \ref{someconforf}.
\end{proof}

There are some quasi-invariance assumptions
on traces that allows us to compare the KMS-states constructed in 
Theorem \ref{kmscp} and the Laca-Neshveyev correspondence to a simpler construction 
involving $\Phi_\infty$.
While the construction of $\Phi_\infty$ depends on the left inner product on $E$, 
the quasi-invariance condition we impose also depends on the left inner product.
The following quasi-invariance condition is a refinement of the notion of $E$-invariant functionals from \cite{RRS}.

\begin{defn}
\label{ass:three}
Let  $\alpha\geq 0$ and $E$ a finitely generated projective bi-Hilbertian bimodule. We say that a positive trace $\tau$ on $A$ is $\alpha$-quasi-invariant with respect to $E$ and the extended limit $\omega_0\in \ell^\infty(\N)^*$ if for all $n\in \N$ and $\mu,\,\nu\in E^{\otimes n}$ we have
\begin{equation}
\e^{-\alpha |\mu|}\tau((\nu|\mu)_A)
=\lim_{k\to\omega_0}\tau(\Phi_k(T_\mu T_\nu^*)\e^{-\beta_k})
=\lim_{k\to\omega_0}\tau({}_A(\mu|\nu \e^{\beta_{k-|\nu|}})\e^{-\beta_k}).
\label{eq:q-i}
\end{equation}
If $\tau$ is $\alpha$-quasi-invariant with respect to $E$ and some extended limit, we simply say that $\tau$ is $\alpha$-quasi-invariant with respect to $E$.
\end{defn}

Note, that $\Phi_k$ are defined on page~\pageref{Phi_k} (formula~\eqref{Phi_k} and the paragraph after).

Observe that if $E$ is W-regular
then the limit in the definition of quasi-invariance exists,
and so is independent of the extended limit $\omega_0$.

\begin{rmk}
Just like in \cite[Lemma 4.2]{RRS}, if $E$ is full as a right module,
then any positive functional $\tau:A\to\C$ 
which is quasi-invariant in the sense of Definition \ref{ass:three} is a positive trace. To see this,
observe that for all $\mu,\,\nu\in \Fock^{\rm alg}$ and $a\in A$, 
the centrality of the Watatani indices $\e^{\beta_k}\in A$ (see formula~\eqref{dfn:Watatani} for the definition) gives
\begin{align*}
\e^{-\alpha|\mu|}\tau((\nu|\mu)_Aa)
&=\e^{-\alpha|\mu|}\tau((\nu|\mu a)_A)
=\lim_{k\to \omega_0}\tau({}_A(\mu a|\nu \e^{\beta_{k-|\nu|}})\e^{-\beta_k})\\
&=\lim_{k\to \omega_0}\tau({}_A(\mu|\nu a^*\e^{\beta_{k-|\nu|}})\e^{-\beta_k})
=\e^{-\alpha|\mu|}\tau((\nu a^*|\mu )_A)\\
&=\e^{-\alpha|\mu|}\tau(a(\nu|\mu )_A).
\end{align*}
\end{rmk}

\begin{example}
We consider the module $E_g$ over $C(Y)$ defined from a 
surjective local homeomorphism $g:Y\to Y$ as in Example 
\ref{localhomeoex} (see page \pageref{localhomeoex}). 
In this case, $\beta_n=0$ for all $n$ and 
quasi-invariance of a positive trace $\tau$ given by a positive measure $\lambda$ on $Y$ 
is equivalent to the condition 
$$(\mathfrak{L}_g)_*\lambda=\e^{\alpha} \lambda.$$
Another computation shows that this condition is equivalent to the Laca-Neshveyev condition. In particular, for the module $E_g$, quasi-invariance is equivalent to satisfying the Laca-Neshveyev condition.
\end{example}

\begin{example}
Let us consider the Cuntz algebra $O_N$ defined as the 
Cuntz-Pimsner algebra of the $\C$-bimodule $\C^N$. 
In this case, $(\nu|\mu)_A=\overline{{}_A(\mu|\nu)}$ for all $\mu,\nu \in (\C^N)^{\otimes n}$ and 
$\beta_n=n\log(N)$. Therefore, quasi-invariance of a trace $\tau$ on $\C$ is equivalent to 
$$\e^\alpha\tau=N\tau.$$
That is, any non-zero trace on $\C$ is $\log(N)$-quasi invariant.
\end{example}

Our immediate aim is to connect the quasi-invariance of Definition
\ref{ass:three} with the
condition imposed by Laca-Neshveyev, \cite{LN}.

\begin{lemma}
\label{lem:qi-LN}
Let $E$ be an fgp bi-Hilbertian bimodule. Suppose that $\tau$ satisfies the $\alpha$-quasi-invariance condition of Definition \ref{ass:three} with respect to $E$.
Then $\tau$ satisfies the Laca-Neshveyev condition for $\alpha$.
\end{lemma}

\begin{proof}
This is a computation using Definition \ref{ass:three} 
of quasi-invariance and
a frame $(e_j)$ for $E_A$. So
\begin{align*}
\Tr^E_\tau(L_a)&=\Tr_\tau\bigg(\sum_{j}a\Theta_{e_j,e_j}\bigg)
=\sum_j\tau((e_j|ae_j)_A)
=\e^\alpha\lim_{k\to \omega_0}\tau(\Phi_k(\sum_{j}T_{ae_j}T_{e_j}^*)\e^{-\beta_k})\\
&=\e^\alpha\lim_{k\to \omega_0}\tau(\sum_{j}a{}_A(e_j|e_j\e^{\beta_{k-1}})\e^{-\beta_k})
=\e^\alpha\lim_{k\to \omega_0}\tau(a\e^{\beta_{k}}\e^{-\beta_k})
=\e^\alpha\tau(a).\qedhere
\end{align*}
\end{proof}

If $\tau:A\to \C$ satisfies the $\alpha$-quasi-invariance condition of Definition \ref{ass:three}, then we can rewrite the Laca-Neshveyev KMS state
$\phi_{LN,\tau}:O_E\to\C$ as
\begin{align*}
\phi_{LN,\tau}(T_\mu T_\nu^*)
&=\delta_{|\mu|,|\nu|}\e^{-\alpha |\mu|}\tau((\mu|\nu)_A)=\lim_{k\to \omega_0}\tau(\Phi_k(T_\mu T_\nu^*)\e^{-\beta_k}).
\end{align*}
This computation proves the following. Note that the next result does not require any W-regularity from the module $E$.

\begin{prop}
\label{prop:LN-tau-Phi}
Let $E$ be an fgp bi-Hilbertian bimodule and consider an $\alpha$-quasi-invariant positive trace $\tau:A\to\C$ with respect to $E$. 
The state on $O_E$ defined by
\begin{equation}
\label{prestadefin}
S_\mu S_\nu^*\mapsto \lim_{k\to \omega_0}\tau(\Phi_k(T_\mu T_\nu^*)\e^{-\beta_k}),
\end{equation}
is $\alpha$-KMS for the gauge action on $O_E$ and coincides with $\phi_{LN,\tau}$. 
\end{prop}

If $E$ is W-regular, the definition $\Phi_\infty(S_\mu S_\nu^*):=\lim_{k\to \infty}\Phi_k(T_\mu T_\nu^*)\e^{-\beta_k}$ shows that the state in Equation \eqref{prestadefin} coincides with $\tau\circ \Phi_\infty$. We conclude the following.

\begin{corl}
\label{koroalala}
Let $\tau:A\to\C$ be an $\alpha$-quasi-invariant positive trace with respect to an fgp bi-Hilbertian bimodule $E$. Assume that $E$ is a W-regular. Then, the state $\tau\circ\Phi_\infty$ on $O_E$ is $\alpha$-KMS for the gauge action on $O_E$ and coincides with $\phi_{LN,\tau}$. 
\end{corl}

By Theorem \ref{prop:tau-phi=LN} and Corollary \ref{koroalala}, we have that $\phi_{\omega,\tau}=\tau\circ \Phi_\infty$ for any $\alpha$-quasi-invariant positive trace $\tau$ and extended limit $\omega$ at $\alpha$ assuming that $E$ is strictly regular.

\subsection{Obstructions to bi-Hilbertian bimodule structures}
\label{sub:no-left}

In the two previous subsections, we assumed our $A$-bimodule $E$ to be an fgp bi-Hilbertian bimodule, and imposed the additional assumption of strict W-regularity (see Definitions \ref{ass:one} 
and \ref{ass:two}). The assumptions allowed us to construct a semifinite spectral
triple from a Kasparov module relying on both the left and the right inner product on $E$. 
Instead, we can just use \cite{LN} to proceed from a KMS-state directly to a semifinite spectral triple whose associated KMS-state as in Corollary \ref{cor:phi-omega} coincides with the original KMS-state.
In order to compare the indirect approach for strictly W-regular modules to the direct approach from the KMS-state we will need to extend our module to von Neumann algebra coefficients, and along the way we derive obstructions to having the structure of a strictly W-regular bi-Hilbertian bimodule structure on an $A-A$-correspondence.

We suppose that we have a finitely generated projective right $A$-module $E_A$, 
with $A$ unital, and carrying a unital left action of $A$. 
Let $\tau:A\to\C$ be a faithful positive trace
satisfying the Laca-Neshveyev condition for $\alpha\geq 0$ (see 
Definition \ref{ass:3.5} on page \pageref{ass:3.5}),
and define the associated KMS-state on $O_E$ by
$$
\phi_{LN,\tau}(S_\mu S_\nu^*):=\delta_{|\mu|,|\nu|}\tau((\nu|\mu)_A)\e^{-\alpha|\mu|},\quad \mu,\,\nu\in \Fock^{\rm alg}.
$$
By construction, $\phi_{LN,\tau}|_A=\tau$.

The Cuntz-Pimsner algebra $O_E$ acts on the GNS-space $L^2(O_E,\phi_{LN,\tau})$ by left multiplication and, since $\phi_{LN,\tau}$ restricts to a trace on $A$, $A$ acts by both left and right multiplication on $L^2(O_E,\phi_{LN,\tau})$. 
For notational simplicity, we identify $\tau$ with its normal extension to $A''$. Note that $A''$ is independent of whether we take the bicommutant in $L^2(O_E,\phi_{LN,\tau})$ or $L^2(A,\tau)$ and by faithfulness of $\tau$ we can identify $A$ with its image under the GNS-representation and obtain an inclusion $A\subseteq A''$.

\begin{prop}
\label{adoubprim}
Let $E$ be an fgp right $A$-Hilbert $C^*$-module with a left unital action of $A$, $\tau$ a faithful positive trace on $A$ satisfying the Laca-Neshveyev condition and let $P_0:L^2(O_E,\phi_{LN,\tau})\to L^2(A,\tau)\subset L^2(O_E,\phi_{LN,\tau})$ denote the orthogonal projection. It then holds that 
$$A''=P_0O_E''P_0.$$
\end{prop}

\begin{proof}
It is clear that $A''\subseteq P_0O_E''P_0$. To prove the converse inclusion, take $T\in P_0O_E''P_0$ and write $T=P_0T_0P_0$ where $T_0$ is the WOT-limit of a net $T_\lambda=\sum_j S_{\mu_{\lambda,j}}S_{\nu_{\lambda,j}}^*\in O_E$. We have that 
$$P_0T_\lambda P_0=\sum_{j: |\mu_{\lambda,j}|=|\nu_{\lambda,j}|=0}S_{\mu_{\lambda,j}}S_{\nu_{\lambda,j}}^*,$$
so $P_0T_\lambda P_0\in A$ and $T\in A''$.
\end{proof}

We can define a conditional expectation
$$\tilde{\Phi}_\infty:O_E''\to A'', \quad\tilde{\Phi}_\infty(S):=P_0SP_0,$$
which is well defined by Proposition \ref{adoubprim}.
Using the expectation $\tilde{\Phi}_\infty$ 
we can define a right module $\Xi_{A''}$ by 
completing $O_E$ in the norm defined by the inner product
$$
(S_1|S_2)_{A''}:=\tilde{\Phi}_\infty(S_1^*S_2),\quad S_1,\,S_2\in O_E.
$$

It is clear that $L^2(O_E,\phi_{LN,\tau})=L^2(\Xi_{A''},\tau)$. 
The construction of $\Xi_{A''}$ does not require $E_A$ 
to be biHilbertian, just an $A$-$A$-correspondence. 
The following result follows from the relations defining the Cuntz-Pimnser algebra and the fact that $\tilde{\Phi}_\infty$ is a conditional expectation.

\begin{prop}
\label{phitildeinfocm}
For $\mu\in E^{\ox k}$ and $\nu\in E^{\ox l}$, 
$$\tilde{\Phi}_\infty(S_\mu^* S_\nu)=\delta_{|\mu|,|\nu|}( \mu|\nu)_A.$$
In particular, the map $\mu\mapsto S_\mu$ extends to an $A''$-linear isometric embedding $\Fock\otimes_A A''\to \Xi_{A''}$ of the Fock module. 
\end{prop}

We now turn to describing the WOT-closure of $E$ inside $O_E''$. We will identify $E$ with an $A$-sub-bimodule of $O_E$ via $\mu\mapsto S_\mu$.

\begin{lemma}
\label{lem:weak-biHilb}
Let $A$ be a unital $C^*$-algebra.
Let $E_A$ be a finitely generated projective right $A$-module with a unital
left action, and suppose that $\tau:A\to\C$ satisfies the Laca-Neshveyev
condition on $E$ for $\alpha\geq 0$. Then 
$$
E'':=\overline{E}^{\rm WOT}\subset O_E''
$$
is an $A''$-bimodule. Moreover, the following hold:
\begin{enumerate}
\item As right $A''$-modules, $E''\cong E\otimes_A A''$ and the isomorphism is an isomorphism of $A''$-Hilbert $C^*$-modules when equipping $E''$ with the right inner product
$$( \mu|\nu)_{A''}:=\tilde{\Phi}_\infty(S_\mu^* S_\nu), \quad\mu,\nu\in E''.$$
\item The right $A''$-Hilbert $C^*$-module $E''$ is finitely generated and projective.
\item If $E$ is finitely generated and projective as a left $A$-module, then $E''$ is finitely generated and projective as a left $A''$-module.
\item If the implication 
\begin{equation}
\label{lefass}
P_0ee^*P_0=0\Rightarrow e=0 \quad \forall e\in E'',
\end{equation}
holds, the expression
$$
{}_A(e|f)^{~}:=P_0S_eS_f^*P_0
$$
gives a left inner product on $E''$ making it into a bi-Hilbertian bimodule. The right Watatani index of $E''$ is $1$ and $\topop_k={\rm Id}_{(E'')^{\ox k}}$ for all $k$ (and so is invertible).
\item If $E''$ is a finitely generated projective module from the left and the implication \eqref{lefass} holds, then $E''$ is a strictly $W$-regular fgp bi-Hilbertian bimodule over $A''$. 
\end{enumerate}
\end{lemma}

\begin{proof}
The initial statement of the lemma is clear since $E$ is an $A$-sub-bimodule of $O_E$, so its WOT closure is a bimodule over $A''$. To prove statement 1. we note that $E\otimes_A A''\cong EA''\subseteq O_E''$. Moreover, using that $E$ is finitely generated and projective, it follows that $E''=EA''$ and therefore $E''\cong E\otimes_A A''$ follows. Statement 1. now follows from Proposition \ref{phitildeinfocm}. Statement 2. follows from statement 1. because $E$ is finitely generated and projective over $A$. Statement 3. is proven in a similar way as Statement 1., indeed if $E$ is finitely generated and projective as a left $A$-module then $E''\cong A''\otimes_A E$ as a left $A$-module. 

Statement 4. is less trivial. Assuming that the implication \eqref{lefass} holds, it is straight-forward to verify that the left and right actions are compatible, i.e. that the $A$-action from the left/right is adjointable for the right/left inner product. For $E''$ to be a bi-Hilbertian bimodule it remains to show that the norm arising from the left inner product is equivalent to the norm arising from the right inner product. For any $e\in E$, we compute that 
\begin{align}
\nonumber
\Vert{}_A(e|e)\Vert_A=\Vert P_0S_{e}S_{e}^*P_0\Vert_{L^2(A,\tau)}
&=\Vert S_{e}^*P_0S_{e}\Vert_{L^2(E^*,\phi_\tau)}
=\Vert S_{e}^*S_{e}\Vert_{L^2(E^*,\phi_\tau)}=\\
\label{relatingthetwoinnprod}
&=\Vert(e|e)_A\Vert_{L^2(E^*,\phi_\tau)}=\Vert(e|e)_A\Vert_A
\end{align}
where the norm of $S_{e}^*S_{e}$
is attained by $S_{e}^*P_0S_{e}$ on $L^2(E^*,\phi_\tau)$ by the 
assumption that ${}_A(\cdot|\cdot)$ is positive definite.
Equation \eqref{relatingthetwoinnprod} shows that the two norms $\sqrt{\Vert(\cdot |\cdot)_A\Vert_A}$ and $\sqrt{\Vert{}_A(\cdot |\cdot)\Vert_A}$ on $E$ are equivalent.

To finalize the proof of statement 4., we compute the right Watatani index of $E$, which exists because $E''$ is finitely generated projective by statement 2. We compute on 
$L^2(A,\tau)\subset L^2(O_E,\phi_\tau)$ that
$$
\langle a, \sum_j{}_A(e_j|e_j)a\rangle
=\langle a,\sum_jP_0\pi^{(1)}(\Theta_{e_j,e_j})P_0a\rangle=
\phi_\tau(a^*P_0\pi^{(1)}({\rm Id_E})P_0a)=\phi_\tau(a^*1_Aa)=\tau(a^*a).
$$
By the faithfulness of $\tau$ and Cuntz-Pimsner covariance 
we can now deduce that the right
Watatani index is equal to $1_A$. This immediately implies that
$\topop_k={\rm Id}_{(E'')^{\ox k}}$ for all $k$.

Finally, statement 5. follows from that under the stated assumptions, $E''$ is an fgp bi-Hilbertian bimodule (using statements 1., 2. and 4.) and by statement 4. $\topop_k$ satisfies the condition in Definition \ref{ass:two}, so $E''$ is strictly W-regular.
\end{proof}

\begin{rmk}
\label{leftdoubprime}
There are examples of $A$-$A$-correspondences $E$ that are fgp from the right but not fgp from the left such that $E''$ is fgp from the left. These examples come from (certain) self-similar dynamical systems, see \cite{KajWat}.

Here is a simple example. Let $A=C([0,1])$, 
$$
\gamma_1:[0,1]\to[0,1]\quad\gamma_1(x)=x/2,\qquad
\gamma_2:[0,1]\to[0,1]\quad\gamma_2(x)=1/2+x/2,
$$
and $E=C(\{(\gamma_1(x),x):x\in[0,1]\}\cup \{(\gamma_2(x),x):x\in[0,1]\})$.
The correspondence structure is defined for $a,\,b\in A$ and $e\in E$ by
$$
(a\cdot e\cdot b)(\gamma_j(x),x)=a(\gamma_j(x))e(\gamma_j(x),x)b(x)
\quad\mbox{and}\quad (e_1|e_2)_A(x)=\sum_{j=1,2}\overline{e_1(\gamma_j(x),x)}e(\gamma_j(x),x).
$$

The graphs of $\gamma_1$ and $\gamma_2$ in $[0,1]\times [0,1]$ are disjoint and their respective characteristic function $\chi_1$ and $\chi_2$ are elements of $E$. One checks directly that $\{\chi_1,\chi_2\}$ is a right frame for $E_A$, and since $(\chi_1|\chi_2)=0$, there is an isomorphism of right Hilbert $C^*$-modules $E_A\cong C[0,1]\oplus C[0,1]$. We conclude that $E_A$ is fgp from the right. Using the frame $\{\chi_1,\chi_2\}$, we can identify ${}_AE\cong C[0,1/2]\oplus C[1/2,1]$ as a left $C[0,1]$-module. As a left module, ${}_AE$ is therefore finitely generated but clearly not projective as the rank of $E_x:=E/C_0([0,1]\setminus \{x\})E$ is discontinuous at $x=1/2$.

We shall now see that it is even impossible for a left inner product compatible with the right inner product to exist. Let $0\leq \phi\in A$ be $1$ on $[0,1/2]$. Then if we have a compatible left $A$-valued inner product
$$
{}_A(\chi_1|\chi_2)(x)={}_A(\phi\cdot\chi_1|\chi_2)(x)=\phi(x){}_A(\chi_1|\chi_2)(x)
={}_A(\chi_1|\chi_2)(x)\phi(x)={}_A(\chi_1|\phi\cdot\chi_2)(x).
$$
Taking the infimum over such $\phi$ we see that the support of ${}_A(\chi_1|\chi_2)$ is contained in $\{1/2\}$. Then for arbitrary $a,\,b\in A$
\begin{align*}
&{}_A(a\chi_1+b\chi_2|a\chi_1+b\chi_2)\\
&=a{}_A(\chi_1|\chi_1)a^*
+b{}_A(\chi_2|\chi_2)b^*+a(1/2)b^*(1/2){}_A(\chi_1|\chi_2)+b(1/2)a^*(1/2){}_A(\chi_2|\chi_1).
\end{align*}
From here one can show that any inner product taking values in the continuous functions takes values in the functions vanishing at $1/2$. Then one shows that the associated norm can not
be equivalent to the right inner product.

The situation is better for $E''$. We consider the trace $\tau(a):=\int_0^1 a(x)\mathrm{d}x$ on $C[0,1]$. A short computation shows that $\Tr_\tau^E=2\tau$ so $\tau$ satisfies the Laca-Neshveyev condition for $\alpha=\log(2)$ and extends to a KMS-state on $O_E$ at inverse temperature $\log(2)$. It is readily verified that $C[0,1]''=L^\infty[0,1]$, $E''\cong L^\infty[0,1]\oplus L^\infty[0,1]$ as a right module and $E''=L^\infty[0,1/2]\oplus L^\infty[1/2,1]$ as a left module. In particular, $E''$ is an fgp bi-Hilbertian bimodule over $L^\infty[0,1]$. 

The same discussion applies to any `graph separated' iterated function system satisfying the open set condition. See \cite{KajWat} for more details.
\end{rmk}

\begin{thm}
\label{prop:left-noleft}
Let $E_A$ be a right $A$-Hilbert $C^*$-module with a unital left action and assume that $E$ is 
finitely generated both as a left and a right $A$-module and $E''$ is finitely generated and projective both as a left and a right $A''$-module. Then $E_A$ has a left inner 
product such that $E$ is an fgp bi-Hilbertian bimodule,
has finite right Watatani index and is $W$-regular with 
$\topop_k$ invertible for all $k$
if and only if $\tilde{\Phi}_\infty$ is faithful and $\tilde{\Phi}_\infty(O_E)\subseteq A$. 
\end{thm}

We remark that if $E$ is W-regular and $\topop_k$ is invertible for all $k$, then $E$ is strictly W-regular by \cite[Lemma 3.8]{GMR}.

\begin{proof}
First suppose that $\tilde{\Phi}_\infty$ is faithful and $\tilde{\Phi}_\infty(O_E)\subseteq A$. 
Lemma \ref{lem:weak-biHilb} shows that $E''$ has a 
left inner product making $E$ (not just $E''$) bi-Hilbertian
with right Watatani index $1_{A}$ and (invertible) 
$\topop_k={\rm Id}_{(E)^{\ox k}}$. Since $E$, both as a left and a right module, 
is finitely generated and admits a Hilbert $C^*$-module structure it is also projective.

Conversely, if $E$ is an fgp bi-Hilbertian bimodule with finite right 
Watatani index and is W-regular with invertible $\topop_k$, 
then \cite{RRS} proves that 
the map $\Phi_\infty(S_\mu S_\nu^*)
={}_A(\mu|\topop_{|\nu|}(\nu))$ is a (faithful)
conditional expectation. 
That it agrees with $\tilde{\Phi}_\infty$ is a computation.
\end{proof}

\begin{rmk} 
The issue with $\topop_k$ being non-invertible is as follows.
Since $\topop_k=\topop_1\ox\topop_1\ox\cdots\ox\topop_1$,
the non-invertibility occurs with $\topop_1$. Supposing
$O_E$ to be strictly W-regular, we can define the right module
$\Xi_A$ as the completion of $O_E$ for the norm coming from $\Phi_\infty$.
Then $\tilde{\Phi}_\infty(S_eS_f^*)={}_A(e|\topop_1(f))$, and so if
$\topop_1$ is not invertible, we do not get a left inner product in this way.
See \cite[Example 3.10]{RRS} for an example where we have strict W-regularity with $\topop_1$ not being invertible.

Heuristically, one should view the passage to $\Xi_A$ as erasing the information about the
left inner product on $E$, corresponding to the kernel of $\topop_1$. On the other hand, if $\topop_k$ is invertible for all $k$ then we can replace
our left inner products ${}_A(\cdot|\cdot)^{E^{\ox k}}$ by
${}_A(\cdot|\topop_k(\cdot))^{E^{\ox k}}$ and obtain an equivalent
inner product structure with right Watatani index 1.
\end{rmk}

We can now relate our constructions above back to KMS-states. 

\begin{prop}
Let $A$ be a unital $C^*$-algebra, $E$ a finitely generated projective right $A$-Hilbert $C^*$-module with a unital adjointable left $A$-action, $\alpha\geq 0$ and $\tau$ a positive trace on $A$. We assume that this data satisfies the following conditions:
\begin{itemize}
\item $\tau$ satisfies the Laca-Neshveyev condition. 
\item The $A''$-bimodule $E''$ is finitely generated and projective from the left.
\item The implication \eqref{lefass} (see page \pageref{lefass}) holds.
\end{itemize}
We define a semifinite spectral triple $(O_E, L^2(O_E,\phi_{LN,\tau}),\D_\psi, \cN, \Tau)$ using that $L^2(O_E,\phi_{LN,\tau})=L^2(O_{E''},\phi_{LN,\tau})=L^2(\Xi_{A''},\tau)$ and pulling back the semi-finite spectral triple defined from the fgp bi-Hilbertian $A''$-bimodule $E''$ as in Lemma \ref{ximodsemi} along the inclusion $O_E\to O_{E''}$. This semi-finite spectral triple is $\mathrm{Li}_1$-summable, $\alpha$-analytic, has positive essential $\Tau$-spectrum with $\beta_{\D_\psi}=\alpha$ and its associated KMS-state for the gauge action (as in Corollary \ref{cor:phi-omega} on page \pageref{cor:phi-omega}) coincides with $\phi_{LN,\tau}$.
\end{prop}

\end{document}